\documentclass[a4paper,11pt,leqno]{amsart}
\usepackage{amssymb, amsmath, amsthm, graphicx, color, epsfig, amsfonts}
\usepackage[foot]{amsaddr}

\newtheorem{thm}{Theorem}[section]
\newtheorem{lem}[thm]{Lemma}

\newtheorem{prop}[thm]{Proposition}

\newtheorem{rem}{Remark}[section]
\newtheorem{exa}{Example}[section]
\numberwithin{equation}{section}

\renewcommand{\a}{\alpha}
\renewcommand{\b}{\beta}
\newcommand{\e}{\varepsilon}
\newcommand{\de}{\delta}
\newcommand{\fa}{\varphi}
\newcommand{\ga}{\gamma}
\renewcommand{\k}{\kappa}
\newcommand{\la}{\lambda}

\newcommand{\si}{\sigma}

\renewcommand{\t}{\tau}

\newcommand{\De}{\Delta}
\newcommand{\Ga}{\Gamma}

\newcommand{\lan}{\langle}
\newcommand{\ran}{\rangle}

\def\R{{\mathbb{R}}}
\def\N{{\mathbb{N}}}
\def\Z{{\mathbb{Z}}}
\def\T{{\mathbb{T}}}

\def\E{{\mathbb E}}
\renewcommand{\P}{{\mathbb P}}

\allowdisplaybreaks

\newcommand{\ue}{u^\e}

\newcommand{\di}{\displaystyle}
\newcommand{\Pe}{(P^{\;\!\e})}
\newcommand{\Pz}{(P^{\;\!0})}

\newcommand{\laconstante}{\lambda _0}

\newcommand{\q}{q}
\newcommand{\G}{G}

\title
[Mean curvature interface limit from interacting particles]
{Mean curvature interface limit \\  from Glauber+Zero-range interacting particles}

\author{Perla El Kettani$^+$}
\address{$^+$Aix Marseille University, Toulon Univeristy \\
Laboratory Centre de Physique Th\'eorique, CNRS,  Marseille, France.}
\email{Kettaneh.perla@hotmail.com}

\author{Tadahisa Funaki$^*$} 
\address{$^*$Department of Mathematics\\
School of Fundamental Science and Engineering\\
Waseda University\\
3-4-1 Okubo, Shinjuku-ku\\
Tokyo 169-8555, Japan}
\email{{\tt funaki@ms.u-tokyo.ac.jp}}

\author{Danielle Hilhorst$^\%$}
\address{$^\%$CNRS and Laboratoire de Math\'ematiques \\ 
University Paris-Saclay\\
Orsay Cedex 91405,  France}
\email{Danielle.Hilhorst@math.u-psud.fr}

\author{Hyunjoon Park$^\dagger$}
\address{$^\dagger$Department of Mathematical Sciences\\ 
Korean Advanced Institute of Science and Technology\\
 291 Daehak-ro, Yuseong-gu, Daejeon 34141, Korea}
\email{hyunjoonps@gmail.com}

\author{Sunder Sethuraman$^\diamond$}
\address{$^\diamond$Department of Mathematics\\
University of Arizona\\
621 N. Santa Rita Ave.\\
Tucson, AZ 85750, USA}
\email{{\tt sethuram@math.arizona.edu}}

\date{\today}

\begin{document}

\begin{abstract}
\noindent
We derive a continuum mean-curvature flow as a certain hydrodynamic scaling limit of a class of Glauber+Zero-range particle systems.  The Zero-range part moves particles while preserving particle numbers, and the Glauber part governs the creation and annihilation of particles and is set to favor two levels of particle density.  When the two parts are simultaneously seen in certain different time-scales, the Zero-range part being diffusively scaled while the Glauber part is speeded up at a lesser rate, a mean-curvature interface flow emerges, with a homogenized `surface tension-mobility' parameter reflecting microscopic rates, between the two levels of particle density.  We use relative entropy methods, along with a suitable `Boltzmann-Gibbs' principle, to show that the random microscopic system may be approximated by a `discretized' Allen-Cahn PDE with nonlinear diffusion.  In turn, we show the behavior, especially
generation and propagation of interface properties, of this `discretized' PDE.
\end{abstract}

\maketitle



\tableofcontents

\section{Introduction}
We study the emergence of continuum mean curvature interface flow from a class of microscopic interacting particle systems.  
Such a concern in the context of phase separating interface evolution is a long standing one in statistical physics; see Spohn \cite{S93} for a discussion.  
The aim of
this paper is to understand the formation of a continuum mean curvature interface flow, with a homogenized `surface tension-mobility' parameter reflecting microscopic rates, as a scaling limit in a general class of reaction-diffusion interacting
particle systems.  We focus on so-called Glauber+Zero-range processes on discrete tori $\T^d_N= (\Z/N\Z)^d$ for dimensions $d\geq 2$ and scaling parameter $N$, where the Glauber part governs reaction rates favoring two levels of mass density, and the Zero-range part controls nonlinear rates of exploration.

A `two step' approach to derive the continuum interface flow would consider scaling the Zero-range part of the dynamics, but not speeding up the Glauber rates.  The first step would be to obtain the space-time mass hydrodynamic limit in terms of an Allen-Cahn reaction-diffusion PDE.  The second step would be to scale the reaction term in this Allen-Cahn PDE and to obtain mean-curvature interface flow in this limit. 

However, in a nutshell, our purpose is to obtain `directly' the mean curvature interface flow, up to the time of singularity, by scaling {\it both} the Glauber and Zero-range parts simultaneously.  The Zero-range part is diffusively scaled while the Glauber part is scaled at a lesser level.  By means of a probabilistic relative entropy method, and a new `Boltzmann-Gibbs' principle, we show that the microscopic system may be approximated by a `discretized' Allen-Cahn equation whose reaction term is being speeded up;
see \eqref{eq:1.DHDeq}.

\subsection{Motion by mean curvature and Allen-Cahn
equation with linear diffusion}

In the continuum, motion by mean curvature is a time evolution of $(d-1)$-dimensional 
hypersurface $\Ga_t$  in $\T^d:= (\R/\Z)^d=[0,1)^d$ with periodic boundary conditions, or in $\R^d$ defined by
\begin{equation}  \label{eq:1.MMC-0}
V=\k,
\end{equation}
where $V$ is a normal velocity and $\k$ is the mean curvature of
$\Ga_t$ multiplied by $d-1$.  Such a flow is of course a well-studied geometric object (cf. book Bellettini \cite{Be}).

Mean curvature flow is known to arise from Allen-Cahn equations with linear diffusion, 
which are reaction-diffusion equations of the form
\begin{equation}  \label{eq:1.1-A}
{\partial_t u} = \De u + \frac1{\e^2} f(u),
\quad t>0,\; x\in D,
\end{equation}
in terms of a `sharp interface limit' as $\e\downarrow 0$.  Here, $D=\T^d$ or
a domain in $\R^d$, for $d\geq 2$, with Neumann boundary conditions at $\partial D$,
$\e>0$ is a small parameter and $f$ is a bistable function with stable points $\alpha_\pm$
and unstable point $\a_* \in (\a_-,\a_+)$ satisfying the balance
condition:
$$\int_{\alpha_-}^{\alpha_+} f(u) \, du \; \big(= F(\alpha_-)-F(\alpha_+)\big)=0,$$
where $F$ is the potential associated with $f$  such that $f=-F'$.
The sharp interface limit is as follows:  The solution $u=u^\e$ of the Allen-Cahn equation satisfies
\begin{equation*}
u^\e(t, x) \underset{\e \downarrow  0}{\longrightarrow} \chi_{\Ga_t}(x) 
:= \left\{ \begin{aligned}
\a_+, & \quad \text{on one side of } \Ga_t, \\
\a_-, & \quad \text{on the other side of } \Ga_t, 
\end{aligned} \right.
\end{equation*}
where $\Ga_t$ moves according to the motion by mean curvature \eqref{eq:1.MMC-0}, 
and the sides are determined from $\Gamma_0$.  This limit has a long history; 
among other works, see 
Alfaro et al.\ \cite{AGHMS}, Bellettini \cite{Be},
Chen et al.\ \cite{CHL}, Funaki \cite{F99}, Chapter 4 of Funaki \cite{F16} and 
references therein.  Although we do not consider the case $d=1$, we remark the phenomenon in dimension $d=1$ is much different given that the `interface' consists of points; see Carr et al. \cite{Carr}.

\subsection{Glauber+Zero-range process, its scaling limits and main result}

Informally, the Zero-range process follows a collection of continuous time random walks on $\T^d_N$ such that each particle interacts infinitesimally only with the particles at its location: At a site $x$, one of the particles there jumps with rate given by a function of the occupation number $\eta_x$ at $x$, say $g(\eta_x)$, and then displaces by $y$ with rate $p(y)$.  We will consider the case that jumps occur only to neighboring sites with equal rate, that is $p(y) = 1(|y|=1)$.
 It is known that, under the diffusive
scaling in space and time, namely when space squeezed by $N$ while time speeded
up by $N^2$, in the limit as $N\to\infty$, 
the evolution of the macroscopic mass density profile of the microscopic particles, namely the `hydrodynamics', follows a nonlinear PDE (cf. \cite{KL})
$$\partial_t u = \Delta \varphi(u),$$
where $\varphi$ can be seen as a homogenization of the microscopic rate $g$.  We remark when $g(k)\equiv k$, and so $\varphi(u)\equiv u$, the associated Zero-range process is the system of independent particles.

We may add the effect of Glauber dynamics to the Zero-range process.  Namely,
we allow now creation and annihilation of particles at a location with rates which depend on occupation numbers nearby.  This mechanism
is also speeded up by a factor $K=K(N)\nearrow \infty$ as $N\to\infty$.  We will impose that $K$ 
grows much slower than the time scale $N^2$ for the Zero-range part, in fact we will take that $K = O((\log N)^{\si/2})$ in our main Theorem \ref{mainthm}; see
below for some discussion.

If $K$ were kept constant 
with respect to $N$, the associated hydrodynamic mass density solves a nonlinear 
reaction-diffusion equation, 
a type of Allen-Cahn equation with nonlinear diffusion, in the diffusive scaling limit:
\begin{equation}  \label{eq:1.AC-c}
\partial_t u = \Delta \varphi(u) + Kf(u)
\end{equation}
where $f$ reflects a homogenization of the Glauber creation and annihilation rates;
cf. \cite{Mourragui}, see also \cite{DFL} and \cite{DP} in which related 
Glauber+Kawasaki dynamics was studied.

As mentioned above, with notation $1/\varepsilon^2$ instead of $K$, in the PDE literature, taking the limit of solutions $u=u^{(K)}$, as $K\uparrow\infty$, in these Allen-Cahn equations, when say $\varphi(u)\equiv u$ and $f$ is bistable, that is $f(u)=-F'(u)$ with $F$ being a `balanced' double-well potential, 
is called the sharp interface limit.  
This scaling limit leads to a continuum motion by mean
curvature of an interface separating two phases, here say two levels of mass density.

In our stochastic setting, by properly choosing the rates of creation and annihilation of particles
in Glauber part, we observe, in the microscopic system itself, the whole domain $\T^d_N$ separates in a short time into 
`denser' and `sparser' regions of particles with an interface region of width $O(K^{-1/2})$ between (cf. Theorems \ref{thm:6.3} and \ref{thm:3.4}).
In particular, our paper derives as a main result, as $N\uparrow\infty$,
motion of a continuum interface by mean curvature directly from these microscopic particle systems as a
combination of the ideas of the hydrodynamic limit (probabilistic part) and 
the sharp interface limit (PDE part); cf.\ Theorem \ref{mainthm}.

\subsection{Probabilistic vs PDE arguments}

In the probabilistic part (Sections \ref{sec:4} and \ref{BG_section}),
for the hydrodynamic limit, we apply the so-called relative
entropy method originally due to Yau \cite{Y}.  As a consequence of the method, we show that the microscopic configurations are not far from the solution to a deterministic discrete approximation to the nonlinear Allen-Cahn equation (cf.\ Theorem \ref{thm:EstHent}); see equation \eqref{eq:1.DHDeq}.  To control the errors in this approximation, we will need a new `quantified' replacement estimate, which can be seen as a type of `Boltzmann-Gibbs' principle (cf.\ Theorem \ref{BG}).  
$L^\infty$-bounds on second discrete derivatives of the solution of discretized
Allen-Cahn equation \eqref{eq:1.DHDeq} (cf.\ Theorem \ref{lem:3.3-a}), derived
by Nash and Schauder estimates in \cite{FS}, play important role.

In the continuum/discrete PDE part (Sections \ref{sec:cont-AC} and
 \ref{sec:6.3}, respectively),
we compare the discretized Allen-Cahn equation \eqref{eq:1.DHDeq} with its continuous counterpart \eqref{eq:1.AC-c} with nonlinear diffusion and, by comparison argument, construct super and sub solutions in terms of those for the continuum PDE;
see Theorem \ref{PDEthm} for the main result of the PDE part.
We note that a sharp interface limit, with respect to the Allen-Cahn equation, now with nonlinear diffusion term $\Delta\varphi(u)$ is shown in a companion paper
\cite{EFHPS}, and summarized in Theorems \ref{Thm_Generation} and \ref{Thm_Propagation}.  Such a derivation is obtained by keeping a `corrector' term in the expansion, or second order term in $\varepsilon=K^{-1/2}$, of the solutions $u=u^{(K)}$ in variables depending on the distance to a certain level set; 
see Section \ref{sec:cont-AC}.  It seems
this sharp interface limit for the nonlinear Allen-Cahn equation is unknown
even in the continuum setting.

\subsection{Comparison to previous works and differences}

Previous work on such problems in particle systems with creation and annihilation rates concentrates on Glauber +Kawasaki dynamics 
(where the Zero-range part is replaced by Kawasaki dynamics) 
\cite{Bonaventura}, \cite{DPPV94}, \cite{KS94}, \cite{G95}, and \cite{FT}.
In these papers, the Kawasaki part is a simple exclusion process.  For $K$ fixed with respect to $N$,
the macroscopic mass hydrodynamic equation is a more standard Allen-Cahn PDE 
\eqref{eq:1.1-A} with linear diffusion $\Delta u$ (instead of $\Delta \varphi(u)$)
and $K$ instead of $\e^{-2}$,
$$\partial_t u = \Delta u + Kf(u).$$
See also related work on Glauber dynamics with Kac type long range mean field interaction \cite{BBP}, \cite{DOPT93}, \cite{DOPT94}, \cite{KS95}, on fast-reaction limit for two-component Kawasaki dynamics \cite{DFPV}, and on spatial coalescent models of population genetics \cite{EFP}.

Phenomenologically, when there is a nonlinear Laplacian, say $\Delta \varphi(u)$, as in our case of the Glauber +Zero-range process,
this nonlinearity affects the limit motion of the hypersurface interface.   When now $f$ satisfies a modified balance condition due to the nonlinearity (cf. condition (BS)),
 we obtain in the limit
a mean curvature motion speeded up by a nontrivial in general `surface tension-mobility' speed $\lambda_0$ reflecting a homogenization of the Glauber and Zero-range microscopic rates, 
\begin{equation}  \label{eq:1.MMC}
V=\lambda_0\kappa
\end{equation} 
(cf. flow $(P^0)$ \eqref{P^0 problem}).  We derive two formulas for $\lambda_0$, one of them below
in \eqref{def-lambdazero-intrinsic-intro},
 and the other found in \eqref{eqn_lambda0}, from which 
$\lambda_0$ is seen as the `surface tension' multiplied by the `mobility' of the interface;
see Appendix of \cite{EFHPS}. 
We remark, in the case of Glauber+Kawasaki dynamics, or for independent particles, the speed $\lambda_0 = 1$ is not affected by the microscopic rates.

The discretized hydrodynamic equation, or discretized Allen-Cahn PDE,
\begin{align}  \label{eq:1.DHDeq}
\partial_t u^N = \Delta^N \varphi(u^N) + Kf(u^N),
\end{align}
with discrete Laplacian $\Delta^N$, plays a role to cancel 
the first order terms in the occupation numbers in the computation of 
the time derivative of the relative entropy of the law of the microscopic configuration at time $t$ with respect to a local equilibrium measure with average profile given by $u^N$.
But, in the present situation, the problem is more complex than say in the application to
Glauber+Kawasaki dynamics since we need to handle
nonlinear functions of occupation numbers, which do not appear in the Glauber+Kawasaki process, by replacing them
by linear ones.  Once this is done, in a quantified way, the relative entropy can be suitably estimated, yielding that the microscopic configuration on $\T^d_N$ is `near' the values $u^N$.

The replacement scheme, a type of `quantified' second-order estimate or `Boltzmann-Gibbs principle, takes on here an important role.  This estimate, in comparison with a related bound for Kawasaki+Glauber systems in \cite{FT}, seems to hold in more generality, and its proof is quite different.  In particular, the technique used in \cite{FT} does not seem to apply for Glauber+Zero-range processes, relying on the structure of the Kawasaki generator.  
Moreover, as a byproduct of the `quantified' second order estimate here, the form of the discretized hydrodynamic
equation found turns out to satisfy a comparison theorem without any additional assumptions,
such as the assumption (A3) for the creation and annihilation rates in \cite{FT}.  
This is another advantage of our Boltzmann-Gibbs principle, beyond its more general validity (cf. Remark \ref{rem2.0-a}).
We remark, in passing, different `quantitative' replacement estimates, in other settings, have been recently considered \cite{JM2}, \cite{Otto}. See also in this context the non-quantitative estimates in \cite{JLS1}, \cite{JLS2}, \cite{Jara_Valentim}.

\subsection{Outline of the paper}

The outline of the paper is as follows:  In Section \ref{sec:2}, we introduce Glauber+Zero-range process in detail.  In particular,
we describe a class of invariant measures $\nu_\rho$ (cf. \eqref{eq:3.2}), and a spectral gap assumption (SP)
for the Zero-range part, and then specify a proper choice of the
creation and annihilation rates for the Glauber part, favoring two levels of mass density (cf.  \eqref{eq:2c+} and \eqref{eq:2c-}), so that the corresponding macroscopic reaction function $f$
satisfies a form of balanced bistability, matched to the nonlinear diffusion 
term $\Delta \varphi(u)$ obtained from the Zero-range part (cf. condition (BS)).

Our main result on the direct passage from the microscopic system to the continuum interface flow is formulated in Theorem \ref{mainthm}.  
Its proof, given in Section \ref{sec:2.3}, relies on two theorems:
Theorem \ref{thm:EstHent}, which is probabilistic, stating that the microscopic system is close to that of a discretized reaction-diffusion equation, and 
Theorem \ref{PDEthm}, which is PDE related, stating that the discrete PDE evolution is close to the continuum interface flow.
Theorem \ref{thm:EstHent} follows as a combination of the relative entropy method
developed in Section \ref{sec:4} and a Boltzmann-Gibbs principle 
stated in Section \ref{BG_statement} and proved in Section \ref{BG_section}.  
On the other hand, Theorem \ref{PDEthm}
is shown via PDE arguments for the sharp interface limit in terms of `generation' and `propagation' of the interface phenomena, in Section \ref{sec:6.3}.

In Section \ref{sec:3}, we develop, in addition to stating the Boltzmann-Gibbs principle, some preliminary results for the discrete PDE, namely 
a comparison theorem, a priori energy estimates, and $L^\infty$-bounds on discrete derivatives due to Nash and Schauder estimates shown in \cite{FS}.

In Section \ref{sec:4}, we prove Theorem \ref{thm:EstHent}, by implementing the method of relative entropy:  We compute the time derivative of the relative entropy of our
dynamics $\mu_t^N$ at time $t$ with respect to the local equilibrium state 
$\nu_t^N$ constructed from the solution of the discretized hydrodynamic 
equation \eqref{eq:HD-discre}.  As remarked earlier, in the case of Kawasaki 
dynamics instead of the Zero-range process, the first order terms appearing 
in these computations are all written already
in occupation numbers $\eta_x$ or its normalized variables; see \cite{FT}.
In our case, in contrast, nonlinear functions of $\eta_x$ appear, that is, the jump
rate $g(\eta_x)$ of the Zero-range part, as well as reaction rates $c_x^\pm(\eta)$ of the Glauber part.  We mention, in \cite{FT}, 
the relative entropy method of Jara and Menezes \cite{JM2}, 
a variant of \cite{Y}, was applied.  This method does not seem to apply for Glauber+Zero-range processes.  However,
because of our Boltzmann-Gibbs principle, the original method of
Yau \cite{Y} turns out to be enough.

The Boltzmann-Gibbs principle with a quantified error is essential in our work to replace nonlinear functions of $\eta_x$, for instance $g(\eta_x)$ and those arising from 
the Glauber part, by linearizations in terms of the occupation numbers $\eta_x$.  Its proof
is given in Section \ref{BG_section}.  The argument makes use of time averaging and mixing properties of the Zero-range process in the form of a spectral gap condition (SP), verified for a wide variety of rates $g$.  Nonlinear functions, such as $g(\eta_x)$, are estimated by their conditional expectation given local average densities $\eta^\ell_x= \ell^{-d}\sum_{y: |y-x|\leq \ell} \eta_y$.  In the standard `one-block' estimate of Guo-Papanicolaou-Varadhan (cf. \cite{KL}), which gives errors of order $o(1)$ without quantification, $\ell$ is of the order $N$, and so $\eta^\ell_x$ is close to the local macroscopic density.  Here, errors multiplied by diverging functions of $K$ need to be controlled, because of the form of certain terms in the discrete hydrodynamic equation.  The idea then is to consider $\ell = N^\alpha$ where $\alpha>0$ is small, and so $\eta^\ell_x$ is a type of `mesoscopic' average.  The spectral gap condition (SP) is also an ingredient used to quantify the errors suitably.  

The growth of $K$ of order $O((\log N)^{\si/2})$ that we impose is due to the Schauder estimate
\cite{FS} for the discrete hydrodynamic equation that we formulate in Section \ref{energy_subsec}.  In the case of the Glauber+Kawasaki model, a growth order of $O(\sqrt{\log N})$ was obtained in \cite{FT}, afforded by the linear diffusion term in its discrete hydrodynamic equation, as opposed to the nonlinear one $\Delta^N \varphi(u^N)$ which seems not as well behaved.  We remark that, in the work of \cite{Bonaventura} and \cite{KS94}, for Glauber+Kawasaki processes, $K$ can be of order $O(N^\beta)$ for a small $\beta>0$, the difference being that the method of correlation functions was used instead of relative entropy.  This method, relying on the structure of the Kawasaki model, does not seem to generalize to the systems considered here.

Finally we explain the PDE part.
In Section \ref{Formal derivation section} we discuss informally our derivation of the sharp interface limit from Allen-Cahn PDE with nonlinear diffusion.  
To study the limit as $K\uparrow\infty$, it is essential to consider the
asymptotic expansion of the solution up to the second order term in $K$.
This plays a role of the corrector in the homogenization theory and,
by the averaging effect for the nonlinear diffusion operator, a
constant speed $\la_0$ arises in the motion by mean curvature,
\begin{equation}\label{def-lambdazero-intrinsic-intro}
\lambda_0 =\frac{\di \int _{\alpha_-} ^{\alpha_+} \fa'(u)\sqrt{W(u)}du}{\di \int
_{\alpha_-} ^{\alpha_+} \sqrt{W(u)} du}\,
\end{equation}
and the potential $W$ is defined by
\begin{equation}\label{potential}
W(u)= \int _ u ^{\alpha_+} f(s) \fa'(s) ds\,.
\end{equation}
We refer also to \eqref{eqn_lambda0} for the other formula for $\lambda_0$ in terms of surface tension and mobility of the interface.

Section \ref{main results on pde}  summarizes results obtained in \cite{EFHPS}
on the `generation' of interface, or `initial layer' property (cf.\ Theorem \ref{Thm_Generation}) and the `propagation' of interface,
or motion by mean curvature with a homogenized `surface tension-mobility' speed, 
for the continuum Allen-Cahn
equation with nonlinear diffusion (cf.\ Theorem \ref{Thm_Propagation}).

Sections \ref{generation_sec} and  \ref{propagation_sec} give outline of the proof of
these two theorems, especially, recording estimates (cf.\ Lemmas
\ref{Lem_generation_with_homo_Neumann} and \ref{Lem_Prop_subsuper})
useful to apply for the discrete PDE \eqref{eq:1.DHDeq}.

In Section \ref{sec:6.3}, we extend the `generation' and `propagation' of the interface results to 
the discrete PDE \eqref{eq:1.DHDeq} as $N\uparrow\infty$ and $K=K(N)\uparrow\infty$, in Theorems \ref{thm:6.3} and \ref{thm:3.4}, by employing a comparison argument. 
Finally, as a consequence, the proof of Theorem \ref{PDEthm} is completed in Section
\ref{sec:6.3_subsec}.

\section{Models and main results}  \label{sec:2}

We now introduce the Glauber+Zero-range model in detail in Section \ref{sec:2.1}, and
state our main results, Theorems \ref{mainthm}, \ref{thm:EstHent} (probabilistic part) and 
\ref{PDEthm} (PDE part), in Section \ref{sec:2.2-main}.
Section \ref{sec:2.3} gives a proof of Theorem \ref{mainthm} assuming
Theorems \ref{thm:EstHent} and \ref{PDEthm}. 

\subsection{Glauber+Zero-range processes}  \label{sec:2.1}

Let $\T_N^d :=(\Z/N\Z)^d = \{1,2,\ldots,N\}^d$ be the $d$-dimensional
lattice of size $N$ with periodic boundary condition.
We consider, on $\T^d_N$, Glauber+Zero-range processes.
The configuration space is $\mathcal{X}_N = \{0,1,2,\ldots\}^{\T_N^d}
\equiv \Z_+^{\T_N^d}$ and its element is denoted by $\eta=\{\eta_x\}_{x\in\T_N^d}$, 
where $\eta_x$ represents the number of particles at the site $x$.
The generator of our process is of the form $L_N = N^2L_{ZR} + KL_G$, 
where $L_{ZR}$ and $L_{G}$ are Zero-range and Glauber operators, 
respectively, defined as follows.  Here, $K$ is a parameter, which will later 
depend on the scaling parameter $N$.

\medskip
\noindent {\it Zero-range specification}

To define the Zero-range part, let the jump rate 
$g=\{g(k)\ge 0\}_{k\in \Z_+}$ be given such that $g(k)=0$ if and only if $k=0$.  
Consider the symmetric simple zero-range process 
with generator
\begin{equation}  \label{eq:3.1}
L_{ZR} f(\eta) = \sum_{x\in \T_N^d}\sum_{e\in \Z^d: |e|=1} g(\eta_x) \{f(\eta^{x,x+e})-f(\eta)\},
\end{equation}
where $\eta =\{\eta_x\}_{x\in \T_N^d}\in \mathcal{X}_N$,
$|e| = \sum_{i=1}^d |e_i|$ for $e=(e_i)_{i=1}^d \in \Z^d$ and
$\eta^{x,y}\in \mathcal{X}_N$ for $x, y \in \T_N^d$ 
is defined from $\eta$ satisfying $\eta_x\ge 1$ by
$$
(\eta^{x,y})_z = \left\{\begin{array}{rl}
\eta_x-1 & {\rm when \ } z=x\\
\eta_{y}+1 & {\rm when \ }z=y\\
\eta_z & {\rm otherwise,}
\end{array}\right.
$$
for $z\in \T_N^d$; $\eta^{x,y}$ describes the configuration after 
one particle at $x$ in $\eta$ jumps to $y$. 

We remark the case $g(k)\equiv k$ corresponding to the motion of independent particles, however when $g$ is not linear, the infinitesimal interaction is nontrivial.

The invariant measures of the Zero-range process
are translation-invariant product measures $\{\bar\nu_\fa: 0\leq \fa< 
\fa^*:=\liminf_{k\to\infty} g(k)\}$ on $\mathcal{X}_N$ with one site 
marginal given by
\begin{equation}  \label{eq:3.2}
\bar\nu_\fa(k) \equiv 
\bar\nu_\fa(\eta_x=k)= \frac1{Z_\fa} \frac{\fa^k}{g(k)!}, \quad
Z_\fa = \sum_{k=0}^\infty \frac{\fa^k}{g(k)!}.
\end{equation}
Here, $g(k)!=g(1)\cdots g(k)$ for $k\geq 1$ and $g(0)!=1$; see Section 2.3 of \cite{KL}.

\begin{itemize}
\item [(De)] We assume that $\rho(\fa)=\sum_{k= 0}^\infty k \bar\nu_\fa(k)$
diverges as $\fa\uparrow \fa^*$, meaning that all densities
$0\leq \rho<\infty$ are possible in the system.
\end{itemize}
  
We denote, for $\rho\geq 0$, that
$$\nu_\rho := \bar\nu_{\fa(\rho)}$$ 
by changing the parameter so that the mean of the marginal
is $\rho$.  In fact, $\rho$ and $\fa=\fa(\rho)$ is related by
$$
\rho=\fa (\log Z_\fa)' \left( = \frac1{Z_\fa}\sum_{k=0}^\infty k \frac{\fa^k}{g(k)!}
=: \lan k \ran_{\bar\nu_\fa}\right).
$$
Also, note that
$$
\fa = \lan g(k) \ran_{\bar\nu_\fa} \left( := \frac1{Z_\fa}\sum_{k=1}^\infty \frac{\fa^k}{g(k-1)!}
\right).
$$
 Moreover, one can compute that $\fa'(\rho)= \frac{1}{\fa(\rho)}E_{\nu_\rho}\big[(\eta_0 - \rho)^2\big]>0$, and so $\fa=\fa(\rho)$ is a strictly increasing function.

We observe when $g(k)\equiv k$ that the marginals of $\nu_\rho$ are Poisson distributions with mean $\rho$.  
When $ak\leq g(k)\leq bk$ for all $k\geq 0$ with $0<a<b<\infty$, we have 
$a\rho \leq \varphi(\rho)\leq b\rho$ for $\rho\geq 0$.  When $g(k)=1(k\geq 1)$, 
i.e., $g(k)=1$ for $k\ge 1$ and $0$ for $k=0$,
we have $\fa(\rho)= \rho/(1+\rho)$ for $\rho\geq 0$.

We will need the following condition to use and prove the `Boltzmann-Gibbs principle'(cf. proofs of Theorem \ref{thm:EstHent} and Theorem \ref{BG}).

\begin{itemize}
\item [(LG)]  We assume $g(k) \le Ck$ for all $k\ge 0$ with some
$C>0$.
\end{itemize}

Later, we also consider $\bar\nu_\fa$ and $\nu_\rho$ as the product 
measures on the configuration space $\mathcal{X}=\Z_+^{\Z^d}$
on an infinite lattice $\Z^d$ instead of $\T_N^d$.

Let $u: \T_N^d \rightarrow [0,\infty)$ be a function.  We define 
the (inhomogeneous) product measure on $\mathcal{X}_N$ by
\begin{equation} \label{eq:nu-u}
\nu_{u(\cdot)}(\eta) = \prod_{x\in \T_N^d} \nu_{u(x)}(\eta_x),
\quad \eta=\{\eta_x\}_{x\in \T_N^d},
\end{equation} 
with means $u(\cdot) = \{u(x)\}_{x\in \T_N^d}$ over sites in $\T_N^d$.

In the sequel, we will assume a certain `spectral gap' bound on the
Zero-range operator:  Let $\Lambda_k = \{x\in \T^d_N: |x|\leq k\}$
for $k\ge 1$ with $N$ large enough.
Let $L_{ZR,k}$ be the restriction of $L_{ZR}$ to $\Lambda_k$, that is
$$L_{ZR,k}f(\eta) = \sum_{\stackrel{|x-y|=1}{x,y\in \Lambda_k}}g(\eta_x)\big\{f(\eta^{x,y}) - f(\eta)\big\}.$$
When there are $j\geq 0$ particles on $\Lambda_k$, the process generated by $L_{ZR,k}$ is an irreducible continuous-time Markov chain.  The operator $L_{ZR,k}$ is self-adjoint with respect to the unique canonical invariant measure $\nu_{k,j}=\nu_\beta\big\{ \cdot| \sum_{x\in \Lambda_k} \eta_x = j\big\}$; here $\nu_{k,j}$ does not depend on $\beta>0$.  For 
the operator $-L_{ZR,k}$, the value $0$ is the bottom of the spectrum.  Let $gap(k,j)$ denote the value of the next smallest eigenvalue.
\begin{itemize}
\item [(SP)] There exists $C_{gp}>0$ so that $gap(k,j)^{-1} \leq C_{gp}k^2(1+j/|\Lambda_k|)^2$ for all $k\geq 2$ and $j\geq 0$.
\end{itemize}
Such bounds have been shown for Zero-range processes with different jump rates $g$:
\begin{itemize}
\item Suppose there is $C$, $r_1>0$ and $r_2\geq 1$ such that $g(k)\leq Ck$ and $g(k+r_2)\geq g(k) + r_1$ for all $k\geq 0$.  Then, there is a constant $C_{gp}>0$ such that $gap(k,j)^{-1}\leq C_{gp}k^2$ independent of $j$ \cite{LSV}.
\item Suppose $g(k) = k^\gamma$ for $0<\gamma<1$.  Then, there is a $C_{gp}>0$ such that $gap(k,j)^{-1}\leq C_{gp}k^2(1 + j/|\Lambda_k|)^{1-\gamma}$ \cite{Nagahata}.
\item Suppose $g(k)=1(k\geq 1)$.  Then, there is a $C_{gp}>0$ such that $gap(k,j)^{-1}\leq C_{gp}k^2(1 + j/|\Lambda_k|)^2$ \cite{Morris}, \cite{LSV}.   
\end{itemize}
We remark that all of these $g$'s satisfy (De) and (LG).

\medskip

\noindent {\it Glauber specification}

For Glauber part, we consider the creation and annihilation of a single particle 
when a change happens,
though it is possible to consider the case that several particles are created or
annihilated at once.  Let $\t_x$ be the shift acting on $\mathcal{X}_N$
so that $\t_x\eta = \eta_{\cdot + x}$ for $\eta\in \mathcal{X}_N$ and $\t_xf(\eta) = f(\t_x \eta)$ for
functions $f$ on $\mathcal{X}_N$.

The generator of the Glauber part is given by
\begin{equation}  \label{eq:3.3}
L_G f(\eta) = \sum_{x\in \T_N^d}\Big[ c_x^+(\eta) \{f(\eta^{x,+})-f(\eta)\} 
+ c_x^-(\eta) 1(\eta_x\geq 1)\{f(\eta^{x,-})-f(\eta)\}\Big]
\end{equation}
where $\eta^{x,\pm}\in \mathcal{X}_N$ are determined from 
$\eta\in \mathcal{X}_N$  by $(\eta^{x,\pm})_z 
= \eta_x\pm 1$ when $z=x$ and $(\eta^{x,\pm})_z=\eta_z$ when $z\neq x$,
note that $\eta^{x,-}$ is defined only for $\eta\in \mathcal{X}_N$ satisfying $\eta_x \ge 1$.
Here, $c_x^\pm(\eta)= \t_x c^\pm(\eta)$
and $c^\pm(\eta)$ are nonnegative local functions on $\mathcal{X}$,
that is, those depending on finitely many $\{\eta_x\}$ so that these can be
viewed as functions on $\mathcal{X}_N$ for $N$ large enough.
We assume that $c^\pm(\eta)$ are written in form 
\begin{equation}  \label{eq:3.4}
c^\pm(\eta) = \hat c^{\pm}(\eta) \hat c^{0,\pm}(\eta_0),
\end{equation}
where $\hat c^{\pm}$ are functions of $\{\eta_y\}_{y\not=0}$ and $\hat c^{0,\pm}$ are functions of
$\eta_0$ only.  Moreover, since the rate of annihilation at an empty site vanishes, namely $c^-(\eta) = c^-(\eta)1(\eta_0\geq 1)$, we may take $\hat c^{0,-}(0)=0$ so that $\hat c^{0,-}(\eta_0)=\hat c^{0,-}(\eta_0)1(\eta_0\geq 1)$ and $c^-(\eta) = c^-(\eta)1(\eta_0\geq 1)$. In particular, we may drop $1(\eta_x\ge 1)$
in \eqref{eq:3.3}, since it is now included in $c_x^-(\eta)$ by the specification that $\hat c^{0,-}(0)=0$.

As an example, we may choose
\begin{equation}  \label{eq:2.5-a}
\hat c^{0,+}(\eta_0) = \frac1{g(\eta_0+1)}
\quad \text{ and } \quad
\hat c^{0,-}(\eta_0) = 1(\eta_0\ge 1)
\end{equation}
and therefore
\begin{equation}  \label{eq:2.5-b}
c_x^+(\eta) = \frac{\hat c_x^+(\eta)}{g(\eta_x+1)}
\quad \text{ and } \quad
c_x^-(\eta) = \hat c_x^-(\eta) 1(\eta_x\ge 1)
\end{equation}
with $\hat c_x^\pm(\eta) = \t_x \hat c^\pm(\eta)$;
see \eqref{eq:2c+} and \eqref{eq:2c-} below with further choices of $\hat c^\pm(\eta)$.

\medskip
{\it Glauber+Zero-range specification.}

Let now $\eta^N(t) = \{\eta_x(t)\}_{x\in \T_N^d}$ be the Markov process on
$\mathcal{X}_N$ corresponding to the Glauber+Zero-range generator 
$L_N=N^2L_{ZR}+ KL_G$.     
The macroscopically scaled empirical measure on $\T^d ( =[0,1)^d$ with 
the periodic boundary) associated with $\eta\in \mathcal{X}_N$ is defined by
\begin{align*}
& \a^N(dv;\eta) = \frac1{N^d} \sum_{x\in \T_N^d}
\eta_x \de_{\frac{x}{N}}(dv),\quad v \in \T^d, 
\intertext{and we denote}
& \a^N(t,dv) = \a^N(dv;\eta^N(t)), \quad t \geq 0.
\end{align*}
Define $\langle \alpha,\phi\rangle$ to be the integral $\int \phi d\alpha$ with respect to test functions $\phi$ and measure $\alpha$ on $\T^d$. Sometimes, when $\alpha$ has a density, $\alpha = rdv$, we will write $\langle r, \phi\rangle= \int \phi rdv$ when the context is clear.

When $K$ is a fixed parameter, one may deduce that 
a hydrodynamic limit can be shown:  The empirical measure
$\langle \a^N(t,dv), \phi\rangle$ with $\phi$ converges to 
$\langle \rho(t,v)dv, \phi\rangle$ as $N\to\infty$ in probability
if initially this limit holds at $t=0$, where $\rho(t,v)$ is a unique weak solution
of the reaction-diffusion or `nonlinear' Allen-Cahn equation,
\begin{equation}  \label{eq:RD}
\partial_t\rho= 
\Delta\varphi(\rho)+K f(\rho), \quad v\in \T^d,
\end{equation}
with an initial value $\rho_0(x)=\rho(0,x)$.  Here, functions $\varphi$ and $f$ are defined by
\begin{align}  \label{eq:P}
& \fa(\rho) \equiv \tilde g(\rho) = E_{\nu_\rho}[g(\eta_0)], \\
& f(\rho) \equiv \widetilde{c^+}(\rho) - \widetilde{c^-}(\rho)
= E_{\nu_\rho}[c^+(\eta)] - E_{\nu_\rho}[c^-(\eta)],  \label{eq:f}
\end{align}
respectively, where $E_{\nu_\rho}$ is expectation with respect to $\nu_\rho$.
As noted earlier, $\varphi$ is an increasing function since $\varphi'(\rho)=\varphi(\rho)/E_{\nu_\rho}[(\eta_0-\rho)^2]>0$.

More generally, we denote the ensemble averages of local functions $h=h(\eta)$ on $\mathcal{X}$ under
$\nu_\rho$ by
$$
\tilde h(\rho) \equiv \lan h\ran_{\nu_\rho} :=
E_{\nu_\rho}[h], \quad \rho\ge 0.
$$
It is known that $\tilde h$ is $C^\infty$-smooth, and so in particular both $\varphi, f\in C^\infty$.

Such hydrodynamic limits, and our later results do not depend on knowledge of the invariant measures of the Glauber+Zero-range process.  Indeed, when the process rates are irreducible, there is a unique invariant measure, but it is not explicit.  See \cite{Durrett_Neuhauser} for some discussion in infinite volume about these measures.

We now impose the following assumptions on the rates $c^{\pm}$:
\begin{itemize}
\item [(P)] $c^{\pm}(\eta)\geq 0$.  
\item [(BR)] $\|c^+(\eta)g(\eta_0+1)\|_{L^\infty}<\infty$ and $\|c^-(\eta^{0,+})g^{-1}(\eta_0+1)\|_{L^\infty}<\infty$.
\item [(BS)] $f$ is a `bistable' function with three zeros at $\alpha_-, \alpha_*, \alpha_+$ such that $0<\alpha_-<\alpha_*<\alpha_+$, $f'(\alpha_-)<0$, $f'(\alpha_*)>0$ and $f'(\alpha_+)<0$.  Also, the `$\fa$-balance' condition $\int_{\alpha_-}^{\alpha_+}f(\rho)\fa'(\rho)d\rho=0$ holds.
\end{itemize}
The first assumption (P)  was already mentioned.  We mention, under the choice \eqref{eq:2.5-a},
if we further impose that $g(k)\ge C_0>0$ for $k\ge 1$, 
(BR) is implied by
$$
\|\hat c^\pm(\eta)\|_{L^\infty}< \infty.
$$

Note also that $\fa(\rho)=\rho$ for the linear Laplacian so that $\fa'(\rho)=1$, in which case the `$\fa$-balance' condition is the more familiar `balance' condition $\int_{\alpha_-}^{\alpha_+} f(\rho)d\rho = 0$.

An example of the rates $c^\pm(\eta)$ and the corresponding reaction
term $f(\rho)$ determined by \eqref{eq:f} is the following. 
Define, with respect to \eqref{eq:2.5-a} and \eqref{eq:2.5-b}, that
\begin{align}  \label{eq:2c+}
& c^+(\eta) = \frac{C}{g(\eta_0+1)}\big\{(a_-+a_*+a_+)1(\eta_{e_1}\geq 1)1(\eta_{e_2}\geq 1)+ a_-a_*a_+\big\}, \\
& c^-(\eta) = \frac{C}{g(\eta_{e_3}+1)}\big\{1(\eta_{e_1}\geq 1)1(\eta_{e_2}\geq 1)+ (a_-a_* + a_-a_+ + a_*a_+)\big\}1(\eta_0\geq 1),  \label{eq:2c-}
\end{align}
where $C>0$ and $a_-, a_+, a_*>0$.  Here,  
$e_1, e_2, e_3 \in \Z^d$ are distinct points not equal to $0\in \Z^d$.  
In this case, setting $r(\rho)=E_{\nu_\rho}[1(\eta_0\geq 1)]$ and $v(\rho) = E_{\nu_\rho}[g(\eta_0 +1)^{-1}]= r(\rho)/\varphi(\rho)$, we have
$$
f(\rho) = -Cv(\rho)(r(\rho) - a_-)(r(\rho)-a_*) (r(\rho) - a_+),
$$
which has three zeros since $r(\rho)$ is strictly increasing from $0$ to $1$ as $\rho$ increases from $0$ to $\infty$.  

One can find $0<a_-<a_*<a_+<1$ so that
$\int_{\alpha_-}^{\alpha_+}f(\rho)\fa'(\rho)d\rho = 0$, where $\a_\pm = r^{-1}(a_\pm)$.
Indeed, take $0<a_-<a_+<1$ arbitrarily and observe that this integral is
negative if $a_*\in (a_-,a_+)$ is close to $a_+$, while it is positive if
$a_*$ is close to $a_-$.  When also $\inf_{k\geq 1}g(k)>0$ say, the rates $c^\pm$ satisfy conditions (P), (BR) and (BS).

\subsection{Main results}  \label{sec:2.2-main}

Let now $\mu_0^N$ be the initial distribution of $\eta^N(0)$ on ${\mathcal X}_N$.  
Let also $\{u^N(0,x)\}_{x\in \T_N^d}$ be a collection of nonnegative values and consider
the inhomogeneous product measure $\nu_0^N
:= \nu_{u^N(0,\cdot)}$ defined by \eqref{eq:nu-u}. 

We make the following assumptions on $\{u^N(0,x)\}_{x\in \T_N^d}$:

\begin{itemize}
\item [(BIP1)] $u_- \leq u^N(0,x) \leq u_+$ for some $0<u_- <u_+$.
\end{itemize}

\begin{itemize}
\item [(BIP2)] 
$u^N(0,x)= u_0(\tfrac{x}N), x\in \T_N^d$  with some $u_0\in C^5(\T^d)$.
Further, $\Ga_0:=\{v\in \T^d; u_0(v)=\a_*\}$ is a
$(d-1)$-dimensional $C^{5+\theta}$, $\theta>0$, hypersurface in $\T^d$ without boundary such that $\nabla u_0$ is
non-degenerate to the normal direction to $\Ga_0$ at every point $v\in \Ga_0$.
Also, $u_0>\alpha_*$ in $D_0^+$ and $u_-<u_0<\alpha_*$ in $D_0^-$ where $D_0^\pm$ are the regions separated by $\Gamma_0$.
\end{itemize}

Consider a family of closed smooth $C^{5+\theta}$, $\theta>0$, hypersurfaces $\{\Gamma_t\}_{t\in [0,T]}$ in $\T^d$, 
without boundary, whose evolution is governed by a `homogenized' mean curvature motion:
\begin{equation}
\label{P^0 problem}
 \Pz\quad\begin{cases}
 \, V= \laconstante\kappa
 \quad \text { on } \Gamma_t \vspace{3pt}\\
 \, \Gamma_t\big|_{t=0}=\Gamma_0\,,
\end{cases}
\end{equation}
where $V$ is the normal velocity of $\Gamma _t$ from the $\alpha_-$-side to the $\alpha_+$-side defined below,
$\kappa$ is the mean curvature at each point of
$\Gamma_t$ multiplied by $d-1$, the constant $\la_0= \la_0(\fa,f)$ is given by
\eqref{def-lambdazero-intrinsic-intro}.

In the linear case of independent particles, that is when $g(k)\equiv k$ and so $\fa(u)\equiv u$, we
recover the value $\lambda_0 =1$.
Here, $T>0$ is the time such that the $\Gamma_t$ is smooth for $t\leq T$.  
If $\Ga_0$ is smooth, such a $T>0$ always exists; see Section \ref{Formal derivation section}.

We comment that the full $C^{5+\theta}$ strength of the smoothness assumption (BIP2) is used only in Section \ref{propagation_subsection} with respect to `propagation of a discrete interface'.

Denote 
\begin{equation}
\label{hypersurface}
\chi_{\Gamma_t}(v) = \left\{\begin{array}{rl}
\alpha_- & \ {\rm for \ } v \ \text{ on one side of } \Gamma_t\\
\alpha_+& \ {\rm for \ }v \ \text{ on the other side of } \Gamma_t.\end{array}\right.
\end{equation}
These sides are determined by how $u_0$ is arranged with respect to $\Ga_0$, and then continuously kept in time
for $\Ga_t$.

We will also denote by $\P_\mu$ and $\E_\mu$ the process measure and expectation with respect to 
$\eta^N(\cdot)$ starting from initial measure $\mu$.  When $\mu = \mu^N_0$, we will call 
$\P_{\mu^N_0} = \P_N$ and $\E_{\mu^N_0}=\E_N$.  Let also $E_\mu$ denote expectation with respect to measure $\mu$.

Recall that the relative entropy between two probability measures 
$\mu$ and $\nu$ on $\mathcal{X}_N$ is given as
$$
H(\mu|\nu) := \int_{\mathcal{X}_N} \frac{d\mu}{d\nu} \log \frac{d\mu}{d\nu}  d\nu.
$$

The main result of this article is now formulated as follows.

\begin{thm} \label{mainthm}
Suppose $d\ge 2$ and the assumptions (De), (LG), (SP), (P), (BR), (BS)
stated in Section \ref{sec:2.1} and (BIP1), (BIP2).
Suppose also that the relative entropy at $t=0$ behaves as 
$$
H(\mu^N_0|\nu_0^N) = O(N^{d-\epsilon})
$$ 
as $N\uparrow\infty$, where $\epsilon>0$.  Suppose further that $K=K(N)\uparrow\infty$ as $N\uparrow\infty$ and satisfies 
$1\leq K(N)\leq  \de (\log N)^{\si/2}$, 
with respect to small $\delta = \delta(\epsilon, T)$, 
where $\si\in (0,1)$ is the H\"older exponent determined by
a Nash estimate; see Theorem \ref{lem:3.3-a}.

Then, for $0< t\leq T$, $\varepsilon>0$ and $\phi\in C^\infty(\T^d)$, we have that
\begin{equation}  \label{eq:mainthm}
\lim_{N\rightarrow\infty} \P_N\Big(\big|\langle \alpha^N(t),\phi\rangle - \langle \chi_{\Gamma_t}, \phi\rangle\big|>\varepsilon\Big) = 0.
\end{equation}
\end{thm}

As we will see in Theorem \ref{thm:3.4}, the macroscopic width of the interface $\Ga_t$
is $O(K^{-1/2})$.  Our result \eqref{eq:mainthm} shows that, apart from this area, the local
particle density, that is the local empirical average of particles' number, is close to either
$\alpha_-$ or $\alpha_+$.  In other words, the whole domain is separated into sparse 
or dense regions of particles and the interface $\Ga_t$ separating these two regions move 
macroscopically according to the motion by mean curvature $(P^0)$.

\begin{rem}  \label{rem2.0-a}\rm
In \cite{Bonaventura} and \cite{KS94}, the growth condition for $K$ was $K=O(N^\beta)$ for a small power $\beta>0$, whereas in \cite{FT}, the growth condition was $K\le \de_0\sqrt{\log N}$.  The condition here on $K$ is worse primarily due to the nonlinearity of the Zero-range rates.  
\end{rem}

The proof of Theorem \ref{mainthm} is given in two main parts. 
The first part establishes that the microscopic evolution is close to a discrete PDE motion through use of the relative entropy method and the Boltzmann-Gibbs principle, Theorem \ref{thm:EstHent}.
The second part shows that the discrete PDE evolution converges to that of the `homogenized' mean curvature flow desired, Theorem \ref{PDEthm}.

To state Theorem \ref{thm:EstHent}, let $u^N(t,\cdot) = \{u^N(t,x)\}_{x\in \T_N^d}$
be the nonnegative solution of the discretized hydrodynamic equation
\eqref{eq:1.DHDeq}, that is, 
\begin{align} \label{eq:HD-discre}
\partial_t u^N(t,x) &= \sum_{i=1}^d \De_i^N \{ \varphi(u^N(t,x))\} + K f(u^N(t,x)),
\end{align}
with initial values $u^N(0,\cdot)=\{u^N(0,x)\}_{x\in \T^d_N}$, where
\begin{align} \label{eq:DeNi}
\De_i^N \fa(u(x)) := N^2 \left(\fa(u(x+e_i)) + \fa(u(x-e_i)) - 2\fa(u(x))\right),
\end{align}
where $u(\cdot) = \{u(x)\}_{x\in\T_N^d}$ and $\{e_i\}_{i=1}^d$ are standard unit basis vectors of $\Z^d$.  Recall also that $\fa$ and $f$ are
functions given by \eqref{eq:P} and \eqref{eq:f}, respectively.
We will later denote
\begin{align} \label{eq:DeN2}
\De^N = \sum_{i=1}^d \De_i^N.
\end{align}

Let $\nu^N_t = \nu_{u^N(t,\cdot)}$ be the inhomogeneous product measure
with Zero-range marginals defined by \eqref{eq:nu-u} from $u^N(t,\cdot)$ 
for $t\geq 0$.

The next theorem shows that the `microscopic motion is close to the discretized 
hydrodynamic equation'.  We note this result holds in all $d\geq 1$.

\begin{thm} \label{thm:EstHent} 
Suppose $d\geq 1$ and let $\mu^N_t$ be the distribution of $\eta^N(t)$ on $\mathcal{X}_N$.  Suppose all conditions in Section \ref{sec:2.1} and that
(BIP1) holds with respect to $u^N(0)$ and the initial measure $\mu_0^N$ is such that
$$H(\mu_0^N|\nu_0^N) = O(N^{d-\epsilon})$$
 as $N\to\infty$ for some $\epsilon>0$.
Then, when $K=K(N)$ is a sequence as in the statement of Theorem \ref{mainthm}, we have, for an $0<\epsilon_1 = \epsilon_1(\epsilon, d)$, that
$$
H(\mu_t^N|\nu_t^N) = O(N^{d-\epsilon_1})
$$
for $t\in [0,T]$ 
as $N\to\infty$.  
\end{thm}

We comment that $\epsilon_1$ can be taken as $\epsilon_1 = (\varepsilon_0\wedge \epsilon)/2$ where $\varepsilon_0 = 2d/(9d+2)$.

We now capture the behavior of $u^N(t)$ as $N\uparrow\infty$ in terms of the motion
by mean curvature $(P^0)$ when $d\geq 2$.
Define the step function
\begin{equation}
\label{u^N(t,v)}
u^N(t,v) = \sum_{x\in \T_N^d} u^N(t,x) 1_{B(\frac{x}N,\frac1N)}(v), \quad v \in \T^d,
\end{equation}
where $B(\frac{x}N,\frac1N) = \prod_{i=1}^d [\frac{x_i}N - \frac1{2N}, 
\frac{x_i}N + \frac1{2N})$ is a box with center $\frac{x}N$, $x=(x_i)_{i=1}^d$, and side length $\frac1N$.
The following theorem is shown in Section \ref{sec:6.3_subsec}.

\begin{thm}
\label{PDEthm}
Let $d\geq 2$ and assume (BS), (BIP1) and  (BIP2).
Then, for $v\not\in \Gamma_t$ and $t\in (0,T]$, we have that
$$\lim_{N\rightarrow\infty} u^N(t,v) =\chi_{\Gamma_t}(v).$$
\end{thm}

\subsection{Proof of Theorem \ref{mainthm}}  \label{sec:2.3}

As we mentioned, Theorem \ref{mainthm} is shown mainly as a combination of Theorems \ref{thm:EstHent} and \ref{PDEthm}.
To make this precise, define, for $\varepsilon>0$ and a test function $\phi\in C^\infty(\T^d)$, the event
$$\mathcal{A}^\varepsilon_{N,t} = \{\eta\in \mathcal{X}_N; |\langle \alpha^N,\phi\rangle - \langle u^N(t,\cdot),\phi\rangle|>\varepsilon\}.$$

\begin{prop}
\label{exp_prop}
There exists $C=C(\varepsilon)>0$ such that 
$$\nu^N_t(\mathcal{A}^\varepsilon_{N,t}) \leq e^{-CN^d}.$$
\end{prop}

\begin{proof}
Write
$$\langle \alpha^N,\phi\rangle - \langle u^N(t,\cdot),\phi\rangle = \frac{1}{N^d} \sum_{x\in \T^d_N} (\eta_x - u^N(t,x))\phi(x/N) + o(1).$$
Under $\nu^N_t$, the variable $\eta_x$ has mean $u^N(t,x)$ and a variance $\sigma^2_{x,t}$ in terms of $u^N(t,x)$.
Under the condition (BIP1), by the comparison Lemma \ref{lem:4.1}, we have that $u^N(t,\cdot)$, and so also $\sigma^2_{x,t}$, is uniformly bounded away from $0$ and $\infty$.
 
The desired bound, since $\phi$ is uniformly bounded, follows from a standard application of exponential Markov inequalities.
\end{proof}

Now note that the entropy inequality, for an event $A$, gives
$$\mu^N_t(A) \leq \frac{\log 2 + H(\mu^N_t|\nu^N_t)}{\log\{1+ 1/\nu^N_t(A)\}}.$$
Combined with Proposition \ref{exp_prop} and the relative entropy Theorem \ref{thm:EstHent},
we have that
$$\lim_{N\rightarrow\infty}
\mu^N_t(\mathcal{A}^\varepsilon_{N,t}) = 0.$$

However, the discrete PDE convergence Theorem \ref{PDEthm} shows that $\langle u^N(t,\cdot), \phi\rangle \rightarrow \langle \chi_{\Gamma_t}, \phi\rangle$ as $N\uparrow\infty$, finishing
the proof of Theorem \ref{mainthm}.

\section{Comparison, a priori estimates, and a `Boltzmann-Gibbs' principle} 
\label{sec:3}

Let $u^N(t, \cdot) = \{u^N(t,x)\}_{x\in \T_N^d}$ be the nonnegative solution of 
the discretized hydrodynamic equation \eqref{eq:HD-discre} or \eqref{eq:1.DHDeq}
with given sequence $1\leq K=K(N)$.  In this section, we do not impose a growth 
condition on $K=K(N)$, stating results in terms of $K$.

\subsection{Comparison theorem}  \label{sec:3.1-comp}

The equation \eqref{eq:HD-discre} satisfies a comparison theorem; 
cf.\ \cite{FS}, Section 2.5.
We will say that profiles $u(\cdot)=(u_x)_{x\in \T_N^d}$ and 
$v(\cdot) =(v_x)_{x\in \T_N^d}$ are ordered $u(\cdot)\ge v(\cdot)$ 
when $u_y\ge v_y$ for all $y\in \T_N^d$.

We say that $u^+(t,\cdot)$ and $u^-(t,\cdot)$ are super and sub solutions of \eqref{eq:HD-discre},
if $u^+$ and $u^-$ satisfy \eqref{eq:HD-discre} with \lq\lq$\ge"$
and \lq\lq$\le"$ instead of \lq\lq$="$ respectively.

\begin{lem}\label{lem:4.1}
Suppose initial conditions $u^-(0,\cdot) \le u^+(0,\cdot)$.   Then, the corresponding super and sub solutions $u^+(t,\cdot)$ and $u^-(t,\cdot)$ to the discrete PDE \eqref{eq:HD-discre}, for all $t\geq 0$, satisfy
$$u^-(t,\cdot) \le u^+(t,\cdot).$$

Furthermore, suppose (BIP1) holds:  $u_-\le u^N(0,x)\le u_+$ for some $0<u_-<u_+<\infty$. Then,
for $t\ge 0$ and $x\in \T^d_N$, we have
$$u_-\wedge \alpha_- \le u^N(t,x)\le u_+\vee\alpha_+.$$
\end{lem}

\begin{proof}
Assume that $u^+(t,\cdot)\ge u^-(t,\cdot)$ and $u^-(t,x)=u^+(t,x)$ holds at some space-time point $(t,x)$.
Then, since the reaction term $f$ cancels, and $\fa$ is an increasing function, we have
\begin{align*}
\partial_t (u^+-u^-)(t,x) 
& \ge \De^N \{\fa(u^+)-\fa(u^-)\}(t,x) + K \big(f(u^+(t,x))- f(u^-(t,x))\big)  \\
& = N^2\sum_{\pm e_i} \big\{ (\fa(u^+)-\fa(u^-))(t,x\pm e_i)
- (\fa(u^+)-\fa(u^-))(t,x) \big\}  \\
& = N^2\sum_{\pm e_i} \{\fa(u^+)-\fa(u^-)\} (t,x\pm e_i) \ge 0.
\end{align*}
This implies $\partial_t (u^+-u^-)(t,x) \ge0$ and
shows that $u^-(t)$ can not exceed $u^+(t)$ for all $t>0$.

In particular, if we take $u^+(0,x)\equiv u_+\vee \alpha_+$, then by the condition 
(BS), the solution $u^+(t,\cdot)$ 
with this initial datum
is decreasing in $t$ so that we obtain $u^N(t,\cdot) \le u^+(t,\cdot) \le u_+\vee \alpha_+$.
We can similarly show $u^N(t,\cdot) \ge u_-\wedge \alpha_-$.
\end{proof}

\subsection{A priori estimates}

Define for $\{u_x = u(x)\}_{x\in \T^d}$ and $1\leq i \leq d$,
\begin{align*}
\nabla_i^Nu(x) &= N\big(u(x+e_i)-u(x)\big), \ \ {\rm and \ }\\
\nabla^Nu(x)&= \big(\nabla_i^Nu(x)\big)_{i=1}^d.
\end{align*}

\begin{lem}  \label{u_cont_lem}
(cf.\ \cite{FS}, Section 4.4)
Suppose bounds (BIP1) hold for $u^N(0,\cdot)$.  Then, for a constant $C>0$,
we have
\begin{align*}
&\frac12\sum_{x\in \T_N^d} u^N(t,x)^2 + c_0  \int_0^T \sum_{x\in \T_N^d} |\nabla^Nu^N(t,x)|^2 dt
\le \frac12\sum_{x\in \T_N^d} u^N(0,x)^2 + CKTN^d,
\end{align*}
where $c_0:= \inf_{\rho>0}\fa'(\rho)>0$ (see \cite{KL} p.30), and as a consequence
\begin{align}
\label{u_continuity}
\frac{N^2}{\ell^2}\frac{1}{N^d}\int_0^T \sum_{x\in \T_N^d} \Big(\frac{1}{(2\ell+1)^d}\sum_{|z-x|\leq \ell} u^N(t, z) - u^N(t,x)\Big)^2 dt \leq \frac{CKT}{c_0},
\end{align}
where $|x|= \sum_{i=1}^d |x_i|$ for $x = (x_i)_{i=1}^d \in \Z^d$.
\end{lem}

\begin{proof}
Recall $u^N(t,\cdot)$ is the solution of \eqref{eq:HD-discre}.  By Lemma \ref{lem:4.1}, we have that
$u^N(t,\cdot)$ is between $u_-^* = u_-\wedge\alpha_-$ and $u_+^* = u_+\vee \alpha_+$
uniformly in time.  Since
$\fa'(u)\ge c_0>0$ and $f(u)$ is bounded for $u$ between $u_-^*$ and $u_+^*$, we have by the mean-value theorem that
\begin{align*}
&\tfrac12 \partial_t \sum_{x\in \T^d_N} u^N(t,x)^2
 = \sum_{x\in \T^d_N} u^N(t,x) \left(\De^N \fa(u^N(t,x)) +Kf(u^N(t,x))\right) \\
&\ \ \ \ \ \  = - \sum_{x\in \T^d_N} \sum_{i=1}^d\nabla_i^N u^N(t,x) \nabla_i^N \fa(u^N(t,x)) + K\sum_{x\in \T^d_N} u^N(t,x) f(u^N(t,x)) \\
& \ \ \ \ \  \ \le - c_0 \sum_{x\in \T^d_N} \sum_{i=1}^d|\nabla_i^Nu(t,x)|^2 + C KN^d.
\end{align*}
Integrating in time gives the first inequality in the lemma.  The second inequality now follows from the first, utilizing Jensen's inequality and the relation $(a_1+\cdots + a_j)^2 \leq j(a_1^2 + \cdots + a_j^2)$.
\end{proof}

\subsection{$L^\infty$-estimates on discrete derivatives}
\label{energy_subsec}

We next state the $L^\infty$-estimates for the (macroscopic) discrete derivatives of
the solution $u^N(t,x)$ of \eqref{eq:HD-discre}. We define the norm
$\|u^N\|_{C_N^n}$ for $u^N=\{u^N(x)\}_{x\in \T_N^d}$ and $n=0,1,2,\ldots$ by
\begin{align*}
\|u^N\|_{C_N^n} = \sum_{k=0}^n \sum_{1\le i_1,\ldots,i_k\le d} \max_{x\in T_N^d}
|\nabla_{i_k}^N\cdots \nabla_{i_1}^Nu^N(x)|
\end{align*}
where for $n=0$ the norm reduces to $\|u^N\|_{L^\infty(\T^d_N)}$.
The following Schauder estimate is shown in \cite{FS} for quasilinear discrete
PDEs.  The constant $\si\in (0,1)$ appears as the H\"older exponent in
Nash estimate; see \cite{FS} for details.
Note that we described $u^N(x)$ or $u^N(t,x)$ as $u^N(\frac{x}N)$ or 
$u^N(t,\frac{x}N)$ in \cite{FS} by using macroscopic spatial variables $\frac{x}N$
instead of microscopic ones $x$, but these two descriptions are equivalent.

\begin{thm}  \label{lem:3.3-a}
Suppose $\|u^N(0)\|_{C_N^4}\le C_0$ and condition (BIP1): 
$0<u_-\le u^N(0,x)\le u_+<\infty$ for all $x\in \T_N^d$.  Then, we have
\begin{align}
& \|u^N(t)\|_{C_N^2} \le CK^{2/\si},   \label{eq:lem3.3-1}
\end{align}
for all $t\in [0,T]$ and some $C>0$. 

In particular, we have
\begin{align}
\label{eq:lem3.3-3}
&  \|\Delta^N\varphi(u^N(t,\cdot))\|_{L^\infty(\T^d_N)} \leq CK^{2/\si}.
\end{align}
\end{thm}

We note that $\|u^N(0)\|_{C_N^4}\leq C_0$ holds under the condition
(BIP2):  $u^N(0,x)=u_0(x/N)$ and $u_0\in C^4(\T^d)$.

\subsection{A `Boltzmann-Gibbs' principle}
\label{BG_statement}
For a local function $h=h(\eta)$, with support in a finite box denoted $\Lambda_h\subset \T_N^d$, and parameter $\beta\geq 0$, let 
$$\tilde h(\beta) = E_{\nu_\beta}[h].$$
    In this section, we suppose that the function $h$ satisfies, in terms of constants $C_1, C_2$, the bound
  \begin{equation}
	\label{h_bounds}
	|h(\eta)|\leq C_1\sum_{y\in \Lambda_h}|\eta_y| + C_2.
	\end{equation}
With respect to an evolution $\{u^N(t,x)\}_{x\in \T^d_N}$
 satisfying the discrete PDE \eqref{eq:HD-discre}, 
let 
\begin{equation}
\label{f_def}
f_x(\eta) = \tau_xh(\eta) -\tilde h(u^N(t,x)) -\tilde h'(u^N(t,x))\big(\eta_x - u^N(t,x)\big).
\end{equation}

		Recall that $\P_N$ is the underlying process measure governing $\eta^N(\cdot)$ starting from $\mu^N_0$ and $\mu^N_t$ is the distribution of $\eta^N(t)$ for $t\geq 0$.  
    Recall $K=K(N)\geq 1$ for $N\geq 1$ is speed of the Glauber jumps in the process $\eta^N(\cdot)$ with generator $L_N$.  We will not impose a growth condition here on $K$ but state results in terms of $K$.
    With respect to the evolution $u^N(t,\cdot)$, define $\nu^N_t=\nu_{u^N(t,\cdot)}$ as the inhomogeneous Zero-range product measure with stationary marginal indexed over $x\in \T^d_N$ with density $u^N(t,x)$.

We now state a so-called `Boltzmann-Gibbs' principle, under the relative entropy assumption $H(\mu^N_0|\nu^N_0) = O(N^d)$, weaker than the one assumed for Theorem \ref{thm:EstHent}.  It is a `second-order' estimate valid in $d\geq 1$ with a remainder given in terms of a relative entropy term and a certain error.

\begin{thm}
\label{BG}
Suppose bounds (BIP1) hold for the initial values $\{u^N(0,x)\}_{x\in \T^d_N}$, and the initial relative entropy $H(\mu^N_0|\nu^N_0) = O(N^d)$. 
Suppose $\{a_{t,x}: x\in \T^d_N, t\geq 0\}$ are non-random coefficients with uniform bound
\begin{equation}
\label{coeff_bound}
\sup_{x\in \T^d_N, t\geq 0} |a_{t,x}| \leq M.
\end{equation}

 Then, there exist $\epsilon_0, C>0$ such that
\begin{equation}
\label{BG_equation}
\E_{N} \left |\int_0^T  \sum_{x\in \T^d_N} a_{t,x} f_x dt\right | \leq O(MKN^{d-\epsilon_0}) 
+ CM \int_0^T H(\mu_t^N| \nu^N_t) \, dt.
\end{equation}
Moreover, we may take $\epsilon_0 = 2d/(9d+2)$.
\end{thm}

The proof of Theorem \ref{BG} is given in Section \ref{BG_section}.  

\begin{rem}\rm
We remark that this proof relies on the form of the discrete PDE \eqref{eq:HD-discre} only in that $u^N$ satisfies the statements in Lemmas \ref{lem:4.1} and \ref{u_cont_lem}.  
\end{rem}

\section{Microscopic motion is close to the `discrete PDE':  Proof of Theorem \ref{thm:EstHent}}
\label{sec:4}

Recall the Glauber+Zero-range process $\eta^N(t)$ generated 
by
$L_N=N^2L_{ZR}+K(N)L_G$, where $K=K(N)$.  For a function $f$ on $\mathcal{X}_N$ 
and a measure $\nu$ on $\mathcal{X}_N$, set
\begin{align*} 
\mathcal{D}_N(f;\nu) = 2N^2 \mathcal{D}_{ZR}(f;\nu) + K \mathcal{D}_G(f;\nu),
\end{align*}
 where
\begin{align}
\label{Dirichlet_eq}
\mathcal{D}_{ZR}(f;\nu) & = \frac14 \sum_{\stackrel{|x-y|=1}{x,y\in\T_N^d}} \int_{\mathcal{X}_N} 
 g(\eta_x) \{f(\eta^{x,y})-f(\eta)\}^2 d\nu,  \\
\mathcal{D}_G(f;\nu) & = \sum_{x\in\T_N^d} \int_{\mathcal{X}_N} 
 c^+_x(\eta)\{f(\eta^{x,+})-f(\eta)\}^2 + c^-_x(\eta)\{f(\eta^{x,-})-f(\eta)\}^2 d\nu,  \nonumber
\end{align}
and recall $c_x^-(\eta)=0$ when $\eta_x=0$.

Recall $\mu_t^N$ is the law of
$\eta^N(t)$ on $\mathcal{X}_N$ and $\nu^N_t = \nu_{u^N(t,\cdot)}$.  
Let $m$ be a reference measure on 
$\mathcal{X}_N$ with full support in $\mathcal{X}_N$.  Define 
$$\psi^N_t := \frac{d\nu^N_t}{dm}.$$
In general, we denote the adjoint of an operator $L$ on $L^2(\nu^N_t)$ by $L^{*,\nu^N_t}$.

We now state an estimate for the derivative of relative entropy.  Such estimates go back to the work of Guo-Papanicolaou-Varadhan (cf. \cite{KL}) and Yau \cite{Y}.  A more recent bound is the following; see \cite{F18}, \cite{FT} and \cite{JM2} for a proof.

\begin{prop} \label{thm:4.2}
\begin{equation*} 
\frac{d}{dt} H(\mu_t^N|\nu^N_t) \le 
- \mathcal{D}_N\left(\sqrt{\frac{d\mu_t^N}{d\nu^N_t}}; \nu^N_t\right) + 
\int_{\mathcal{X}_N} (L_N^{*,\nu^N_t}1 - \partial_t \log \psi^N_t) d\mu_t^N.
\end{equation*}
\end{prop}

We remark that in our later development we need only the inequality, originally derived in \cite{Y}, where the Dirichlet form term is dropped:
\begin{equation}\label{eq:dH}
\frac{d}{dt} H(\mu_t^N|\nu^N_t) \le 
\int_{\mathcal{X}_N} (L_N^{*,\nu^N_t}1 - \partial_t \log \psi^N_t) d\mu_t^N.
\end{equation}

To control the relative entropy $H(\mu^N_t| \nu^N_t)$ we will develop a bound of the right-hand side of \eqref{eq:dH} in the following subsection.  With the aid of these bounds, which use a `Boltzmann-Gibbs' estimate shown in Section \ref{BG_section}, we later give a proof of Theorem \ref{thm:EstHent} in Section \ref{proof_EstHent_section}.

\subsection{Computation of $L^{*, \nu^N_t}_N 1 - \partial_t \log \psi^N_t(\eta)$}
\label{entropy_subsection}

We first formulate a few lemmas in the abstract.
Let $\{u(x)\ge 0\}_{x\in \T_N^d}$ be given and
let $\nu=\nu_{u(\cdot)}$ be the product measure 
given as in \eqref{eq:nu-u}.  Recall that $\De^N_i$ and $\De^N$ are
defined in \eqref{eq:DeNi} and \eqref{eq:DeN2}, respectively.

\begin{lem}  \label{lem:3.1}
  We have
\begin{align*}
L_{ZR}^{*,\nu} 1 & = \sum_{x\in\T_N^d} \frac{N^{-2}(\De^N\fa)(u(x))}{\fa(u(x))} g(\eta_x) \\
& = \sum_{x\in\T_N^d} \frac{N^{-2}(\De^N\fa)(u(x))}{\fa(u(x))} 
\{g(\eta_x) - \fa(u(x))\}.
\end{align*}
\end{lem}

\begin{proof} 
Similar computations results are found in \cite{KL}, pp.120--121.
Take any $f=f(\eta)$ on $\mathcal{X}_N$ as a test function and compute
\begin{align*}
\int L_{ZR}^{*,\nu} 1\cdot f d \nu &= \int L_{ZR}f d \nu \\
&=  \sum_{x\in \T_N^d}\sum_{|e|=1} \sum_{\eta\in\mathcal{X}_N}
 g(\eta_x) \{f(\eta^{x,x+e})-f(\eta)\} \nu(\eta).
\end{align*}
Then, by fixing $x,e$ and making change of variables $\zeta = \eta^{x,x+e}$, we have
$$
\sum_\eta g(\eta_x) f(\eta^{x,x+e}) \nu(\eta)
= \sum_\zeta g(\zeta_x+1) f(\zeta) \nu(\zeta^{x+e,x}).
$$
However, since
\begin{align*}
\nu(\zeta^{x+e,x}) &= \frac{\nu_{u(x+e)}(\zeta_{x+e}-1)}{\nu_{u(x+e)}(\zeta_{x+e})}
\frac{\nu_{u(x)}(\zeta_x+1)}{\nu_{u(x)}(\zeta_x)} \nu(\zeta) \\
& = \frac{g(\zeta_{x+e})}{\fa(u(x+e))} \frac{\fa(u(x))} {g(\zeta_x+1)} \nu(\zeta),
\end{align*}
we obtain
\begin{align*}
L_{ZR}^{*,\nu} 1 & = \sum_{x,e} \left\{ \frac{\fa(u(x))}{\fa(u(x+e))} g(\eta_{x+e}) - g(\eta_x)\right\} \\
& =  \sum_{x,e} \left\{ \frac{\fa(u(x-e))}{\fa(u(x))} -1 \right\} g(\eta_x)
= \sum_{x} \frac{N^{-2}(\De^N\fa)(u(x))}{\fa(u(x))} g(\eta_x).
\end{align*}
The last equality follows by noting that $\sum_{x} (\De^N\fa)(u(x))=0$.
\end{proof}

\begin{lem}  \label{lem:3.2}
We have
\begin{align*}
L_G^{*,\nu} 1 = \sum_{x\in\T_N^d}
 \left\{ c_x^+(\eta^{x,-}) \frac{g(\eta_x)}{\fa(u(x))}
+ c_x^-(\eta^{x,+}) \frac{\fa(u(x))}{g(\eta_x+1)} -c_x^+(\eta) -c_x^-(\eta)
\right\}.
\end{align*}
\end{lem}

\begin{proof} 
Taking any $f=f(\eta)$ on $\mathcal{X}_N$, we have
\begin{align*}
&\int L_G^{*,\nu} 1\cdot f d \nu = \int L_G f d \nu \\
&=  \sum_{x\in \T_N^d} \sum_{\eta\in \mathcal{X}_N}
 \Big\{c^+_x(\eta) \{f(\eta^{x,+})-f(\eta)\}
+ c^-_x(\eta)1(\eta_x\geq 1) \{f(\eta^{x,-})-f(\eta)\} \Big\}\nu(\eta)
\end{align*}
Then, by making change of variables $\zeta = \eta^{x,\pm}$, we have
\begin{align*}
\sum_\eta c^+_x(\eta) f(\eta^{x,+})\nu(\eta) &= \sum_\zeta 
c^+_x(\zeta^{x,-})1(\zeta_x\geq 1) f(\zeta)\nu(\zeta^{x,-}),   \\
 \sum_\eta c^-_x(\eta)1(\eta_x\geq 1)f(\eta^{x,-})\nu(\eta)& = \sum_\zeta 
c^-_x(\zeta^{x,+}) f(\zeta)\nu(\zeta^{x,+}).
\end{align*}

However, since
\begin{align*}
& \nu(\zeta^{x,-})1(\zeta_x\geq 1) = 1(\zeta_x\geq 1)\frac{\nu_{u(x)}(\zeta_x-1)}{\nu_{u(x)}(\zeta_x)} \nu(\zeta)
= \frac{g(\zeta_x)}{\fa(u(x))} \nu(\zeta),\\
& \nu(\zeta^{x,+}) = \frac{\nu_{u(x)}(\zeta_x+1)}{\nu_{u(x)}(\zeta_x)} \nu(\zeta)
= \frac{\fa(u(x))}{g(\zeta_x+1)} \nu(\zeta),
\end{align*}
we obtain
\begin{align*}
L_G^{*,\nu} 1 & = \sum_{x} \left\{ c_x^+(\eta^{x,-}) \frac{g(\eta_x)}{\fa(u(x))}
+ c_x^-(\eta^{x,+}) \frac{\fa(u(x))}{g(\eta_x+1)} -c_x^+(\eta) -c_x^-(\eta)1(\eta_x\geq 1)\right\}.
\end{align*}
Finally, by our convention with respect to $c_x^-$, we have that $c_x^-(\eta)1(\eta_x\geq 1) = c_x^-(\eta)$.  \end{proof}

\begin{exa}  \label{cor:3.3}
If we choose $c_x^\pm(\eta)$ as in \eqref{eq:2.5-b}, noting that 
$\hat c_x^\pm(\eta)$ do not depend on $\eta_x$, we have $L_G^{\nu,*}1$ equals
\begin{align*}
  \sum_{x\in\T_N^d} \hat c_x^+(\eta) \left(\frac{1(\eta_x\geq 1)}{\fa(u(x))} -\frac{1}{g(\eta_x+1)}\right)
+ \sum_{x\in\T_N^d} \hat c_x^-(\eta) \left(\frac{\fa(u(x))}{g(\eta_x + 1)} -1(\eta_x\geq 1)\right).
\end{align*}
\end{exa}

\begin{lem}
\label{psi_derivative_lem}
Now we take $u(\cdot)=\{u^N(t,x)\}_{x\in\T_N^d}$.  Then, we have
\begin{equation*}
\partial_t \log\psi^N_t(\eta) = \sum_{x\in\T_N^d}
\frac{\partial_t \fa(u^N(t,x))}{\fa(u^N(t,x))} (\eta_x-u^N(t,x)).
\end{equation*}
\end{lem}

\begin{proof}
Since
$$
\psi^N_t(\eta) = \frac{\nu_{u^N(t,\cdot)}(\eta)}{m(\eta)} = \frac{\prod_x\nu_{u^N(t,x)}(\eta_x)}{m(\eta)},
$$
we have
$$
\partial_t \log \psi^N_t(\eta) = \sum_{x\in \T^d_N} \frac{\partial_t \nu_{u^N(t,x)}(\eta_x)}{\nu_{u^N(t,x)}(\eta_x)}.
$$
Here,
\begin{align*}
&\partial_t \nu_{u^N(t,x)}(k) 
= \partial_t \left( \frac1{Z_{\fa(u^N(t,x))}} \frac{\fa(u^N(t,x))^k}{g(k)!}\right)  \\
&\ \ \  = \frac1{Z_{\fa(u^N(t,x))}} \frac{k\fa(u^N(t,x))^{k-1}}{g(k)!} \partial_t \fa(u^N(t,x))
- \frac{Z_{\fa(u^N(t,x))}'\partial_t \fa(u^N(t,x))}{Z_{\fa(u^N(t,x))}^2} \frac{\fa(u^N(t,x))^k}{g(k)!} \\
&\ \ \  = \nu_{u^N(t,x)}(k)  \partial_t \fa(u^N(t,x)) \frac1{ \fa(u^N(t,x)) } (k-u^N(t,x)),
\end{align*}
where we have used the formula
$
\frac{\partial}{\partial\fa} \log Z_\fa = {\rho}/{\fa}$.
This shows the conclusion.
\end{proof}

These three lemmas, combined with the comparison estimates, discrete derivative bounds, and Boltzmann-Gibbs principle in Section \ref{sec:3}, are the main ingredients for the following
theorem.

\begin{thm}
\label{cor_BG}
Suppose $u^N(t,x)$ satisfies \eqref{eq:HD-discre}, with $K\ge 1$.
Then, there are $\varepsilon_0, C>0$ such that 
\begin{align*}
&\int_0^T \int_{\mathcal{X}_N} \Big\{L^{*,\nu^N_t}_N 1 - \partial_t\log \psi^N_t\Big\}d\mu^N_t dt\\
&\ \ \ \ \ \leq CK^{2/\si}\int_0^TH(\mu^N_t|\nu^N_t)dt+ O\big(K^{1+2/\si}N^{d-\varepsilon_0}\big).
\end{align*}
\end{thm}

\begin{proof}
By Lemmas \ref{lem:3.1}, \ref{lem:3.2} and \ref{psi_derivative_lem}, we have
$L^{*,\nu^N_t}_N 1 - \partial_t\log \psi^N_t$ equals
\begin{align}
\label{eq:cor_BG}
&\sum_{x} \frac{(\De^N\fa)(u^N(t,x))}{\fa(u^N(t,x))} \{g(\eta_x) - \fa(u^N(t,x))\}\\
&\ \ + K \sum_{x\in \T^d_N}  \left\{ c_x^+(\eta^{x,-})\frac{g(\eta_x)}{\fa(u^N(t,x))} - c_x^+(\eta) + c_x^-(\eta^{x,+})\frac{\fa(u^N(t,x))}{g(\eta_x+1)} - c_x^-(\eta)1(\eta_x\geq 1) \right\}  \nonumber\\
&\ \ -\sum_{x\in \T^d_N} \frac{\partial_t \fa(u^N(t,x))}{\fa(u^N(t,x))} (\eta_x-u^N_x(t)). \nonumber
\end{align}
To analyze further, we will apply the 
Boltzmann-Gibbs principle, along with comparison estimates and bounds for the discrete derivatives of the discrete PDE, with respect to the first two lines in the above display \eqref{eq:cor_BG}.

First, let $h(\eta) = g(\eta_x) -\fa(u^N(t,x))$.  By the assumption (LG), $h$ satisfies the bound in \eqref{h_bounds}.  Observe that
$\tilde h(\b) \equiv E_{\nu_\b}[h] = \fa(\b)-\fa(u^N(t,x))$
as $\tilde g(\beta) \equiv E_{\nu_\beta}[g] = \fa(\beta)$ for $\b\ge 0$.
This implies $\tilde h(u^N(t,x))=0$
and $\tilde h'(u^N(t,x)) = \fa'(u^N(t,x))$.

Let now $a_{t,x} = \Delta^N\varphi(u^N(t,x)/\varphi(u^N(t,x))$.  Since $u^N$ is bounded between $u_-\wedge \alpha_-$ and $u_+\vee \alpha_+$ according to Lemma \ref{lem:4.1}, $\varphi(u^N(t,x))$ is uniformly bounded away from $0$.  Also, by Theorem \ref{lem:3.3-a}, we have the estimate $\|\Delta^N\varphi(u^N(t,\cdot))\|_{L^\infty} = O(K^{2/\si})$. Then, we conclude that $\|a(t,\cdot)\|_{L^\infty} = O(K^{2/\si})$.

Therefore, by the Boltzmann-Gibbs principle (Theorem \ref{BG}), we obtain that
\begin{align*}
&\E_{N}\Big| \int_0^T \sum_{x\in \T^d_N} \frac{(\De^N\fa)(u^N(t,x))}{\fa(u^N(t,x))}
\left(g(\eta_x(t))-\varphi(u^N(t,x))\right) dt \\
& \quad - \int_0^T\sum_{x\in \T^d_N} 
\frac{(\De^N\fa)(u^N(t,x))}{\fa(u^N(t,x))}
\fa'(u^N(t,x))\big(\eta_x(t)-u^N(t,x)\big)dt \Big| \\
& \ \ \leq CK^{2/\si}\int_0^T H(\mu^N_t| \nu^N_t)dt + O(K^{1+2/\si} N^{d-\varepsilon_0}).
\end{align*}

Secondly, observe that
\begin{equation}  \label{eq:g-1}
\widetilde{\Big(\frac{c^-_x(\eta^{x,+})}{g(\eta_x + 1)}\Big)}(\beta) 
\equiv E_{\nu_\beta} \left[\frac{c_x^-(\eta^{x,+})}{g(\eta_x+1)}\right]
 = \frac{1}{\varphi(\beta)}E_{\nu_\beta} [c_x^-(\eta)1(\eta_x\geq 1)].
\end{equation}
Indeed, recall $c_x^-(\eta) = \hat c_x^-(\eta) \hat c_x^{0,-}(\eta_x)$ where $\hat c_x^-$ does not depend on $\eta_x$.  Then, 
$$
E_{\nu_\beta}[c^-_x(\eta^{x,+})g(\eta_x+1)^{-1}] = E_{\nu_\beta}[\hat c_x^-(\eta)]E_{\nu_\beta}[\hat c_x^{0,-}(\eta_x +1) g(\eta_x+1)^{-1}].
$$
The factor $E_{\nu_\beta}[\hat c_x^{0,-}(\eta_x +1) g(\eta_x+1)^{-1}]$ is rewritten as 
\begin{align*}
&\frac1{Z_\fa} \sum_{k=0}^\infty \frac{\hat c_x^{0,-}(k+1)}{g(k+1)} \frac{\fa^k}{g(k)!}
 = \frac1{Z_\fa} \fa^{-1} \sum_{k=0}^\infty \frac{\fa^{k+1}}{g(k+1)!}\hat c_x^{0,-}(k+1) \\
&\ \ \ \ \ \ \ \ \ \ \ \ \  = \fa^{-1} \frac1{Z_\fa} \sum_{k=0}^\infty \hat c_x^{0,-}(k)1(k\ge 1) \frac{\fa^k}{g(k)!} = \frac{1}{\fa}E_{\nu_\beta}[c_x^{0,-}(\eta_x)1(\eta_x\geq 1)],
\end{align*}
where $\fa=\fa(\b)$.  This shows \eqref{eq:g-1} by noting the
independence of $\hat c_x^-(\eta)$ and functions of $\eta_x$ under $\nu_\b$.

Let now
$$
h(\eta) =  c_x^-(\eta^{x,+}) \frac{\fa(u^N(t,x))}{g(\eta_x + 1)} - c_x^-(\eta_x)1(\eta_x\geq 1).
$$
Since $h$ is seen to be uniformly bounded by assumption (BR), condition \eqref{h_bounds} holds.  Moreover,
from \eqref{eq:g-1}, we see
$$
\tilde h(\b) = E_{\nu_\b}[\hat c_x^-(\eta)]\frac{E_{\nu_\b}[\hat c^{0,-}_x(\eta_x)1(\eta_x\geq 1)]}{\fa(\b)}
\big(\fa(u^N(t,x))-\fa(\b)\big).
$$
Then, in particular $\tilde h(u^N(t,x))=0$ and 
$$
\tilde h'(u^N(t,x)) = - E_{\nu_{u^N(t,x)}}[\hat c_x^-(\eta)]\frac{E_{\nu_{u^N(t,x)}}[\hat c_x^{0,-}(\eta_x)1(\eta_x\geq 1)]}{\fa(u^N(t,x))}
\fa'(u^N(t,x)).
$$
Since $\hat c_x^{0,-}(0)=0$ by our convention, we see that $E_{\nu_{\beta}}[\hat{c}_x^-(\eta)]E_{\nu_{\beta}}[\hat c_x^{0,-}(\eta_x)1(\eta_x\geq 1)]
= E_{\nu_\b}[c_x^-(\eta)]$.

Let now $a_{t,x} \equiv K$.  By
the Boltzmann-Gibbs principle, Theorem \ref{BG}, we conclude that
\begin{align*}
&
\E_{N}\Big|K\int_0^T \sum_{x\in \T^d_N} \left\{c_x^{-}(\eta^{x,+}(t))
\frac{\varphi(u^N(t,x))}{g(\eta_x(t) +1)} 
  - c_x^-(\eta_x(t))1(\eta_x(t)\geq 1)  \right\}  dt \\
&\ \ \ \ \ \  + K\int_0^T \sum_{x\in \T^d_N} E_{\nu_{u^N(t,x)}}[c^-(\eta)]\frac{\fa'(u^N(t,x))}{\fa(u^N(t,x))} \big(\eta_x(t) - u^N(t,x)\big)dt \Big|\\
&\ \ \ \ \ \ \ \ \ \leq CK \int_0^T H(\mu^N_t|\nu^N_t)dt + O(K^2N^{d-\varepsilon_0}).
\end{align*} 

Thirdly, we consider
$$
h(\eta) = c_x^+(\eta^{x,-})\frac{g(\eta_x)}{\varphi(u^N(t,x))} - c_x^+(\eta).
$$
Again, by the assumption (BR), $h$ is uniformly bounded and so satisfies \eqref{h_bounds}.
Also, from a calculation similar to \eqref{eq:g-1}, we see
\begin{align*}
\tilde h(\b) 
&= \frac{E_{\nu_\b}[c^+(\eta)]}{\fa(u^N(t,x))}(\fa(\b)-\fa(u^N(t,x))).
\end{align*}
Therefore, for this choice $\tilde h(u^N(t,x))=0$ and
$$
\tilde h'(u^N(t,x)) = \frac{E_{\nu_{u^N(t,x)}}[c^+(\eta)]}{\fa(u^N(t,x))} \fa'(u^N(t,x)).
$$ 

Here, also let $a_{t,x}\equiv K$.
Again, by the Boltzmann-Gibbs principle, Theorem \ref{BG}, we have that
\begin{align*}
&\E_{N}\Big| K\int_0^T \sum_{x\in \T^d_N} \left\{
c_x^+(\eta^{x,-}(t))\frac{g(\eta_x(t))}{\varphi(u^N(t,x))} 
   - c_x^+(\eta(t))  \right\}  dt \\
&\ \ \ \ \ \ -K\int_0^T
 \sum_{x\in \T^d_N} E_{\nu_{u^N(t,x)}}[c^+(\eta)]\frac{\fa'(u^N(t,x))}{\fa(u^N(t,x))}\big(\eta_x(t) - u^N(t,x)\big) dt\Big|\\
 &\ \ \ \ \ \ \ \ \leq CK\int_0^T H(\mu^N_t|\nu^N_t)dt + O(K^2N^{d-\varepsilon_0}).
\end{align*}

Finally, we note, with respect to the third line of \eqref{eq:cor_BG}, that
$$\partial_t\fa(u^N(t,x)) = \fa'(u^N(t,x))\partial_t u^N(t,x).$$

Then, combining these observations, 
$\int_0^T \big( L^{*,\nu^N_t}_N 1 - \partial_t\log \psi^N_t\big) dt$ is approximated in $L^1(\mathbb{P}_N)$ by
\begin{align}\label{cancelling_eqn}
&\int_0^T \sum_{x} \left[ \frac{(\De^N\fa)(u^N(t,x))}{\fa(u^N(t,x))} \fa'(u^N(t,x))\{\eta_x(t) - u^N(t,x)\}   \right.   \\
&\ \ + K\sum_{x} \frac{\fa'(u^N(t,x))}{\fa(u^N(t,x))}E_{\nu_{u^N(t,x)}}\big[c^+(\eta) - c^-(\eta)\big]\{\eta_x(t) - u^N(t,x)\}\nonumber\\
&\ \ \left.  -\sum_{x\in \T^d_N} \frac{\fa'(u^N(t,x))}{\fa(u^N(t,x))} \partial_t u^N(t,x) \{\eta_x(t)-u^N(t,x)\} \right]dt
\notag
\end{align}
with error $CK^{2/\si}\int_0^T H(\mu^N_t|\nu^N_t)dt + O(K^{1+2/\si}N^{d-\varepsilon_0})$.
Since $u^N(t,x)$ satisfies the discretized equation \eqref{eq:HD-discre},
the display \eqref{cancelling_eqn} vanishes.
Hence,
$\int_0^T \big( L^{*,\nu^N_t}_N 1 - \partial_t\log \psi^N_t \big) dt$ is within the $L^1$ error bound desired.
\end{proof}

\subsection{Proof of Theorem \ref{thm:EstHent}}  \label{proof_EstHent_section}

From \eqref{eq:dH} and Theorem \ref{cor_BG}, we have, for $t\in [0.T]$, that
$$
H(\mu^N_t|\nu^N_t) \leq H(\mu^N_0|\nu^N_0) +
C K^{2/\si} \int_0^t H(\mu^N_s| \nu_s^N) ds + O(K^{1+2/\si} N^{d-\varepsilon_0}),
$$
where $\varepsilon_0=2d/(9d+2)$.
Then, by Gronwall's estimate, we obtain, for $t\in [0,T]$, that
\begin{align*}
H(\mu^N_t| \nu^N_t) \leq \left\{ H(\mu^N_0| \nu^N_0) + O(K^{1+2/\si} N^{d-\varepsilon_0})
\right\} \exp\big\{CTK^{2/\si}\big\}.
\end{align*}

Suppose now that
$$K(N)\leq  \delta (\log N)^{\si/2}$$ 
for $\delta>0$ such that 
$CT\delta^{2/\si}< (\e_0\wedge \epsilon)/2$. Since the initial entropy 
$H(\mu^N_0|\nu^N_0)= O(N^{d-\epsilon})$, we will have 
for $t\le T$ that
$$H(\mu^N_t| \nu^N_t) = o(N^{d-(\e_0\wedge \epsilon)/2}).$$
This finishes the proof. \qed

\section{ Interface limit for continuum Allen-Cahn equations with nonlinear diffusion}
\label{sec:cont-AC}

We first discuss a formal derivation of the interface motion in the continuous PDE setting in Section \ref{Formal derivation section}, before stating precise results in Section \ref{main results on pde} found in \cite{EFHPS}.  Then, we turn to outline of proofs of Theorems \ref{Thm_Generation} and \ref{Thm_Propagation} on generation and propagation of the continuous interface motion in Sections \ref{generation_sec} and \ref{propagation_sec}.  Especially, we gather necessary bounds to apply for the discrete PDE in Section \ref{sec:6.3}; see Lemmas \ref{Lem_Generation_Matthieu}, 
\ref{Lem_generation_with_homo_Neumann} and \ref{Lem_Prop_subsuper}.

\subsection{Formal derivation}
\label{Formal derivation section}

We first give, through formal asymptotic expansions,
the derivation of the interface motion equation corresponding to Problem 
\begin{align}
\label{P^e problem}
(P^\varepsilon)~~
	\begin{cases}
	\partial_t u
	= \Delta \fa(u)
	+ \displaystyle{ \frac{1}{\varepsilon^2}} f(u)
	&\mbox{ in } [0,\infty)\times \T^d \\
	u(0,v) = u_0(v)
	&\text{ for } v \in \T^d,
	\end{cases}
\end{align}
where the unknown function $u$ denotes say `mass density', $d\geq 2$,
and $\varepsilon > 0$ is a small parameter.  We remark the parameter $\varepsilon$ can be viewed in terms of $K$, which we use to describe the microscopic Glauber+Zero-range dynamics, as $\varepsilon = K^{-1/2}$ or $\varepsilon^{-2} = K$.  

This equation is determined by the two first terms
of the asymptotic expansion. We refer to \cite{NMHS}, \cite{A},
\cite{AHM} for a similar formal analysis for other equations with
a bistable nonlinear reaction term. Let us also mention
some other papers \cite{ABC}, \cite{Fi} and \cite{RSK} involving
the method of matched asymptotic expansions for related phase
transition problems.

Problem $\Pe$ possesses a unique solution $\ue$. As 
$\e \rightarrow 0$, the qualitative behavior of this solution is the
following. In the very early stage, the nonlinear diffusion term
is negligible compared with the reaction term $\e ^{-2}f(u)$.
Hence, rescaling time by $\tau=t/\e^2$, the equation is well
approximated by the ordinary differential equation $u_\tau=f(u)$ where $u_\tau = \partial_\tau u$.
In view of the bistable nature of $f$, $\ue$ quickly approaches
the values $\alpha_-$ or $\alpha_+$, the stable equilibria of the ordinary
differential equation, and an interface is formed between the
regions $\{\ue\approx \alpha_-\}$ and $\{\ue\approx \alpha_+\}$. Once such an
interface is developed, the nonlinear diffusion term becomes large
near the interface, and comes to balance with the reaction term so
that the interface starts to propagate, on a much slower time
scale.

To study such interfacial behavior, it is useful to consider a
formal asymptotic limit of $\Pe$ as $\e\rightarrow 0$. Then, the
limit solution will be a step function taking the value $\alpha_-$ on one
side of the interface, and $\alpha_+$ on the other side. This sharp
interface, which we will denote by $\Gamma_t$, obeys a certain law
of motion, which is expressed as $(P^0)$ (cf. \eqref{P^0 problem})

It follows from the standard local existence theory for parabolic
equations that Problem $\Pz$ possesses locally in time a unique
smooth solution. In fact, by using an appropriate parametrization,
one can express $\Gamma _t$ as a graph over a $N-1$ manifold
without boundary and transfer the motion equation $\Pz$ into a
parabolic equation on the manifold, at least locally in time.  Let
$0\leq t <T^{max}$, $T^{max} \in (0,\infty]$ be the maximal time
interval for the existence of the solution of $\Pz$ and denote
this solution by $\Gamma=\cup _{0\leq t < T^{max}}
(\{t\}\times\Gamma_t)$. Hereafter, we fix $T$ such that
$0<T<T^{max}$ and work on $[0,T]$. Since $\Gamma _0$ is a
$C^{5+\theta}$ hypersurface, we also see that $\Gamma $ is of
class $C^{ \frac{5+\theta}2, 5+\theta}$. For more details
concerning problems related to $\Pz$, we refer to Chen \cite{Chen-1},
\cite{Chen-2} or Chen and Reitich \cite{ChenR}.

In fact, formal derivation of the interface motion from $(P^\varepsilon)$ is
discussed in a companion paper \cite{EFHPS}, under the Neumann boundary 
condition.  We repeat the argument on $\T^d$ for readers' convenience.  We set
$Q_T:=(0,T) \times \T^d$ 
and, for each $t\in [0,T]$, we denote by $\Omega^{(1)}_t$ the region of one side
of the hypersurface $\Gamma_t$, and by $\Omega^{(2)}_t$ the
region of the other side of $\Gamma_t$. We define a
step function $\tilde u(t,v)$ by
\begin{equation}\label{u}
\tilde u(t,v)=\begin{cases}
\, \alpha_- &\text{in } \Omega^{(1)}_t\vspace{3pt}\\
\, \alpha_+ &\text{in } \Omega^{(2)}_t
\end{cases} \quad\text{for } t\in[0,T]\,,
\end{equation}
which represents the formal asymptotic limit of $\ue$ (or the {\it
sharp interface limit}) as $\e\to 0$.

 More specifically, we define $\Gamma^\varepsilon_t$ using the solution $u^\varepsilon$ of $(P^\varepsilon)$. Denote $\Gamma^\varepsilon_t$ as follows;
$$
	\Gamma^\varepsilon_t 
	:= 
	\{ 
	v \in \T^d : u^\varepsilon(t,v) = \alpha_*	
	\}.
$$
Assume that, for some $T > 0$,  $\Gamma^\varepsilon_t$ is a smooth hypersurface without boundary for each $t \in [0,T], \varepsilon > 0$. Define the signed distance function to $\Gamma^\varepsilon_t$ as follows;
\begin{align*}
	\overline{d}^\varepsilon(t,v)
	:=
	\begin{cases}
		{\rm dist}(v,\Gamma^\varepsilon_t) 
		&
		\text{ for } v \in \overline{D^{\varepsilon,-}_t}
		\\
		- {\rm dist}(v,\Gamma^\varepsilon_t) 
		&
		\text{ for } v \in D^{\varepsilon,+}_t
	\end{cases}
\end{align*}
where $D^{\varepsilon,-}_t$ is the region `enclosed' by $\Gamma^\varepsilon_t$ and 
$D^{\varepsilon,+}_t := \T^d \setminus \{ D^{\varepsilon,-}_t \cup \Gamma^\varepsilon_t \}$.
Note that $\overline{d}^\varepsilon = 0$ on $\Gamma^\varepsilon_t$ and $| \nabla \overline{d}^\varepsilon | = 1$ near $\Gamma_t^\varepsilon$. Suppose further that $\overline{d}^\varepsilon$ is expanded in the form 
$$
	\overline{d}^\varepsilon (t,v)
	= \overline{d}_0(t,v)
	+ \varepsilon \overline{d}_1(t,v)
	+ \varepsilon^2 \overline{d}_2(t,v)
	+ \cdots. 
$$
Define 
\begin{align*}
	& \Gamma_t 
	:= 	\{v \in \T^d : \overline{d}_0(t,v) = 0	\},
	\qquad \; \;
	\Gamma
	:=	\cup_{0 \leq t \leq T}	(	\{ t \} \times \Gamma_t	),
	\\
	& D^-_t 
	:= 	\{v \in \T^d : \overline{d}_0(t,v) > 0	\},
	\qquad
	D^+_t 
	:= 	\{	v \in \T^d : \overline{d}_0(t,v) < 0\}.
\end{align*}
As we will see later, the values of $u^\varepsilon$ are close to $\alpha_\pm$ on the domains $D_t^\pm$, which is consistent with $D_0^\pm$ in (BIP2) and \eqref{cond_u0_inout}.

 Assume that $u^\varepsilon$ has the expansions
\begin{align*}
	u^\varepsilon(t,v)
	= \alpha_\pm + \varepsilon u^\pm_1(t,v) + \varepsilon^2 u^\pm_2(t,v) + \cdots
\end{align*}
away from the interface $\Gamma$ and 
\begin{align}\label{eqn_u^eps_expansion}
	u^\varepsilon(t,v)
	= U_0(t,v,\xi)
	+ \varepsilon U_1(t,v,\xi)
	+ \varepsilon^2 U_2(t,v,\xi)
	+ \cdots
\end{align}
near $\Gamma$, where $\displaystyle{\xi = \frac{\overline{d}_0}{\varepsilon}}$. Here the variable $\xi$ was given to describe the rapid transition between the regions $\{ u^\varepsilon \simeq \alpha_+ \}$ and $ \{ u^\varepsilon \simeq \alpha_- \}$. In addition, we normalize $U_0$ and $U_k$ in a way that
\begin{align}\label{cond_Uk_normal}
	U_0(t,v,0) = \alpha_*, \quad
	U_k(t,v,0) = 0.
\end{align}
	
To match the inner and outer expansions, we require that 
\begin{align}\label{cond_U0_matching}
	U_0(t,v,\pm \infty) = \alpha_\mp, \quad
	U_k(t,v,\pm \infty) = u^\mp_k(t,v)
\end{align}
for all $k \geq 1$. 

After substituting the expansion (\ref{eqn_u^eps_expansion}) into $(P^\varepsilon)$ we consider collecting the $\varepsilon^{-2}$  terms, which yields the following equation
\begin{align*}
	\fa(U_0)_{zz} + f(U_0) = 0.
\end{align*}
Since the equation only depends on the variable $z$, we may assume that $U_0$ is only a function of the variable $z$. Thus we may assume $U_0(t,v,z) = U_0(z)$. In view of the conditions (\ref{cond_Uk_normal}) and (\ref{cond_U0_matching}), we find that $U_0$ is the unique solution of the following problem
\begin{align}\label{eqn_AsymptExp_U0}
	\begin{cases}
	(\fa(U_0))_{zz} + f(U_0) 
	= 0,
	\\
	U_0(-\infty) = \alpha_+, U_0(0)= \alpha_*,  U_0(\infty) = \alpha_-.
	\end{cases}
\end{align}

To understand this more clearly, for $u\geq 0$, we set
\begin{align*}
	b(u) := f(\fa^{-1}(u)),
\end{align*}
where $\fa^{-1}$ is the inverse function of $\fa:\R_+\rightarrow \R_+$ and define $V_0(z) := \fa(U_0(z))$; note that such transformation is possible by the condition (\ref{cond_phi'_bounded}).  The condition (BS) on $f$ implies that $b(u)$ has exactly three zeros $\varphi(\a_-)$, $\varphi(\a_*)$ and $\varphi(\a_+)$ where
$$b'(\varphi(\alpha_-))<0, \ b'(\varphi(\alpha_*))>0, \ \ {\rm and \ \ } b'(\varphi(\alpha_+))<0.$$
  Substituting $V_0$ into equation (\ref{eqn_AsymptExp_U0}) yields
\begin{align}\label{eqn_AsymptExp_V0}
	\begin{cases}
		V_{0zz} + b(V_0) = 0, 
		\\
		V_0(-\infty) = \fa(\alpha_+), V_0(0)= \fa(\alpha_*), 
		V_0(\infty) = \fa(\alpha_-).
	\end{cases}
\end{align}
Condition (\ref{cond_fphi_equipotential}) then implies 
$$\int_{\varphi(\alpha_-)}^{\varphi(\alpha_+)} b(u)du =0,$$
which gives the existence and uniqueness up to translations of the solution of 
(\ref{eqn_AsymptExp_V0}), and especially in our case that the speed of the traveling wave solution $V_0$ vanishes.

Next, we consider the collection of $\varepsilon^{-1}$ terms in the asymptotic expansion. In view of the definition of $U_0(z)$ and the condition (\ref{cond_Uk_normal}), for each $(t,v)$, this yields the following problem
\begin{align}\label{eqn_AsymptExp_U1}
	\begin{cases}
	(\fa'(U_0) \overline{U_1})_{zz} + f'(U_0)\overline{U_1} 
	= U_{0z} \partial_t\overline{d}_0   - (\fa(U_0))_z \Delta \overline{d}_0,
	\\
	\overline{U_1}(t,v,0) = 0, ~~~ \fa'(U_0) \overline{U_1} \in L^\infty(\mathbb{R}).
	\end{cases}
\end{align}
To see the existence of the solution of (\ref{eqn_AsymptExp_U1}) we perform the change of unknown function $\overline{V_1} = \fa'(U_0)\overline{U_1}$, which yields the problem
\begin{align}\label{eqn_AsymptExp_V1}
	\begin{cases}
		\overline{V_{1}}_{zz} + b'(V_0)\overline{V_1} 
		= 
		\displaystyle{\frac{V_{0z}}{\fa'(\fa^{-1} (V_0) )}} \partial_t\overline{d}_0
		-  
		V_{0z}	\Delta \overline{d}_0,
		\\
		\overline{V_1}(t,v,0) = 0, ~~~
		\overline{V_1} \in L^\infty(\mathbb{R}).
	\end{cases}.
\end{align}
Lemma 2.2 of \cite{AHM} implies the existence of $V_1$ provided that  
\begin{align*}
	\int_\R 
	\left(
		\frac{1}{\fa'(\fa^{-1}(V_0))}
		\partial_t\overline{d}_0
		- \Delta \overline{d}_0
	\right)
	V_{0z}^2dz
	= 0.
\end{align*}
Substituting $V_0 = \fa(U_0)$ and $ V_{0z} = \fa'(U_0) U_{0z} $ in the above equation yields 
\begin{align}
	\partial_t\overline{d}_0
	= \frac
	{\int_\R V_{0z}^2dz}
	{\int_\R \frac{V_{0z}^2}{\fa'(\fa^{-1}(V_0)}dz}
	\Delta \overline{d}_0
	= \frac
	{\int_\R (\fa'(U_0) U_{0z})^2dz}{\int_\R \fa'(U_0) U_{0z}^2dz}
	\Delta \overline{d}_0.
\end{align}
It is well known that $\partial_t\overline{d}_0$ is equal to the normal velocity $V$ of the interface $\Gamma_t$, and $\Delta \overline{d}_0$ is equal to $\kappa$ where $\kappa$ is the mean curvature of $\Gamma_t$ multiplied by $d-1$. Thus, we obtain the interface motion equation on $\Gamma_t$:
\begin{align*}
	V = \lambda_0 \kappa,
\end{align*}
where 
\begin{align}\label{eqn_lambda0}
	\lambda_0 
	=
	\frac
	{\int_\R (\fa'(U_0) U_{0z})^2dz}{\int_\R \fa'(U_0) U_{0z}^2dz}.
\end{align}
This speed $\lambda_0$ is interpreted as the `surface tension' multiplied by the `mobility' of the interface; 
see  Appendix of \cite{EFHPS} and also \cite{S93}.
The constant $\lambda_0$ has another explicit form \eqref{def-lambdazero-intrinsic-intro}.
Its derivation is given in the last part of Section 2 of \cite{EFHPS}.

\subsection{Results on Allen-Cahn equation with nonlinear diffusion}
\label{main results on pde} 

Here we briefly summarize the results obtained in \cite{EFHPS} on generation
and propagation
of interface properties for an Allen-Cahn equation $(P^\varepsilon)$
with nonlinear diffusion,
and state estimates on sub and super solutions necessary to study discrete
Allen-Cahn equation in Section \ref{sec:6.3}.

The nonlinear functions $\fa$ and $f$ satisfy the following properties:  
In line with the previous specification of the microscopic dynamics, we assume (minimally) that $f \in C^2(\mathbb{R}_+)$ has exactly three zeros $f(\alpha_-) = f(\alpha_+) = f(\alpha_*) = 0$, where $\R_+ = [0,\infty)$, $0<\alpha_- < \alpha_* < \alpha_+$, and 
\begin{align}\label{cond_f_bistable}
	f'(\alpha_-) < 0, f'(\alpha_+) < 0, f'(\alpha_*) > 0.
\end{align}
Also, $f(0)>0$ so that the later evolution starting positive stays positive.

In addition, we assume that $\fa \in C^4(\mathbb{R}_+)$ and
\begin{align}\label{cond_phi'_bounded}
	\fa'(u) \geq C(\fa, u_-, u_+) \ \ {\rm for \ \ } u_-\leq u\leq u_+ 
\end{align}
for some positive constant $C(\fa, u_-, u_+)$. We give one more assumption on $f$ and $\fa$, namely  
\begin{align}\label{cond_fphi_equipotential}
	\int_{\alpha_-}^{\alpha_+}
	\fa'(s) f(s) ds
	= 0.
\end{align}

We note in the particle system context that $\fa,f\in C^\infty(\R_+)$ and $\fa'(u)>0$ for $u>0$, and so $\fa'(u)$ is bounded away from $0$ and $\infty$ for $u\in [u_-, u_+]$.

As for the initial condition $u_0$, following (BIP1) and (BIP2), we assume $u_0 \in C^5(\T^d)$ 
and $0<u_-\leq u_0\leq u_+$.  As a consequence, $u(t,\cdot)$ is also bounded between $u_-$ and $u_+$.   We define $C_0$ as follows,
\begin{align}\label{cond_C0}
	C_0 
	:= \| u_0 \| _{C^0 \left( \T^d \right)}
	+ \| \nabla u_0 \| _{C^0 \left( \T^d \right)}
	+ \| \Delta u_0 \| _{C^0 \left( \T^d \right)}.
\end{align} 
Furthermore we define $\Gamma_0$ by
\begin{align}
	\Gamma_0 
	:= 
	\{
		v \in \T^d: u_0(v) = \alpha_*
	\}.
\end{align}
In addition, recalling assumption (BIP2), we suppose $\Gamma_0$  is a $C^{5+\theta}, 0 < \theta < 1$, hypersurface without boundary such that 
\begin{align}
\nabla u_0(v) \cdot n(v) \neq 0 \text{ if}~ v \in \Gamma_0 \label{cond_gamma0_normal} \\
	u_0 > \alpha_* \text{ in } D_0^+, ~~~~~~ u_0 < \alpha_* \text{ in } D_0^- \label{cond_u0_inout}
\end{align}
where $D_0^\pm$ denote the regions separated by $\Gamma_0$ 
and $n$ is the outward normal vector to $D_0^+$.
It is standard that Problem $(P^\varepsilon)
$ possesses a unique classical solution $u^\varepsilon$. 

The goal is to study the singular limit of $u^\varepsilon$ as 
$\varepsilon \downarrow 0$.   We first present the generation of interface result
(cf.\ \cite{EFHPS}, Theorem 1.2).   We will use below the following notation:
\begin{align}\label{cond_mu_eta0}
	\gamma = f'(\alpha_*)
	, ~~ 
	t^\varepsilon = \gamma^{-1} \varepsilon^2 |\log \varepsilon|
	, ~~
	\delta_0 := \min(\alpha_* - \alpha_-, \alpha_+ - \alpha_*).
\end{align}

\begin{thm}\label{Thm_Generation}
Let $u^\varepsilon$ be the solution of the problem $(P^\varepsilon)$, $\delta$ be an arbitrary constant satisfying $0 < \delta < \delta_0$. Then, there exist positive constants $\varepsilon_0$ and $M_0$ such that, for all $\varepsilon \in (0, \varepsilon_0)$, we have the following: 
\begin{enumerate}
\item For all $v \in \T^d$,
\begin{align}\label{Thm_generation_i}
	\alpha_- - \delta
	\leq
	u^\varepsilon(t^\varepsilon,v)
	\leq
	\alpha_+ + \delta.
\end{align}

\item If $u_0(v) \geq \alpha_* +  M_0 \varepsilon$, then
\begin{align}\label{Thm_generation_ii}
	u^\varepsilon(t^\varepsilon,v) \geq \alpha_+ - \delta.
\end{align}

\item If $u_0(v) \leq \alpha_*  - M_0 \varepsilon$, then
\begin{align}\label{Thm_generation_iii}
	u^\varepsilon(t^\varepsilon,v) \leq \alpha_- + \delta.
\end{align}

\end{enumerate}

\end{thm}

To understand more this statement, we remark that the assumption \eqref{cond_gamma0_normal} implies that $u_0(v)$ is away from $\alpha_*$ when $v$ is away from $\Gamma_0$.

After the interface has been generated, the diffusion term has the same order as the reaction term. As a result the interface starts to propagate slowly. Later we will prove that the interface moves according to the motion equation $(P^0)$ (cf. \eqref{P^0 problem}).
 It is well known that Problem $(P^0)$ possesses locally in time a unique smooth solution. Let $T>0$ be the maximal time interval for the existence of the smooth solution of $(P^0)$ and denote this solution by $\Gamma = \cup_{0\leq t < T} (\{t\}\times\Gamma_t )$. Moreover we deduce from \cite{Chen-1} that the regularity of the interface exactly follows the regularity of the initial interface, so that $\Gamma \in C^{\frac{5 + \theta}{2}, 5+\theta}$.

Let $D^{+}_t$ denote the region `enclosed' by the interface $\Gamma_t$, continuously determined from $D_0^+$, and set $D^-_t := \T^d \setminus\overline{D^+_t}$. Let $\overline{d}(t,v)$ be the signed distance function to $\Gamma_t$ defined by 
\begin{align*}
	\overline{d}(t,v)
	:=
	\begin{cases}
		{\rm dist}(v, \Gamma_t)
		& \text{ for } v \in \overline{D^-_t}
		\\
		- {\rm dist}(v, \Gamma_t)
		& \text{ for } v \in D^+_t.
	\end{cases}
\end{align*}
The second is the propagation of the interface (cf. \cite{EFHPS}, Theorem 1.3).

\begin{thm}\label{Thm_Propagation}
	Under the conditions given in Theorem \ref{Thm_Generation} and those mentioned above, for any given $0 < \delta < \delta_0$ there exist $\varepsilon_0 > 0$ and $C > 0$ such that 
\begin{align}
	u^\varepsilon(t,v)
	\in
	\begin{cases}
		[\alpha_- - \delta, \alpha_+ + \delta]
		&
		\text{ for } v \in \T^d
		\\
		[\alpha_+ - \delta, \alpha_+ + \delta]
		&
		\text{ if } \overline{d}(t,v) \leq - \varepsilon C
		\\
		[\alpha_- - \delta, \alpha_- + \delta]
		&
		\text{ if } \overline{d}(t,v) \geq   \varepsilon C
	\end{cases}
\end{align}
for all $\varepsilon \in (0, \varepsilon_0)$ and for all $t \in (t^\varepsilon,T]$.
\end{thm}

\subsection{Generation of the interface: Outline of proof of Theorem \ref{Thm_Generation}}
\label{generation_sec}

The main idea of the proof is based on the comparison principle. Thus, we need to construct appropriate sub and super solutions for the problem $(P^\varepsilon)$. In this first stage, we expect that the solution behaves as that of the corresponding ordinary differential equation and we construct sub and super solutions as solutions of the following initial value problem ordinary differential equation;
\begin{align}  \label{eq:7.1-E}
	\begin{cases}
		\partial_\tau Y(\tau, \zeta) = f(Y(\tau,\zeta)), & \tau > 0,\\
		Y(0,\zeta) = \zeta, & \zeta \in \R_+.
	\end{cases} 
\end{align}

Recall $C_0$ defined in \eqref{cond_C0},
$\gamma = f'(\alpha_*), t^\epsilon, \delta_0$ defined in \eqref{cond_mu_eta0}, and set 
$$-\bar\gamma = \min_{\zeta\in 
[u_- \wedge \a_-, u_+\vee\a_+]}f'(\zeta);$$ 
note that $\gamma, \bar\gamma>0$.  The following bounds on $Y(\t,\zeta)$ are
used for the proofs of Lemma \ref{Lem_generation_with_homo_Neumann}
and also Theorem \ref{thm:6.3} below.

\begin{lem}\label{Lem_Generation_Matthieu}
	Let $\delta \in (0, \delta_0)$ be arbitrary. 
	
\begin{enumerate}
\item There exists a constant $C_1 = C_1(\delta)>0$ such that 
	$$
		0<e^{-\bar\gamma \tau}< Y_\zeta(\tau,\zeta) \leq C_1 e^{\gamma \tau}
	$$
for all $\zeta\in [u_-,u_+]$ and $\t\geq0$.
\item There exists a constant $C_2 = C_2(\delta)>0$ such that, for all $\tau > 0$ and all $\zeta \in (0, 2C_0)$,
$$
	\left|
	\frac{Y_{\zeta \zeta}(\tau, \zeta)}{Y_\zeta(\tau, \zeta)} 
	\right|
	\leq C_2 (e^{\gamma \tau} - 1), \ \ \   |Y_{\zeta\zeta}(\tau, \zeta)|\leq C_2(e^{\gamma\tau}-1)e^{\gamma\tau}, \ \ {\rm and }
$$
\begin{align}\label{eqn_Y'''_bound}
	|Y_{\zeta\zeta\zeta}(\tau, \zeta)| \leq 2 C_2 (e^{2\gamma\tau}-1)e^{\gamma\tau}.
\end{align}

\item There exist constants $\varepsilon_0, C_3>0$ such that for all $\varepsilon \in (0, \varepsilon_0)$:
\begin{enumerate}
\item For all $\zeta \in (0, 2C_0)$, in terms of a constant $C_0>0$,
\begin{align}\label{Lem_Generation_i}
	\alpha_- - \delta
	\leq
	Y(\gamma^{-1} |\log \varepsilon|, \zeta)
	\leq
	\alpha_+ + \delta.
\end{align}

\item If $\zeta \geq \alpha_* +  C_3 \varepsilon$, then
\begin{align}\label{Lem_Generation_ii}
	Y(\gamma^{-1} |\log \varepsilon|, \zeta) \geq \alpha_+ - \delta.
\end{align}

\item If $\zeta \leq \alpha_* - C_3 \varepsilon$, then
\begin{align}\label{Lem_Generation_iii}
	Y(\gamma^{-1} |\log \varepsilon|, \zeta) \leq \alpha_- + \delta.
\end{align}
\end{enumerate}

\end{enumerate}	
	
\end{lem}

\begin{proof}
We refer to \cite{AHM} and \cite{EFHPS}, Lemma 2 for the proof
except that of \eqref{eqn_Y'''_bound}.
To show \eqref{eqn_Y'''_bound},
we use
\begin{align*}
& Y_{\zeta\zeta}(\tau,\zeta) = A(\tau,\zeta)Y_\zeta(\tau,\zeta),
\quad  A(\tau,\zeta) = \int_0^\tau f''(Y(r,\zeta)) Y_\zeta(r,\zeta) dr,  \\
& |A(\tau,\zeta)| \le C_A(e^{\gamma \tau}-1),
\end{align*}
given in Lemmas 3.3 and 3.4 of \cite{AHM} where $C_A>0$ is some constant.  Indeed, we have
$$
Y_{\zeta\zeta\zeta}(\tau,\zeta) = A_\zeta(\tau,\zeta)Y_\zeta(\tau,\zeta)
+ A(\tau,\zeta)Y_{\zeta\zeta}(\tau,\zeta).
$$
Thus, there exists $C' > 0$ such that $A_\zeta$ in the first term can be estimated as
\begin{align*}
|A_\zeta(\tau,\zeta)| & = \left| \int_0^\tau \left\{f'''(Y(r,\zeta)) Y_\zeta^2(r,\zeta) 
+ f''(Y(r,\zeta)) Y_{\zeta\zeta}(r,\zeta) \right\} dr\right|  \\
& \le C'\int_0^\tau e^{2\gamma r}dr \le C'(e^{2\gamma \tau}-1).
\end{align*}
Thus, by choosing $C_2$ bigger if necessary, we obtain
\begin{align*}
|Y_{\zeta\zeta\zeta}(\tau,\zeta)| \le C_2(e^{2\gamma \tau}-1) e^{\gamma \tau}
+ C_2(e^{\gamma \tau}-1)^2 e^{\gamma \tau} \le 2 C_2(e^{2\gamma \tau}-1) e^{\gamma \tau}.
\end{align*}
\end{proof}

Define sub and super solutions on $\T^d$ for the proof of Theorem 
\ref{Thm_Generation} as follows
\begin{equation}  \label{eq:7.2-E}
	w^{\pm}_\varepsilon(t,v)
	= Y 
	\left(
		\frac{t}{\varepsilon^2},
		u_0(v) 
		\pm P(t)
		\right),
\end{equation}
where
$$P(t) = \varepsilon^2 C_4 
		\left( 
			e^{\gamma t/\varepsilon^2} - 1\right),$$
			for some constant $C_4>0$.
Note that $P(t)\leq \varepsilon^2C_4(\varepsilon^{-1}-1)\leq \varepsilon C_4$ for $t\leq t^\varepsilon$, where $t^\varepsilon$ is defined in \eqref{cond_mu_eta0}.  In particular, since $u_0(v)\geq u_->0$, we have $u_0(v)-P(t)>0$ for sufficiently small $\varepsilon>0$.  Given that we work on the torus $\T^d$, or on $\R^d$ with periodic $u_0$, the constructed sub and super solutions $w_\varepsilon^\pm (t,v)$ are periodic for all $t\in [0,t^\varepsilon]$.

Denote also the operator $\mathcal{L}$ by
$$
	\mathcal{L} u = \partial_tu - \Delta \fa(u) - \frac{1}{\varepsilon^2} f(u).
$$
We set also, noting $\varphi(u), \varphi'(u)>0$,
$$
	C_\varphi := 
	\max \fa(u) +
	\max  \fa'(u) +
	\max |\fa''(u)|,
$$
where `$\max$' is maximum over $u\in [0,(2C_0)\vee \alpha_+]$.
Then, we have the following bounds; see \cite{EFHPS}, Lemma 3.

\begin{lem}\label{Lem_generation_with_homo_Neumann}
There exist constants $\varepsilon_0, C_4>0$ such that, for all $\varepsilon \in (0,\varepsilon_0)$, $w^{\pm}_\varepsilon$ is a pair of sub and super solutions of $(P^\varepsilon)$ in the domain $[0, t^\varepsilon]\times \T^d$. 

In particular, in terms of a constant $C_5>0$, we have
\begin{equation}
\label{eq:7.4}
\mathcal{L} w^+_\varepsilon \geq C_5e^{-\bar\gamma \tau/\e^2} \ \ 
{\rm and }\ \ \mathcal{L} w^-_\varepsilon \leq -C_5 e^{-\bar{\gamma}\tau/\e^2}, 
\ (\tau,v)\in [0,t^\varepsilon]\times \T^d.
\end{equation}
\end{lem}

\begin{rem}
\rm 
It follows from $\mathcal{L}w_\varepsilon^-\leq 0\leq \mathcal{L}w_\varepsilon^+$ that $w_\varepsilon^\pm$ are sub and super solutions.  However, the stronger estimate \eqref{eq:7.4} will be useful in the proof of Theorem \ref{thm:6.3} in the discrete setting.
\end{rem}

\subsection{Propagation of the interface: Outline of proof of Theorem
 \ref{Thm_Propagation}}
\label{propagation_sec}

	 We now argue the propagation of the interface given in Theorem \ref{Thm_Propagation}.  Again, we will need to construct appropriate sub and super solutions, but now in terms of functions $U_0$ in \eqref{eqn_AsymptExp_U0} and a $U_1$ similar to that in \eqref{eqn_AsymptExp_U1}.

We first introduce a cut-off signed distance function $d=d(t,v)$ as follows. Choose $d_0 > 0$ small enough so that the signed distance function $\overline{d}=\overline{d}(t,v)$ from the interface $\Gamma_t$ evolving under $(P^0)$ is smooth in the set  
$$
\{
(t,v) \in [0,T] \times \T^d, | \overline{d}(t,v) | < 3 d_0
\}.
$$
Let $h(s)$ be a smooth non-decreasing function on $\mathbb{R}$ such that 
$$
	h(s) = 
	\begin{cases}
		s & \text{if}~ |s| \leq d_0\\
		-2d_0 & \text{if}~ s \leq -2d_0\\
		2d_0 & \text{if}~ s \geq 2d_0.
	\end{cases}
$$
We then define the cut-off signed distance function $d$ by 
$$
	d(t,v) = h(\overline{d}(t,v)), ~~~ (t,v) \in [0,T]\times\T^d.
$$
Note that, as $d$ coincides with $\overline{d}$ in the region
$$
	\{ 
	(t,v) \in [0,T]\times\T^d : | d(t,v)| < d_0
	\},
$$
we have 
$$
	\partial_t d = \lambda_0 \Delta d \text{ on } \Gamma_t.
$$
Moreover, $d$ is constant far away from $\Gamma_t$.

In terms of this function $d$, we now define $U_1 : [0,T] \times \T^d \times \R \to \R$ satisfying the following problem
\begin{align*}
	\begin{cases}
		(\fa'(U_0) U_1)_{zz} + f'(U_0)U_1 
		= (\lambda_0 U_{0z} - (\fa(U_0))_z) \Delta d(t,v)\\
		U_1(t,v,0) = 0, ~~~
		\fa'(U_0) U_1(t,v) \in L^\infty(\R)
	\end{cases}
\end{align*}
where $U_0$ is the solution of \eqref{eqn_AsymptExp_U0}. Since $d \in C^{ \frac{{5}+\theta}2, {5}+\theta}([0,T] \times \T^d)$, we have that $\Delta d \in C^{ \frac{{3}+\theta}2, {3}+\theta}([0,T] \times \T^d)$. As a consequence, we have  $U_1(\cdot,\cdot, z) \in C^{ \frac{{3}+\theta}2, {3}+\theta}([0,T] \times \T^d)$ 
for each $z \in \R$. Moreover, $U_1(t,v,\cdot) \in C^3(\R)$ for each $(t,v) \in [0,T] \times \T^d$
by a similar argument given in the proof of Lemma 6 of \cite{EFHPS}.

We construct the sub and super solutions as follows:  Given $0<\varepsilon <1$, we define  
\begin{equation}  \label{eq:8.3-E}
	u^\pm(t,v) \equiv u_\e^\pm(t,v)
	= U_0
	\left(
		\frac{d(t,v) \pm \varepsilon p(t)}{\varepsilon}
	\right)
	+ \varepsilon U_1 
	\left(t,v,
		\frac{d(t,v) \pm \varepsilon p(t)}{\varepsilon}
	\right)
	\pm q(t),
\end{equation}
where 
\begin{align*}
	&p(t) = e^{- \beta t/ \varepsilon^2} - e^{Lt} - \hat{L},\\
	&q(t) = \tilde\sigma
	\left(
	\beta e^{- \beta t / \varepsilon^2} + \varepsilon^2 L e^{Lt}
	\right).
\end{align*}
Here $\beta, \tilde\sigma, L, \hat{L}>0$ are constants determined by Lemma
\ref{Lem_Prop_subsuper} below.
Although we work on $\T^d$, if we take the viewpoint of working on $\R^d$, we may regard the signed distance function $d$ as periodic with period $1$ so that $u^\pm(t,v)$ are periodic as well for all $t \in [0,T]$.
Then, we have the following bounds; see \cite{EFHPS}, Lemma 10 and Section 4.4.

\begin{lem}\label{Lem_Prop_subsuper}
	One can choose  $\beta, \tilde\sigma>0$
such that, for each $\hat{L} > 1$ there exist $L > 0$ large enough and $\varepsilon_0 > 0$ small enough such that for a constant $C>0$ we have
\begin{align}\label{eqn_Prop_subsuper}
		\mathcal{L} u^- \leq -C<C \leq \mathcal{L}u^+
		\text{ in } [0,T]\times\T^d
\end{align}
for every $\varepsilon \in (0, \varepsilon_0)$, and
\begin{align*}
	u^-(0,v) 
	\leq 
	u^\e (t^\e,v)
	\leq 
	u^+(0,v)
\end{align*}
holds. 
Hence, $u^\pm(t - t^\e,v)$ are sub and super solutions for Problem $(P^\varepsilon)$ for $t \in [t^\e,T]$.
\end{lem}

\begin{rem}\rm
To show that $u^\pm$ are sub and super solutions, it would have been enough in the above proof to show that
$\mathcal{L}u^-\leq 0\leq \mathcal{L}u^+$.  However, the stronger estimate \eqref{eqn_Prop_subsuper}
found will be useful in the proof of Theorem \ref{thm:3.4}
in the discrete setting.
\end{rem}

\section{Generation and propagation of the interface for the `discrete PDE': Proof of Theorem \ref{PDEthm}}  \label{sec:6.3}

Recall that the initial data $\{u^N(0,x)\}_{x\in \T_N^d}$ of the discrete PDE \eqref{eq:HD-discre} satisfy (BIP1) and (BIP2).
Previously, in \eqref{eq:7.2-E} and \eqref{eq:8.3-E}, we have constructed 
super and sub solutions 
$$w_\e^\pm(t,v) \equiv w_K^\pm(t,v) \ \ {\rm and \ \ }
 u_\e^\pm(t,v) \equiv u_K^\pm(t,v), t\ge 0, v\in \T^d,$$
of the problem $(P^\e)$ with $\e=K^{-1/2}$.

We will show that these functions, $w_K^\pm(t,v)$ and $u_K^\pm(t,v)$,
restricted to the discrete torus $\frac1N \T_N^d$  actually play
the role of super and sub solutions of the discretized hydrodynamic
equation \eqref{eq:HD-discre}.  As we noted, we abuse notations $\frac{x}N$ and $x$ 
for the discrete spatial variables.  The proof relies on the comparison argument.

More precisely, we show  
$$\mathcal{L}^{N,K}w_K^+ \ge 0 \ge 
\mathcal{L}^{N,K}w_K^- \ \ {\rm and  \ \ }\mathcal{L}^{N,K}u_K^+\ge 0\ge
\mathcal{L}^{N,K}u_K^-,$$
 where $\mathcal{L}^{N,K}$ is the 
operator associated with \eqref{eq:HD-discre}.  These estimates will follow
from estimates shown in the continuum setting, namely $\mathcal{L} w_\e^+ 
\ge C_5 e^{-\bar\gamma \t/\e^2} >
- C_5 e^{-\bar\gamma \t/\e^2} \ge \mathcal{L} w_\e^-$ (cf. \eqref{eq:7.4}), and 
$\mathcal{L} u_\e^+ \ge C > - C\ge \mathcal{L} u_\e^-$ (cf. \eqref{eqn_Prop_subsuper}),
in combination with the error estimates
on $(\mathcal{L} - \mathcal{L}^{N,K})w_K^\pm$ and 
$(\mathcal{L} - \mathcal{L}^{N,K})u_K^\pm$.

\subsection{
Generation of a discrete interface}

Recall $Y(\t)=Y(\t,\zeta)$ for $\t\ge0, \zeta\in \R_+$, is the solution of the ordinary differential
equation \eqref{eq:7.1-E}, with the initial value $Y(0)=\zeta$.

\begin{thm}  \label{thm:6.3}
Let $u^N(t,\cdot)$ be the solution of the discrete PDE
\eqref{eq:HD-discre} with 
initial value $u^N(0,\cdot)$.  Let also $\de\in (0,\de_0)$ where
$\de_0 = \min\{\a_*-\alpha_-,\alpha_+-\a_*\}$, 
and $t^N =
\tfrac1{2\gamma K} \log K$.  Suppose that $K\equiv K(N)=o(N^{2\gamma/(3\gamma+\bar\gamma)})$.
Then, there exist $N_0, M_0>0$ such that the following hold for
every $N \ge N_0$: 

\noindent (1) For all $x\in \T_N^d$,
$$
\alpha_--\de\le u^N(t^N,x) \le \alpha_++\de.
$$
(2) If $u_0(\tfrac{x}N) \ge \a_*+M_0K^{-1/2}$, then
$$
u^N(t^N,x) \ge \alpha_+-\de.
$$
(3) If $u_0(\tfrac{x}N) \le \a_*-M_0K^{-1/2}$, then
$$
u^N(t^N,x) \le \alpha_-+\de.
$$
\end{thm}

\begin{proof}
Using $Y(\t,\zeta)$ and $u_0=u_0(x)$, we define sub and
super solutions of the continuous system as
$$
w_K^\pm(t,v) = Y(Kt, u_0(v)\pm P(t)), \quad v\in \T^d,
$$
where $P(t) = C_4(e^{K\gamma t}-1)/K$.  
Define the operators
$\mathcal{L}^K$ and $\mathcal{L}^{N,K}$ by
$$
\mathcal{L}^K u = \partial_t u - \De \fa(u) -K f(u), \quad v \in \T^d,
$$
with respect to the continuous Laplacian $\De$ on $\T^d$ and also continuous functions $u=\{u(t,v)\}_{v\in \T^d}$, and
$$
\mathcal{L}^{N,K} u = \partial_t u - \De^N \fa(u) -K f(u), \quad x \in \T_N^d,
$$
for discrete functions $u=\{u(t,x)\}_{x\in \T_N^d}$, respectively.

We now make use of an estimate in the proof of Theorem \ref{Thm_Generation}:  In Lemma \ref{Lem_generation_with_homo_Neumann}, it is shown that
$$
\mathcal{L}^K w_K^+ \ge C_5 e^{-\bar\gamma Kt^N} = C_5K^{-\bar\gamma/2\gamma}>0
$$
holds for some $C_5>0$ and large enough $K$; note that $t^N=t^\e=\ga^{-1}\e^2 |\log \e|$
and $K=\e^{-2}$.  However,
$$
\mathcal{L}^{N,K} w_K^+  = \mathcal{L}^K w_K^+ + (\De\fa(w_K^+)
- \De^N \fa(w_K^+)),
$$
and, by Taylor's formula, the second term is bounded by
$$
\tfrac{C_2}N \sup_{v\in \T^d} \left| D_v^3 \{ \fa(w_K^+(t,v))\}\right|,
$$
where $|D^3_v \{\, \cdot\,\}|$ means the sum of the absolute values of all
third derivatives in $v$.

Since $u_0\in C^3(\T^d)$ and $\fa\in C^3(\R_+)$ (note that
$w_K^\pm$ takes only bounded values so that $\fa\in C_b^3([0,M])$),
from (1)-(3) especially \eqref{eqn_Y'''_bound} of Lemma \ref{Lem_Generation_Matthieu} 
and noting $e^{3\gamma Kt^N} = K^{3/2}$, we obtain
$$
\sup_{0\le t \le t^N, x\in \T_N^d}
|\De\fa(w_K^+(t,\tfrac{x}N))- \De^N \fa(w_K^+(t,\tfrac{x}N))|
\le \tfrac{C_3}N K^{3/2}.
$$
Thus, this term is absorbed by $C_5K^{-\bar\gamma/2\gamma}$ if $K=o(N^{2\gamma/(3\gamma+\bar\gamma)})$
and $N$ is large enough.

Therefore, we obtain $\mathcal{L}^{N,K} w_K^+ \ge 0$ for $N\ge N_0$ with some $N_0>0$.
By Lemma \ref{lem:4.1}, we see $u^N(t,x) \le w_K^+(t,\tfrac{x}N)$.
Similarly, one can show $w_K^-(t,\frac{x}N)\leq u^N(t,x)$.  Thus, the proof of the theorem is concluded similarly to the proof of Theorem \ref{Thm_Generation}; see \cite{EFHPS}.
\end{proof}

\subsection
{Propagation of a discrete interface}
\label{propagation_subsection}

Recall the interface flow $\Gamma_t$, and the two functions
$u^{\pm}(t,v) \equiv u^{\pm}_K(t,v)$ defined by \eqref{eq:8.3-E}, namely
$$
u^\pm(t,v) = U_0\left( K^{1/2}d(t,v) \pm p(t)\right)
+ K^{-1/2} U_1\left(t, v, K^{1/2}d(t,v) \pm p(t)\right) \pm q(t),
$$
and $u^N(t,v)$ defined in \eqref{u^N(t,v)} from the discretized hydrodynamic equation
\eqref{eq:HD-discre}.

\begin{thm}  \label{thm:3.4}
Assume that the following inequality \eqref{eq:3.3-com}
holds at $t=0$ and $K =o(N^{2/3})$ for $K=K(N) \uparrow \infty$.
Then, taking $\b, \tilde\si, L, \hat{L} >0$ in $p(t)$ and $q(t)$ as in Lemma \ref{Lem_Prop_subsuper},
there exists $N_0\in \N$ such that
\begin{equation}  \label{eq:3.3-com}
u^-(t,v) \le u^N(t+t^N,v) \le u^+(t,v),
\end{equation}
holds for every $t\in [0,T-t^N]$, $v= x/N, x \in \T_N^d$ and $N\ge N_0$.
\end{thm}

\begin{proof}
The upper bound in \eqref{eq:3.3-com} follows from Lemma \ref{lem:4.1}, 
once we can show that
\begin{align}  \label{eq:7.3.1}
\mathcal{L}^{N,K} u^+ = \partial_t u^+ - \De^N \fa(u^+) - K f(u^+) \ge 0, \quad x \in \T_N^d,
\end{align}
for every $N\ge N_0$ with some $N_0\in \N$.  As in the proof of Theorem \ref{thm:6.3},
we decompose
\begin{align}  \label{eq:7.3.2}
 \mathcal{L}^{N,K} u^+ = \mathcal{L}^K u^+ 
+( \De \fa(u^+) - \De^N\fa(u^+)),
\end{align}
where 
$\mathcal{L}^K u^+ = \partial_t u^+ -\De\fa(u^+) -K f(u^+)$.

We now make use of an estimate derived in the proof of Theorem \ref{Thm_Propagation}:  
By Lemma \ref{Lem_Prop_subsuper}, the first term $\mathcal{L}^Ku^+$ in \eqref{eq:7.3.2} 
is bounded on $[0,T]\times\T^d$ as
\begin{align}  \label{eq:7.3.4}
\mathcal{L}^Ku^+ \ge C>0,
\end{align}
if we choose parameters 
$\b, \tilde\si, L, \hat{L}>0$ there properly.

For the second term in \eqref{eq:7.3.2}, since $u^+ \in C^{ \frac{{3}+\theta}2, {3}+\theta}$ by the regularity of $d, U_0$ and $U_1$ (cf. discussion above \eqref{eq:8.3-E}), and also
$$
\sup_{t\in [0,T], v \in \T^d} \big| (\nabla_v)^i u^+(t,v)\big| \le CK^{i/2}, \quad i=1,2,3,
$$
we have
\begin{align*}
\left|\De\fa(u^+(t,\tfrac{x}N)) - \De^N \fa(u^+(t,\tfrac{x}N))\right|
\le C_1 \tfrac{K^{3/2}}N.
\end{align*}
Indeed, this follows from Taylor expansion for $\De^N \fa(u^+)$ up to the
third order term, noting that $\fa\in C^3(\R_+)$ and $u^+(t,v)$ is bounded.
Therefore, if $K=o(N^{2/3})$, this term is absorbed by the positive constant $C$
in \eqref{eq:7.3.4} for $\mathcal{L}^Ku^+$.  
This proves \eqref{eq:7.3.1}.

The lower bound by $u^-(t,v)$ is shown similarly.
\end{proof}

\subsection{Proof of Theorem \ref{PDEthm}} 
\label{sec:6.3_subsec}

The proof of Theorem \ref{PDEthm} follows from Theorems \ref{thm:6.3} and \ref{thm:3.4}.  By the assumption (BIP2), $\nabla u_0(v)\cdot n(v) \neq 0$ for $v\in \Gamma_0$.  Hence, for $v\not\in \Gamma_0$, we have that $u_0(v) \neq \alpha_*$.  Then, for $N$ large enough, we would have
$|u_0(v) - \alpha_*|\geq \epsilon_v> M_0K^{-1/2}$, where $M_0$ is the constant in Theorem \ref{thm:6.3}.  

Recall $u^N(t,v)$ in \eqref{u^N(t,v)}.  By Theorem \ref{thm:6.3}, at time $t^N = (2\gamma K)^{-1}\log K$, either $u^N(t^N,v)\geq \alpha_+ - \delta$ or $u^N(t,v)\leq \alpha_-+\delta$ for a small $\delta>0$.    

Since for large $N$, we have $u^-(0,v)\leq u^N(t^N,v)\leq u^+(0,v)$, thinking of $u^N(t^N, \cdot)$ as an initial condition, by Theorem \ref{thm:3.4}, we can `propagate' and obtain 
$u^-(t-t^N, v)\leq u^N(t,v)\leq u^+(t-t^N,v)$ for $t^N\leq t\leq T$.  As $N\uparrow\infty$, we obtain, for each $0<t\leq T$ and $v\not\in \Gamma_t$ that $u^N(t,v)\rightarrow \chi_{\Gamma_t}(v)$, concluding the proof.
\hfill \qed

\section{A `Boltzmann-Gibbs' principle: Proof of Theorem \ref{BG}}  \label{BG_section}

We give now an outline of the proof of Theorem \ref{BG}, referring to statements proved in the following subsections. 
In this section, the constant $C>0$ depending on fixed parameters will change from line to line.

We have, by Lemmas \ref{BG_1} and \ref{BG_2}, bounding $|a_{t,x}|\leq M$, that
\begin{align}
&\E_{N} \left |\int_0^T  \sum_{x\in \T^d_N} a_{t,x}f_x dt\right |\nonumber\\
&\leq \E_{N} \left |\int_0^T  \sum_{x\in \T^d_N} a_{t,x}f_x 1(\sum_{y\in \Lambda_h}\eta_{y+x}\leq A) dt\right | + \int_0^T CMH(\mu^N_t|\nu^N_t)dt + \frac{CMTN^d}{A}\nonumber\\
&\leq \E_{N} \left |\int_0^T  \sum_{x\in \T^d_N} a_{t,x}f_x 1(\sum_{y\in \Lambda_h}\eta_{y+x}\leq A)1(\eta^\ell_x\leq B) dt\right |\nonumber\\
\label{main_1}
&\ \ \ \ \ \ \ \ \ \ 
+ \int_0^T CM\left(1+ \frac{A+1}{B}\right) H(\mu^N_t|\nu^N_t)dt + \frac{CMTN^d(A+1)}{B} + \frac{CMTN^d}{A}.
\end{align}

The expectation in the right-side of \eqref{main_1} is bounded by
\begin{align}
\label{main_2}
&\E_{N}\left| \int_0^T \sum_{x\in \T^d_N} a_{t,x}m_x dt\right| + \E_{N}\left|\int_0^T E_{\nu_\beta}\Big[\sum_{x\in \T^d_N} a_{t,x}f_x1\big(\sum_{y\in \Lambda_h}\eta_{y+x}\leq A\big) \big|\eta^\ell_x\Big]1(\eta^\ell_x\leq B)dt\right| 
\end{align}
where $m_x$ is defined in \eqref{m_equation}.
Via Lemma \ref{BG_3}, and Lemmas \ref{BG_6} and \ref{BG_7}, the first plus second term of \eqref{main_2} is bounded, with simple overestimates, by 
\begin{align*}
&\frac{C(T+1)MKN^d}{\G} + \frac{CTM\G \ell^{d+2}A^2B^2N^d}{N^2} \\
&\ \  + 2CM\int_0^T H(\mu^N_t|\nu^N_t)dt   + \frac{2CTMN^d}{\ell^d} +
\frac{CTMKN^d\ell^2(B+1)}{N^2} + \frac{CTMN^d}{A}.
\end{align*}
Here, $A,B, \G, \ell$ are in form $A=N^{\alpha_A}$, $B=N^{\alpha_B}$, $\G=N^{\alpha_{\G}}$ and $\ell=N^{\alpha_\ell}$ for parameters $\alpha_A, \alpha_B,\alpha_{\G}, \alpha_\ell>0$. By the assumptions of Lemmas \ref{BG_2} and \ref{BG_3}, we assume that $\alpha_B = 2\alpha_A$ and 
\begin{equation}
\label{constraint}
\alpha_A+\alpha_{\G} + (d+2)\alpha_\ell + 2\alpha_B -2 = 5\alpha_A +\alpha_{\G}+ (d+2)\alpha_\ell -2<0.
\end{equation}

Combining the estimates, as $A/B = 1/A (\leq 1)$ and $K\geq 1$, the left-hand side of \eqref{main_1}
is bounded by 
\begin{align}
\label{combined_estimate}
C\int_0^T  M H(\mu_t^N| \nu^N_t))dt + C(T+1)MN^d\Big(\frac{1}{A}  + \frac{K}{\G} + \frac{\G\ell^{d+2}A^2B^2}{N^2}+\frac{1}{\ell^d} + \frac{KB\ell^2}{N^2} \Big).
\end{align}

So that the second term on the right-hand side of \eqref{combined_estimate} is bounded by $C(T)MKN^{-\kappa}$ for a $\kappa>0$, 
we now fix $\alpha_B=2\alpha_A$, $\alpha_A$, $\alpha_{\G}$, $d\alpha_\ell$ so that $2-[\alpha_{\G} +[(d+2)/d]d\alpha_\ell + 2\alpha_A + 2\alpha_B]>0$.  Then, the constraint \eqref{constraint} would also hold. 
A convenient choice is $\varepsilon_0=\alpha_A=\alpha_{\G}=d\alpha_\ell = 2- (7+ (d+2)/d)\varepsilon_0$, or when $\varepsilon_0=2d/(9d+2)$.  Inserting into \eqref{combined_estimate} yields the right-hand side of \eqref{BG_equation} as desired.
\hfill \qed

We now turn to the estimates used in the proof of Theorem \ref{BG}.  We will assume throughout this section condition (BIP1) and that $H(\mu^N_0|\nu^N_0)=O(N^d)$. 

To simplify notation, we will drop $t$-dependence in the notation $\eta_x = \eta_x(t)$, and related quantities when the context is clear.

\subsection{Preliminary estimates}

Recall the `entropy inequality' following from the variational form of the relative entropy between two probability measures $\mu$ and $\nu$:
$$E_{\mu}[F] \leq H(\mu|\nu) + \log E_{\nu}\big[e^{F}\big].$$

\begin{lem}
\label{totalsum_lem}
We have, for a small $\gamma>0$, and uniformly over $t\in [0,T]$ that
\begin{equation*}
E_{\mu^N_t}\big[\sum_{x\in \T^d_N} \eta_x\big] \leq \frac{H(\mu^N_t|\nu^N_t)}{\gamma} +O(N^d).
\end{equation*}
\end{lem}

\begin{proof}
Write
\begin{align*}
E_{\mu^N_t}\big[\sum_{x\in \T^d_N} \eta_x\big] &\leq \frac{H(\mu^N_t|\nu^N_t)}{\gamma} + \frac{1}{\gamma}\log E_{\nu^N_t}e^{\gamma\sum_{x\in \T^d_N} \eta_x}\\
&\leq \frac{H(\mu^N_t|\nu^N_t)}{\gamma} + \frac{N^d}{\gamma}\max_{x} E_{\nu^N_t}e^{\gamma \eta_x}\ \leq \ 
 \frac{H(\mu^N_t|\nu^N_t)}{\gamma} + O(N^d),
\end{align*}
given $\max_{x\in \T^d_N} E_{\nu^N_t}e^{\gamma \eta_x}<\infty$ for a $\gamma>0$ small (relative to $\varphi^*$ defined near \eqref{eq:3.2} say), noting the uniform estimate on $u^N$ in Lemma \ref{lem:4.1}.
\end{proof}

  \begin{lem}
\label{global_entropy_lemma}
For $\b>0$ and the $\gamma$ in Lemma \ref{totalsum_lem}, uniformly over $t\in [0,T]$, we have
$$H(\mu^N_t|\nu_\b) \leq (1+C(\beta, u_+)\gamma^{-1})H(\mu^N_t|\nu^N_t) +  O(N^d).$$
In particular, when $H(\mu^N_0|\nu^N_0)=O(N^d)$, we have $H(\mu^N_0|\nu_\beta)=O(N^d)$.
\end{lem}

\begin{proof}
Write
\begin{equation}
\label{help_8}
H(\mu^N_t|\nu_\b) = \int \log \frac{d\mu^N_t}{d\nu_\b}d\mu^N_t = 
H(\mu^N_t|\nu^N_t) + \int \log \frac{d\nu^N_t}{d\nu_\b}d\mu^N_t
\end{equation}
and
$
\frac{d\nu^N_t}{d\nu_\b} =\prod_x \frac{d\nu^N_t}{d\nu_\b}(\eta_x)
$.

From Lemma \ref{lem:4.1}, we have that $u^N$ is uniformly bounded between $c_- =u_-\wedge \alpha_-$ and $c_+=u_+\vee \alpha_+$.
Since $Z_{u^N(t,x)}=\sum \varphi(u^N(t,x))^k/g(k)!$ and $\varphi$ is an increasing function,
we also have $Z_{u^N(t,x)}\geq Z_{c_-}$.  In addition, $\varphi(u^N(t,x))\leq \varphi(c_+)$.
Then,
$$
\frac{d\nu_{u^N(t,x)}}{d\nu_\b}(k) = \frac{Z_{u^N(t,x)}^{-1} \frac{\fa(u^N(t,x))^k}{g(k)!} }
{Z_{\b}^{-1} \frac{\fa(\b)^k}{g(k)!} } = \frac{Z_\b}{Z_{u^N(t,x)}} \frac{\fa(u^N(t,x))^k}{\fa(\b)^k}
\le \frac{Z_\b}{Z_{c_-}} \left( \frac{\fa(c_+)}{\fa(\b)} \right)^k.
$$
Therefore, \eqref{help_8} is bounded by
$$
H(\mu^N_t|\nu^N_t) + N^d \log \frac{Z_\beta}{Z_{c_-}} + \log\big(\tfrac{\fa(c_+)}{\fa(\beta)}\big)E_{\mu^N_t}\big[\sum_{x\in \T^d_N} \eta_x\big].
$$
Noting that $E_{\mu^N_t}\big[\sum_{x\in \T^d_N} \eta_x\big] \leq \gamma^{-1}H(\mu^N_t|\nu^N_t)+ O(N^d)$ by Lemma \ref{totalsum_lem}, the proof is complete. 
\end{proof}

  We now give an estimate to be used several times in the sequel.  Recall $\Lambda_{k}=\{x\in \T^d_N: |x|\leq k\}$ is a cube of width $2k+1$.  Let $\q=\q(\eta)$ be a function supported in $\Lambda_k$.  Denote $\q_x = \tau_x \q$ for $x\in \T^d_N$.  Consider the collection of $|\Lambda_k|$ regular sublattices $\T^d_{N, z, k}\subset \T^d_N$, where $z\in \Lambda_k$ and neighboring points in the grid are separated by $2k+1$.
	
\begin{lem}
\label{grid_lemma}
We have, uniformly over $t\in [0,T]$, that
\begin{align}
\label{grid}
\log E_{\nu^N_t}\big[ e^{\sum_{x\in \T^d_N} \q_x}\big] &\leq \frac{1}{|\Lambda_k|} \sum_{z\in \Lambda_k} \log E_{\nu^N_t} \big[e^{|\Lambda_k| \sum_{w\in \T^d_{N,z,k}} \q_w}\big] \\
&= \frac{1}{|\Lambda_k|}\sum_{x\in \T^d_N} \log E_{\nu^N_t} \big[e^{|\Lambda_k| \q_x}\big].\nonumber
\end{align}
\end{lem}

\begin{proof}
One can write $\sum_{x\in \T^d_N} \q_x = \sum_{z\in \Lambda_k}\sum_{w\in \T^d_{N,z,k}} \q_w$. 
The inequality in \eqref{grid} results from a H\"older's inequality.
The last equality follows since elements $\{\q_w: w\in \T^d_{N,z,k}\}_{z\in \Lambda_k}$ are independent under $\nu^N_t$.
\end{proof}

\subsection{Truncation estimates}
We now develop some truncation estimates, since under the Glauber+Zero-range dynamics, there is no a priori bound on the number of particles at a site $x\in \T^d_N$. 

The first limits the particle numbers in $\tau_x \Lambda_h$, where we recall that $\Lambda_h$ denotes a finite box containing the support of the function $h$ through which $f_x$ is defined in \eqref{f_def}.
\begin{lem}
\label{BG_1}
Let $A=A_N = N^{\alpha_A}$ for $\alpha_A>0$.  Then, uniformly over $t\in [0,T]$, we have
$$E_{\mu^N_t}\Big[ \sum_{x\in \T^d_N} |f_x| 1(\sum_{y\in \Lambda_h} \eta_{y+x}>A)\Big] \leq CH(\mu_t^N| \nu^N_t) + \frac{CN^d}{A}.$$
\end{lem}

\begin{proof}
Write, through the entropy inequality and Lemma \ref{grid_lemma}, with respect to a $\gamma_1>0$, that
\begin{align*}
&E_{\mu^N_t} \Big[\sum_{x\in \T^d_N} |f_x| 1(\sum_{y\in \Lambda_h} \eta_{y+x}>A) \Big]\\
& \leq \frac{ H(\mu_t^N| \nu^n_t)}{\gamma_1} + \frac{1}{\gamma_1}\log E_{\nu^N_t}\Big[ e^{\gamma_1\sum_{x\in \T^d_N} |f_x|1(\sum_{y\in \Lambda_h}\eta_{y+x}>A)}\Big]\\
&\leq \frac{ H(\mu_t^N| \nu^N_t)}{\gamma_1} +\frac{1}{\gamma_1|\Lambda_h|}\sum_{x\in \T^d_N}\log E_{\nu^N_t}\Big[e^{\gamma_1|\Lambda_h||f_x|1(\sum_{y\in \Lambda_h}\eta_{y+x}>A)}\Big]\\
& = \frac{ H(\mu_t^N| \nu^N_t)}{\gamma_1}\\
&\ \  + \frac{1}{\gamma_1|\Lambda_h|}\sum_{x\in \T^d_N}\log \Big\{1- P_{\nu^N_t}\big(\sum_{y\in \Lambda_h}\eta_{y+x}>A\big)
 + E_{\nu^N_t}\Big[1(\sum_{y\in \Lambda_h}\eta_{y+x}>A) e^{\gamma_1|\Lambda_h| |f_x|}\Big]\Big\}.
\end{align*}
The last line is further estimated with the inequality $\log (1+x)\leq x$ for $x\geq 0$, and then Markov's inequality:
\begin{align}
&  \frac{ H(\mu_t^N| \nu^N_t)}{\gamma_1} + \frac{1}{\gamma_1|\Lambda_h|}\sum_{x\in \T^d_N}
E_{\nu_t}\Big[1(\sum_{y\in \Lambda_h}\eta_{y+x}>A) \big(e^{\gamma_1 |\Lambda_h| |f_x|} - 1\big)\Big]\nonumber \\
\label{10.4.1}
&\ \ \ \ \ \ \ \ \ \leq \frac{ H(\mu_t^N| \nu^N_t)}{\gamma_1} + \frac{1}{\gamma_1|\Lambda_h|A}\sum_{x\in \T^d_N}
E_{\nu^N_t}\Big[\sum_{y\in \Lambda_h}\eta_{y+x}e^{\gamma_1 |\Lambda_h| |f_x|}\Big].
\end{align}

We note that $f_x(\eta) \leq C_1\sum_{y\in \Lambda_h}\eta_{x+y} + C_2$ through the bounds \eqref{h_bounds}.  Then, by the uniform estimate Lemma \ref{lem:4.1}, we may choose $\gamma_1$ small enough (relative to $\varphi^*$ in \eqref{eq:3.2}) so that 
$$\sup_{x\in \T^d_N} E_{\nu^N_t}\big[\sum_{y\in \Lambda_h}\eta_{y+x}e^{\gamma_1|\Lambda_h| |f_x|}\big]<\infty$$ 
The display \eqref{10.4.1} is then
bounded by
$CH(\mu_t^N| \nu^N_t) + CN^d/A$, as desired. \end{proof}

We now truncate the average number of particles in a block of width $\ell$ around $x$.  Define 
$$\eta^\ell_x = \frac{1}{(2\ell+1)^d}\sum_{z\in \Lambda_\ell}\eta_{z+x}.$$

\begin{lem}
\label{BG_2}  Let $B=B_N=A^2_N = N^{2\alpha_A}$ for $\alpha_A>0$, and $\ell\geq 1$.
Then, for large $N$, uniformly over $t\in [0,T]$, we have
\begin{align*}
&E_{\mu_t^N} \Big[\sum_{x\in \T^d_N} |f_x| 1(\sum_{y\in \Lambda_h}\eta_{y+x}\leq A) 1(\eta^\ell_x>B) \Big]\\
&\ \ \ \ \ \ \ \ \ \ \ \ \ \  \leq \frac{C (A+1)}{B}H(\mu_t^N|\nu^N_t) + \frac{C(A+1)}{B}N^d.
\end{align*}
\end{lem}

\begin{proof}
Since $f_x(\eta) \leq C_1\sum_{y\in \Lambda_h} \eta_{x+y} + C_2$ by \eqref{h_bounds},
and $\sum_{x\in \T^d_N} \eta^\ell_x = \sum_{x\in \T^d_N} \eta_x$, we have
\begin{align}
E_{\mu_t^N} &\Big[\sum_{x\in \T^d_N} |f_x| 1(\sum_{y\in \Lambda_h}\eta_{y+x}\leq A) 1(\eta^\ell_x>B)\Big] \\
 &\leq   \frac{\max(C_1,C_2)(A+1)}{B}E_{\mu_t^N}\big[\sum_{x\in \T^d_N} \eta^\ell_x\big]
 \ = \ \frac{\max(C_1,C_2)(A+1)}{B}E_{\mu_t^N}\big[\sum_{x\in \T^d_N} \eta_x\big].
\label{step2_eq1}
\end{align}

Now, by Lemma \ref{totalsum_lem}, 
$E_{\mu_t^N}\big[\sum_{x\in \T^d_N} \eta_x\big] \leq CH(\mu_t^N| \nu^N_t) + O(N^d)$.
Then, the display \eqref{step2_eq1} is bounded as desired, for large $N$, by
\begin{equation*}
\frac{C (A+1)}{B}H(\mu_t^N|\nu^N_t) + \frac{C (A+1)}{B}N^d.
\end{equation*}
We note this bound does not depend on the size of $\ell\geq 1$.
\end{proof}

\subsection{Main estimates}

We now estimate the remaining portions of $\sum_{x\in \T^d_N} f_x$.  In Section \ref{bound on concentration}, we show that $f_x$ is in a sense close to its conditional mean given the local density of particles.  In Section \ref{equivalence section}, we estimate this conditional mean.

\subsubsection{Bound on `concentration' around conditional mean}
\label{bound on concentration}
Fix $\beta>0$ and, for $x\in \T^d_N$, let
\begin{equation}
\label{m_equation}
m_x =\Big(f_x 1(\sum_{y\in \Lambda_h}\eta_{y+x}\leq A) - E_{\nu_\beta}\big[f_x 1(\sum_{y\in \Lambda_h}\eta_{y+x}\leq A)|\eta^\ell_x\big]\Big)1(\eta^\ell_x\leq B).
\end{equation}

We remark that the quantified estimate on the spectral gap in (SP) is used now in the proof of the following Lemma \ref{BG_3}. 

\begin{lem}
\label{BG_3}
Let $\ell= \ell_N = N^{\alpha_\ell}$ and $\G=\G_N=N^{\alpha_{\G}}$ for $\alpha_\ell,\alpha_{\G}>0$. Suppose $\alpha_A + \alpha_{\G} + (d+2)\alpha_\ell + 2\alpha_B -2<0$.
Then, we have
\begin{align*}
&\E_{N}\Big| \int_0^T \sum_{x\in \T^d_N} a_{t,x}m_x dt \Big| \leq \frac{C(T+1)MKN^d}{\G} + \frac{CTM\G \ell^{d+2} A^2 B^2 N^d}{N^2}.
\end{align*}
\end{lem}

\begin{proof}
We apply the entropy inequality, with respect to the Zero-range invariant measure $\nu_\beta$, to obtain
$$
\E_{N}\Big| \int_0^T \sum_{x\in \T^d_N} a_{t,x}m_x dt \Big|
\le \frac{H(\mu_0^N| \nu_\beta)}{\gamma}
+ \frac1{\ga} \log \E_{\nu_\b}\left[ e^{\ga | \int_0^T \sum_{x\in \T^d_N} a_{t,x}m_x dt |} \right],
$$
for every $\gamma>0$.
The second term, on the right-hand side of the display, noting $e^{|z|}\le e^{z}+e^{-z}$, is bounded by the Feynman-Kac formula
in Appendix 1.7 in \cite{KL} (whose proof does not require $\nu_\beta$ to be an invariant measure of $L_N$).

 Then, considering $\gamma = \G/M$, we have
\begin{align*}
&\E_{\mu^N}\Big| \int_0^T \sum_{x\in \T^d_N} a_{t,x}m_x dt \Big|\\
&\ \ \ \ \ \ \ \ \ \ \ \  \leq \frac{MH(\mu_0^N| \nu_\beta)}{\G}
+ 2M\int_0^T \sup_h\Big\{\langle M^{-1}\sum_{x\in \T^d_N} a_{t,x}m_x, h\rangle_{\nu_\beta} + \frac{1}{\G}D_N(\sqrt{h})\Big\}dt
\end{align*}
where $h$ is a density with respect to $\nu_\beta$, and $D_N(f) = E_{\nu_\beta}[f(S_N f)]$ is the quadratic form given in terms of $S_N = (L_N + L_N^*)/2$ and the $L^2(\nu_\beta)$ adjoint $L_N^*$.

By Lemma \ref{global_entropy_lemma} and our initial assumption, $H(\mu_0^N|\nu_\beta) \leq O\big(H(\mu^N_0|\nu^N_0)\big) + O(N^d) = O(N^d)$.

To estimate the supremum, write $D_N(f) = -N^2 D_{ZR}(f) + KQ_G(f)$.
By \eqref{Dirichlet_eq}, 
\begin{align*}
D_{ZR}(f) &= E_{\nu_\beta}[f(-L_{ZR}f)] = {\mathcal D}_{ZR}(f;\nu_\beta)\\
& = \frac{1}{4}\sum_{\stackrel{|x-y|=1 }{x,y\in \T^d_N}}E_{\nu_\beta}\big[g(\eta_x)\big(f(\eta^{x,y}) - f(\eta)\big)^2\big].
\end{align*}
 Also, $Q_G(f) = E_{\nu_\beta}[(L_G f)f]$ is explicit following calculations say in Lemma
\ref{lem:3.2} as 
\begin{eqnarray*}
Q_G(f) &=& -\sum_{x\in \T^d_N} E_{\nu_\beta} \Big[c^+_x(\eta)\big(f(\eta^{x,+})-f(\eta)\big)^2\Big]\\
&&\ \ \ -\sum_{x\in \T^d_N} E_{\nu_\beta} \Big[c^-(\eta)1(\eta_x\geq 1)\big(f(\eta^{x,-})-f(\eta)\big)^2\Big]\\
&&\ \ \ -\sum_{x\in \T^d_N} E_{\nu_\beta}\Big[f(\eta)f(\eta^{x,+})c^+_x(\eta) + f(\eta)f(\eta^{x,-})c_x^-(\eta)1(\eta_x\geq 1)\Big]\\
&&\ \ \ + \sum_{x\in \T^d_N} E_{\nu_\beta}\Big[f^2(\eta)\Big(c_x^+(\eta^{x,-})\frac{g(\eta_x)}{\fa(\beta)} + c_x^-(\eta^{x,+})\frac{\fa(\beta)}{g(\eta_x+1)}\Big)\Big].
\end{eqnarray*} 
As the rates $c^\pm\geq 0$, only the last line in the display for $Q_G(f)$ is nonnegative.  By our assumption (BR), however, we have that
$c_x^+(\eta^{x,-})g(\eta_x)$ and $c_x^-(\eta^{x,+})/g(\eta_x+1)$ are bounded.  When $f$ is a nonnegative function such that $f^2$ is a density with respect to $\nu_\b$, that is $E_{\nu_\b}[f^2(\eta)]=1$, we have the upper bound
$$\frac{1}{\G}D_N(f) = -\frac{N^2}{\G}D_{ZR}(f) + \frac{K}{\G}Q_G(f) \leq -\frac{N^2}{\G}D_{ZR}(f) + \frac{CKN^d}{\G}$$
and therefore
\begin{align}
&\E_{\mu^N}\Big| \int_0^T \sum_{x\in \T^d_N} a_{t,x}m_x dt \Big| \leq 
\frac{CMN^d}{\G}\nonumber \\
\label{step3_eq_1}
&\ \ \ \ \ \ \ \ + 2M\int_0^T\sup_h\Big\{M^{-1}\langle \sum_{x\in \T^d_N} a_{t,x}m_x, h\rangle_{\nu_\beta} - \frac{N^2}{\G}D_{ZR}(\sqrt{h})\Big\} dt+ \frac{CTMKN^d}{\G}.
\end{align}

To analyze further, 
define
$$D_{\ell, x}(f) = E_{\nu_\beta}[f(-L_{\ell,x}f)] = \frac{1}{4}\sum_{|w-z|=1\atop w,z\in \Lambda_{\ell,x} } E_{\nu_\beta}\Big[g(\eta_w)\big(f(\eta^{w,z})-f(\eta)\big)^2\Big]$$
where $L_{\ell,x}$ is the Zero-range generator restricted to sites $\Lambda_{\ell,x} = \{y +x: |y|\leq \ell\}$.

Define also the associated canonical process on $\Lambda_{\ell,x}$ where the number of particles $\sum_{y\in \Lambda_{\ell,x}}\eta_y = j$ is fixed for $j\geq 0$.  Let $L_{\ell,x,j}$ denote its generator and let $\nu_{\ell, x,j} = \nu_\beta(\cdot|\sum_{y\in \Lambda_{\ell,x,j}}\eta_y=j)$ be its canonical invariant measure on the configuration space $\big\{\{\eta_z\}_{z\in \Lambda_{\ell,x}}:\sum_{y\in \Lambda_{\ell,x,j}}\eta_y=j\big \}$.    By translation-invariance, $\nu_{\ell,x,j}$ does not depend on $x$.

Then, counting the overlaps, we have
$$\sum_{x\in \T^d_N} D_{\ell, x}(\sqrt{h}) = (2\ell+1)^d D_{ZR}(\sqrt{h}).$$
The supremum on the right-hand side of \eqref{step3_eq_1} is less than
\begin{equation}
\label{step3_eq_4}
\sum_{x\in \T^d_N}\sup_h \Big\{ E_{\nu_\beta}[(a_{t,x}/M)m_x h] - \frac{N^2}{\G \ell_*^d}D_{\ell, x}(\sqrt{h})\Big\}
\end{equation}
where $\ell_*= 2\ell +1$.

Recall that $\G = N^{\alpha_{\G}}$ for a small $\alpha_{\G}>0$, and $m_x$ vanishes unless the density of particles in the $\ell$-block is bounded, $\eta^\ell(x)\leq B$.  By conditioning on the number of particles in $\Lambda_{\ell,x}$, and dividing and multiplying by $E_{\nu_\beta}[h|\sum_{z\in \Lambda_{\ell,x}}\eta_z=j]$, we have for each $x$ and $t$ that
\begin{align*}
&\sup_h \Big\{ E_{\nu_\beta}[(a_{t,x}/M)m_x h] - \frac{N^2}{\G \ell_*^d}D_{\ell, x}(\sqrt{h})\Big\}\nonumber\\
&\ \ \ \ \ \leq \sup_{j\leq B\ell_*^d} \sup_h \Big\{ E_{\nu_{\ell, x, j}}[(a_{t,x}/M)m_x h] - \frac{N^2}{\G \ell_*^d}D_{\ell, x, j}(\sqrt{h})\Big\}\nonumber
\end{align*}
where $h$ is a density with respect to $\nu_{\ell,x,j}$.

 Now, by the Rayleigh estimate in \cite{KL} p. 375, Theorem 1.1, in terms of the spectral gap of the canonical process $gap(\ell,j)$, which does not depend on $x$ by translation-invariance, the last display is bounded by
\begin{equation}
\label{denom_eq}
 \sup_{j\leq B\ell_*^d}\frac{\G\ell_*^d}{N^2} \frac{ E_{\nu_{\ell, x, j}}[(a_{t,x}/M)m_x \{(-L_{\ell, x,j})^{-1} (a_{t,x}/M)m_x\}]}{1-2\|(a_{t,x}/M)m_x\|_{L^\infty} \frac{\G\ell_*^d}{N^2}gap(\ell, j)^{-1}}.
\end{equation}

By the bounds on $f_x$ via \eqref{h_bounds} and those on $\{a_{t,x}\}$, we have $\|(a_{t,x}/M)m_x\|_{L^\infty} = O(A)$.  Since $m_x$ is mean-zero with respect to $\nu_{\ell,x,j}$, we have 
$$E_{\nu_{\ell,x,j}}\big[ (a_{t,x}/M)m_x\{ (-L_{\ell,x,j})^{-1}(a_{t,x}/M)m_x\}\big] \leq gap(\ell,j)^{-1}\|(a_{t,x}/M)m_x\|^2_{L^\infty}.$$

Recall the spectral gap assumption (SP) that $gap(\ell,j)^{-1} \leq C_{gp}\ell^2 (j/\ell^d)^2$.  Since $j/\ell^d \leq CB$, we have that $gap(\ell,j)^{-1}\leq C_{gp}\ell^2B^2$.  Choosing 
$\alpha_A + \alpha_{\G} + (d+2)\alpha_\ell + 2\alpha_B -2<0$, we have
$$A \G \ell^d gap(\ell, j)^{-1}/N^2 \leq C_{gp} A\G\ell^{d+2}B^2N^{-2} = o(1)$$
and so the denominator in \eqref{denom_eq} is bounded below.
 
Hence, summing over $x$, \eqref{step3_eq_4} is bounded above by
$$\frac{C\G\ell^d}{N^2}\|m_x\|^2_{L^\infty}\ell^{2}N^d \leq \frac{C\G\ell^d A^2\ell^2B^2N^d}{N^2},$$
and the desired estimate follows by inserting back into \eqref{step3_eq_1}. 
\end{proof}

\subsubsection{Bound on conditional mean}
\label{equivalence section}
To treat the conditional expectation 
\begin{equation}
\label{conditional mean}
E_{\nu_\b}\big[a_{t,x}f_x 1(\sum_{y\in \Lambda_h}\eta_{y+x}\leq A)|\eta_x^\ell\big]1(\eta^\ell_x\leq B),
\end{equation}
 we will need two preliminary estimates (Lemmas \ref{BG_4} and \ref{unif_estimate}).  

Note that $f_x$ is mean-zero with respect to $\nu_{u^N(t,x)}$ and $E_{\nu_{u^N(t,x) + \kappa}}[f_x] = \tilde h\big(u^N(t,x) + \kappa\big) - \tilde h\big(u^N(t,x)\big) - \tilde h'\big(u^N(t,x)\big)\kappa$ from the definition \eqref{f_def}.

\begin{lem}
\label{BG_4}
For $x\in \T^d_N$, let $y_x = \eta^\ell_x - u^N(t,x)$.  Fix also $\delta>0$.  We have, uniformly over $t\in [0,T]$, that
$$\Big|E_{\nu_\b}\big[(a_{t,x}/M)f_x 1(\sum_{y\in \Lambda_h}\eta_{y+x}\leq A)|\eta_x^\ell\big]\Big|1(|y_x|\leq \delta) \leq C y_x^21(|y_x|\leq \delta) + \frac{C}{\ell^d} + \frac{C}{A}.$$
\end{lem}

\begin{proof} 
The argument makes use of an equivalence of ensembles estimate and properties of $f_x$.  Let $b_x = (a_{t,x}/M)f_x 1(\sum_{y\in \Lambda_h}\eta_{y+x}\leq A)$. Recall $\|a_{t,x}\|_\infty/M \leq 1$.  By Corollary 1.7 in Appendix 2 of \cite{KL}, when $|y_x|\leq \delta$, we have that
$$\Big|E_{\nu_\b}[b_x
|\eta_x^\ell]\Big| \leq \big|E_{\nu_{u^N(t,x) + y_x}}[b_x
]\big| + \frac{C}{\ell^d}.$$

 We now expand $E_{\nu_{u^N(t,x) + y_x}}[b_x]$ in terms of $y_x$ around $0$.  Choose $\lambda = \lambda(y_x)$ so that 
\begin{equation}
\label{help_4}
\frac{E_{\nu_{u^N(t,x)}}[\eta_x e^{\lambda (\eta_x - u^N(t,x))}]}{E_{\nu_{u^N(t,x)}}[e^{\lambda (\eta_x - u^N(t,x))}]} = u^N(t,x) + y_x.
\end{equation}
Note from \eqref{help_4} that $\lambda(0)=0$ and $\lambda'(0):=\frac{d}{dy_x} \lambda(0) = E_{\nu_{u^N(t,x)}}[(\eta_x-u^N(t,x))^2]^{-1}$.  In terms of this change of measure, 
\begin{align*}
\frac{d}{dy_x} E_{\nu_{u^N(t,x)+y_x}}[b_x]\big|_{y_x=0} 
&= \lambda'(0)E_{\nu_{u^N(t,x)}}\Big[b_x\big(\sum_{y\in \Lambda_h}(\eta_{y+x}-u^N(t,x))\big)\Big].
\end{align*}
  Since $u^N$ is uniformly bounded away from $0$ and infinity (Lemma \ref{lem:4.1}), $\lambda'(0)$ is bounded.  Also, from \eqref{help_4}, one can see that $\lambda''(a) = \frac{d^2}{dy^2_x} \lambda(a)$, for $|a|\leq \delta$, is also bounded say by $C(\delta)$
for $|a|\leq \delta$.
Then,
\begin{equation}
\label{help_5}
E_{\nu_{u^N(t,x) + y_x}}[b_x] = E_{\nu_{u^N(t,x)}}[b_x] + \Big[\frac{d}{dy_x}E_{\nu_{u^N(t,x)+y_x}}[b_x]\big|_{y_x=0}\Big]y_x + r_x
\end{equation}
where $|r_x|\leq (C(\delta)/2)y_x^2$.

We now estimate that the first two terms on the right-hand side of \eqref{help_5} are of order $A^{-1}$ to finish the argument.
Indeed,
\begin{align*}
&|E_{\nu_{u^N(t,x)}}[b_x]| = \big|E_{\nu_{u^N(t,x)}}[(a_{t,x}/M)f_x] - E_{\nu_{u^N(t,x)}}[(a_{t,x}/M)f_x1(\sum_{y\in \Lambda_h}\eta_{y+x}> A)]\big|\\
&\leq \frac{1}{A}E_{\nu_{u^N(t,x)}}\big[|f_x|\sum_{y\in \Lambda_h}\eta_{y+x}\big] \leq \frac{C}{A}
\end{align*}
as $f_x$ is mean-zero with respect to $\nu_{u^N(t,x)}$ and $E_{\nu_{u^N(t,x)}}[|f_x|\sum_{y\in \Lambda_h}\eta_{y+x}]$ is uniformly bounded as $u^N(t,\cdot)$ is uniformly bounded in Lemma \ref{lem:4.1}.

The other term is similar:
\begin{align*}
&\big|\frac{d}{dy_x}E_{\nu_{u^N(t,x)+y_x}}[b_x]|_{y_x=0}\big|
\ \leq \ \big|\frac{d}{dy_x}E_{\nu_{u^N(t,x)+y_x}}[(a_{t,x}/M)f_x]|_{y_x=0}\big| \\
&\ \ \ \ \ \ \ + \big|\frac{d}{dy_x}E_{\nu_{u^N(t,x)+y_x}}[(a_{t,x}/M)f_x(\sum_{y\in \Lambda_h}\eta_{y+x}> A)]|_{y_x=0}\big|\\
&\leq \frac{\lambda'(0)}{A} \Big|E_{\nu_{u^N(t,x)}}\big[(a_{t,x}/M)f_x(\sum_{y\in \Lambda_h}\eta_{y+x})(\sum_{y\in \Lambda_h}(\eta_{y+x}-u^N(t,x)))\big]\Big|
\leq \frac{C}{A},
\end{align*}
since first $a_{t,x}$ is non-random and $f_x$ satisfies $0=\frac{d}{dy_x}E_{\nu_{u^N(t,x)+y_x}}[f_x]|_{y_x=0}$ (cf. \eqref{f_def}), and second
$$E_{\nu_{u^N(t,x)}}\big[(a_{t,x}/M)f_x\big(\sum_{y\in \Lambda_h}\eta_{y+x}\big)\big(\sum_{y\in \Lambda_h}(\eta_{y+x}-u^N(t,x))\big)\big]$$
 is uniformly bounded as $u^N(t,\cdot)$ is uniformly bounded (Lemma \ref{lem:4.1}).
\end{proof}

Recall $\ell_* = 2\ell+1$ and let now 
$$\tilde y_x = \frac{1}{(2\ell+1)^d}\sum_{|z|\leq \ell} \big(\eta_{z+x} - u^N(t,z+x)\big).$$
We will need that the following exponential moment is uniformly bounded.
\begin{lem}
\label{unif_estimate}
For $\gamma, \delta>0$ small, uniformly over $t\in [0,T]$, we have
$$\sup_\ell E_{\nu^N_t} \Big[e^{\gamma \ell_*^d \tilde y_x^2 }1(|\tilde y_x|\leq \delta)\Big] <\infty.$$
\end{lem}

To get a feel for this estimate, consider the case that the variables are i.i.d. Poisson with parameter $\kappa$.  Then, $\tilde y_x$ has the distribution of $\ell_*^{-d}$ times a centered Poisson($\ell_*^d) \kappa$ random variable.  In this case, the expectation in this lemma  
equals
\begin{align*}
\sum_{|k-\ell_*^d\kappa|\leq \ell_*^d \delta} e^{\gamma\ell_*^{-d}(k-\ell_*^d\kappa)^2} e^{-\ell_*^d \kappa} \frac{\big(\ell_*^d\kappa\big)^k}{k!}.
\end{align*}
A typical summand, say with $k\sim (\kappa + \delta)\ell_*^d$ is estimated by Stirling's formula as
\begin{align*}
 e^{\gamma\ell_*^d\delta^2}e^{-\ell_*^d\kappa}\frac{\big(\ell_*^d\kappa\big)^{\ell_*^d(\kappa + \delta)}}{\big(\ell_*^d(\kappa + \delta)\big)!}
 \sim e^{-c\ell_*^d},
 \end{align*}
 with $c>0$, when $\gamma \delta^2 < \kappa$.  Since there are only $\ell_*^d$ order summands, Lemma \ref{unif_estimate} holds in this setting.
 
 We now give an argument for the general case through use of a local central limit theorem.  Denote now $\kappa=\ell_*^{-d}\sum_{|z|\leq \ell}u^N(t,x+z)$.

\medskip
\noindent {\it Proof of Lemma \ref{unif_estimate}.}
Write the expectation in the display of Lemma \ref{unif_estimate} as
 \begin{align}
 \label{unif_estimate1}
& \sum_{|k-\ell_*^d\kappa|<\ell_*^d\delta} e^{\gamma \ell_*^{-d}(k-\ell_*^d\kappa)^2} \nu^N_t\big(\sum_{|z|\leq \ell} \eta_{z+x} = k\big)\\
&\ \ \ \ \ \ \ \ \ = \sum_{k=\lceil \ell_*^d\kappa\rceil}^{\lfloor \ell_*^d(\kappa+\delta)\rfloor}e^{\gamma \ell_*^{-d}(k-\ell_*^d\kappa)^2} {\nu^N_t}\big(\sum_{|z|\leq \ell} \eta_{z+x} = k\big)\nonumber\\
&\ \ \ \ \ \ \ \ \ \ \ \ \ \ \ \ \ \ + \sum_{k=\lceil \ell_*^d(\kappa-\delta)\rceil}^{\lceil \ell_*^d\kappa\rceil - 1}e^{\gamma \ell_*^{-d}(k-\ell_*^d\kappa)^2} \nu^N_t\big(\sum_{|z|\leq \ell} \eta_{z+x} = k\big).\nonumber
 \end{align}
We now bound uniformly the first sum, and discuss the second sum afterwards.

Write the first sum on the right-hand side of \eqref{unif_estimate1}, in terms of a positive constant $a$, as
\begin{align}
\label{unif_estimate2}
&\sum_{k=\lceil \ell_*^d\kappa\rceil}^{\lceil \ell_*^d\kappa\rceil+ a\lceil \ell_*^{d/2}\rceil} e^{\gamma \ell_*^{-d}(k-\ell_*^d\kappa)^2} \nu^N_t\big(\sum_{|z|\leq \ell} \eta_{z+x} = k\big) \\
&\ \ \ \ \ \ \ \ \ + \sum_{k=\lceil \ell_*^d\kappa\rceil+ a\lceil \ell_*^{d/2}\rceil +1}^{\lfloor \ell_*^d(\kappa+\delta)\rfloor}  e^{\gamma \ell_*^{-d}(k-\ell_*^d\kappa)^2} \nu^N_t\big(\sum_{|z|\leq \ell} \eta_{z+x} = k\big).\nonumber
\end{align}
The first term in \eqref{unif_estimate2}, since $0\leq k-\ell_*^d\kappa\leq a\ell_*^{d/2}$, is bounded by $e^{a^2\gamma}$.

To estimate the second term in \eqref{unif_estimate2}, noting $\ell_*^{d/2}\tilde y_x = \ell_*^{-d/2}\big(\sum_{|z|\leq \ell} \eta_{z+x} - \ell_*^d \kappa\big)$, we write the probability in the sum as a difference of $1-F(\ell_*^{-d/2}(k-1-\ell_*^d\kappa))$ and $1-F(\ell_*^{-d/2}(k-\ell_*^d\kappa))$, where $F$ is the distribution function of $\ell_*^{d/2}\tilde y_x$.  Then, we may rewrite the second term in \eqref{unif_estimate2}, summing by parts, as
\begin{align}
\label{help_6}
&\sum_{k=\lceil \ell_*^d\kappa\rceil+ a\lceil \ell_*^{d/2}\rceil +1}^{\lfloor \ell_*^d(\kappa+\delta)\rfloor-1} \big[e^{\gamma \ell_*^{-d}(k+1 - \ell_*^d\kappa)^2} - e^{\gamma \ell_*^{-d}(k - \ell_*^d\kappa)^2}\big]\big[1- F(\ell_*^{-d/2}(k-\ell_*^d\kappa))\big]\\
& \ \ + e^{\gamma \ell_*^{-d}(k-\ell_*^d\kappa)^2}\big[1-F(\ell_*^{-d/2}(k-1-\ell_*^d\kappa))\big]|_{k=\lceil \ell_*^d\kappa\rceil+ a\lceil \ell_*^{d/2}\rceil +1}\nonumber\\
&\ \  - e^{\gamma \ell_*^{-d}(k-\ell_*^d\kappa)^2}\big[1-F(\ell_*^{-d/2}(k-\ell_*^d\kappa))\big]|_{k=\lfloor \ell_*^d(\kappa+\delta)\rfloor}.\nonumber
\end{align}

  In Theorem 10 of Chapter 8 in \cite{Petrov} (page 230), subject to assumptions, namely that equations (2.3), (2.4) and (2.5) in Chapter 8 \cite{Petrov} hold,
a uniform estimate on the tail of the distribution function is given.  These assumptions hold when there is an $H>0$ small where $R_{z,t}(u)= \log E_{\nu^N_t} e^{u\eta_z}$ is uniformly bounded in $z$ and $t$ for $|u|\leq H$, and also when $\sigma^2_{z,t}=E_{\nu^N_t}[(\eta_z-u^N(t,z))^2]$ is uniformly bounded away from $0$ in $z$ and $t$.  These specifications follow straightforwardly from the uniform bounds on $u^N$ (Lemma \ref{lem:4.1}).  

Then, $v_\ell := \sqrt{\ell_*^{-d}\sum_{|z|\leq \ell}\sigma^2_{z,t}}$ is uniformly bounded away from $0$ and $\infty$.

Therefore, by Theorem 10 in Chapter 8 \cite{Petrov}, there is a constant $\tau$ such that for $0\leq x\leq \tau\ell_*^{d/2}$ we have
\begin{align}
1-F(x) &\leq C(\tau)\big(1-\Phi(x/v_\ell)\big) \exp\Big\{ \frac{x^3}{v^3_\ell\ell_*^{d/2}} \kappa^1(x/(v_\ell\ell_*^{d/2}))\Big\} \ \ {\rm and }\nonumber\\
\label{10.8.1}
F(-x)&\leq C(\tau)\Phi(x/v_\ell)\exp\Big\{ - \frac{x^3}{v^3_\ell\ell_*^{d/2}}\kappa^1(-x/(v_\ell\ell_*^{d/2}))\Big\},
\end{align}
where $\kappa^1(\cdot)$ is uniformly bounded for small arguments, and $\Phi$ is the Normal$(0,1)$ distribution function.
Note that 
\begin{align*}
&\exp\big\{\gamma \ell_*^{-d}(k+1 - \ell_*^d\kappa)^2\big\} - \exp\big\{\gamma \ell_*^{-d}(k - \ell_*^d\kappa)^2\big\}\\
& = \exp\big\{\gamma \ell_*^{-d} (k-\ell_*^d\kappa)^2\big\}\Big(\exp\big\{2\gamma\ell_*^{-d}(k-\ell_*^d\kappa) + \gamma\ell_*^{-d}\big\} -1\Big).
\end{align*}
Also, when $x/v_\ell=\ell_*^{-d/2}(k-\ell_*^d\kappa)/v_\ell\geq 1$, which is the case when $k\geq \ell_*^d\kappa + a\ell_*^{d/2}$ and $a$ is fixed large enough, we have that 
\begin{align*}
&\Big\{1-\Phi(\ell_*^{-d/2}(k-\ell_*^d\kappa)/v_\ell) \Big\} \exp\big\{\kappa^1\big(x/(v_\ell\ell_*^{d/2})\big)v_\ell^{-3}\ell_*^{-2d}(k-\ell_*^d\kappa)^3\big\}\\
&\ \ \ \  \leq 
\frac{1}{\sqrt{2\pi}} \exp\big\{-\ell_*^{-d}(k-\ell_*^d\kappa)^2/(2v^2_\ell)\big\} 
\exp\big\{\kappa^1
\big(x/(v_\ell 
\ell_*^{d/2})\big)
v_\ell^{-3}\ell_*^{-2d}(k-\ell_*^d\kappa)^3\big\}.
\end{align*}

With the aid of these observations, we deduce now that \eqref{help_6} is uniformly bounded in $\ell$.    
Indeed, to see that the sum in \eqref{help_6} is bounded, observe since
$a\ell_*^{d/2}\leq  k-\ell_*^d\kappa \leq \delta\ell_*^d$ that
$$\exp\big\{2\gamma\ell_*^{-d}(k-\ell_*^d\kappa) + \gamma\ell_*^{-d}\big\} -1 \leq 2\Big(2\gamma\ell_*^{-d}(k-\ell_*^d\kappa) + \gamma\ell_*^{-d}\Big),$$
and $\kappa^1(x/(v_\ell \ell_*^{d/2})) \leq \bar\kappa$ where $\bar\kappa$ is a constant, when $\gamma$ and $\delta$ are small.  Then,
each summand in the sum in \eqref{help_6} is bounded by
\begin{align*}
\frac{2}{\sqrt{2\pi}}\big(2\gamma\ell_*^{-d}(k-\ell_*^d\kappa) + \gamma \ell_*^{-d}\big)
 \exp\Big\{\big(\ell_*^{-d}(k-\ell_*\kappa)^2\big)\big[\gamma -\frac{1}{2v_\ell^2} + \frac{\delta\bar\kappa}{v_\ell^3}\big]\Big\},
\end{align*}
which in turn is bounded by a multiple of
$$\gamma\ell_*^{-d}(k-\ell_*^d\kappa + 1)e^{-c(\gamma,\delta)\ell_*^{-d}(k-\ell_*^d\kappa)^2}$$
for $\gamma,\delta>0$ chosen small, with $c(\gamma,\delta)>0$.  
Hence, the sum may be bounded uniformly in $\ell$ in terms of the integral $C(\gamma)\int_{a}^{\infty} z e^{-c(\gamma,\delta) z^2}dz$, for some constant $C(\gamma)$.

The other two terms in \eqref{help_6} are bounded using similar ideas. 

Finally, the second sum in \eqref{unif_estimate1} is bounded uniformly in $\ell$ analogously, using the left tail estimate in \eqref{10.8.1}.
\qed

With these preliminary bounds in place, we resume the argument and consider the conditional expectation \eqref{conditional mean} when $|y_x|\leq \delta$.

\begin{lem}
\label{BG_6}
For $\delta>0$ small, we have
\begin{align*}
&\int_0^TE_{\mu^N_t}\Big[\sum_{x\in \T^d_N}E_{\nu_\b}[a_{t,x}f_x 1(\sum_{y\in \Lambda_h}\eta_{y+x}\leq A)|\eta_x^\ell]1(\eta^\ell_x\leq B)1(|y_x|\leq \delta)\Big]dt\\
&\leq CM\int_0^T H(\mu^N_t|\nu^N_t)dt + \frac{CTMN^d}{\ell^d} + \frac{CTMKN^d\ell^2}{N^2}  + \frac{CTMN^d}{A}.
\end{align*}
\end{lem}

\begin{proof}
We first divide and multiply the left-hand side of the display by $M$.
By Lemma \ref{BG_4}, we first bound the term 
\begin{align*}
&\Big|E_{\nu_\b}[(a_{t,x}/M)f_x 1(\sum_{y\in \Lambda_h}\eta_{y+x}\leq A)|\eta_x^\ell]1(\eta^\ell_x\leq B)1(|y_x|\leq \delta)\Big|\\
&\ \ \ \ \ \ \  \ \ \ \ \ \ \  \leq \ C y_x^21\big(|y_x|\leq \delta\big) + \frac{C}{\ell^d} + \frac{C}{A}.
\end{align*}
The last two terms when multiplied by $M$, summed over $x\in \T^d_N$, and integrated over $[0,T]$, give rise to those terms $CTMN^d/\ell^d + CTMN^d/A$ present in the right-hand side of the display of Lemma \ref{BG_6}.

We now concentrate on the terms $y_x^21(|y_x|\leq \delta)$.  Recall $\tilde y_x = \ell_*^{-d}\sum_{|z-x|\leq \ell} (\eta_z - u(t,z))$.
Bound
\begin{align*}
1(|y_x|\leq \delta)
&\leq 1(|\tilde y_x| \leq 2\delta)+1(|y_x|\leq \delta) 1(|y_x - \tilde y_x|\geq \delta). 
\end{align*}

Hence,
\begin{align}
&M\int_0^T  E_{\mu^N_t} \Big[\sum_{x\in \T^d_N} y_x^21(|y_x|\leq \delta)\Big]dt 
 \ \leq \ M\int_0^T E_{\mu^N_t}\Big[ \sum_{x\in \T^d_N} y_x^2 1(|\tilde y_x|\leq 2\delta)\Big]dt\nonumber \\
 \label{help6.5}
& \ \ \ \ \ \ + M\int_0^T E_{\mu^N_t} \Big[\sum_{x\in \T^d_N} y_x^2 1(|y_x|\leq \delta)1(|y_x - \tilde y_x|\geq \delta)\Big]dt.
\end{align}
To bound the second term in \eqref{help6.5}, since $|y_x - \tilde y_x| = \big|\ell_*^{-d}\sum_{|z-x|\leq \ell} u^N(t,z)-u^N(t,x)\big|$, we have by Markov's inequality and Lemma \ref{u_cont_lem} that 
\begin{align}
\label{help_7}
&M\int_0^TE_{\mu^N_t}\Big[\sum_{x\in \T^d_N}  y_x^2 1(|y_x|\leq \delta)1(|y_x-\tilde y_x|\geq \delta)\Big] dt
\leq \delta^2\frac{CTMKN^d\ell^2}{\delta^2 N^2}.
\end{align}

To
bound the first term in \eqref{help6.5}, write
\begin{align}
\label{y_to_tilde_y}
y_x^2 \leq 2\tilde y^2_x
+ 2\Big(\frac{1}{\ell_*^d}\sum_{|z-x|\leq \ell} (u^N(t, z) - u^N(t,x))\Big)^2.
\end{align}
By Lemma \ref{u_cont_lem} again,
\begin{align}
\label{error_repl}
&M\int_0^T  E_{\mu^N_t} \Big[\sum_{x\in \T^d_N} y_x^21(|\tilde y_x|\leq 2\delta)\Big]dt \\
&\ \ \ \ \leq \frac{CTMKN^d\ell^2}{N^2}
 + 2M\int_0^T E_{\mu^N_t}\Big[\sum_{x\in \T^d_N} \tilde y_x^2 1(|\tilde y_x|\leq 2\delta)\Big] dt. \nonumber
\end{align}
The sum of the first terms on the right-hand sides of \eqref{help_7} and \eqref{error_repl} gives the third term in Lemma \ref{BG_6}, writing $2C$ as $C$.

To address the remaining second term in \eqref{error_repl}, write
\begin{align}
M E_{\mu^N_t} \Big[\sum_{x\in \T^d_N} \tilde y_x^2 1(|\tilde y_x|\leq 2\delta)\Big]  &\leq \frac{M H(\mu^N_t| \nu^N_t)}{\gamma_2} 
+ \frac{M}{\gamma_2}\log E_{\nu^N_t}\Big[ e^{\gamma_2 \sum_{x\in \T^d_N} \tilde y_x^21(|\tilde y_x|\leq 2\delta)}\Big]\nonumber\\
\label{help_88}
& \leq \frac{M H(\mu^N_t| \nu^N_t)}{\gamma_2} + \frac{M}{\gamma_2\ell_*^d} \sum_{x\in \T^d_N} \log E_{\nu^N_t}\Big[ e^{\gamma_2 \ell_*^d \tilde y_x^2 1(|\tilde y_x|\leq 2\delta)}\Big],
\end{align}
using Lemma \ref{grid_lemma} where the grid spacing is $2\ell+1$.
By Lemma \ref{unif_estimate}, we have that
\begin{align*}
\log E_{\nu^N_t} \Big[e^{\gamma_2 \ell_*^d \tilde y_x^2 1(|\tilde y_x|\leq 2\delta)} \Big]
&\leq \log \Big\{ 1 + E_{\nu^N_t} \Big[e^{\gamma_2 \ell_*^d \tilde y_x^2} 1(|\tilde y_x|\leq 2\delta)\Big]\Big\}\\
&\leq E_{\nu^N_t} \Big[e^{\gamma_2 \ell_*^d \tilde y_x^2} 1(|\tilde y_x|\leq 2\delta)\Big]
\end{align*}
is uniformly bounded in $\ell$ for small $\gamma_2, \delta>0$.  Hence, the right-hand side of \eqref{help_88} is bounded by 
$M H(\mu^N_t| \nu_t)/\gamma_2 + CMN^d/(\gamma_2\ell^d)$,
finishing the argument. 
\end{proof}

Finally, our last estimate bounds the conditional expectation in \eqref{conditional mean} when $|y_x|>\delta$. 

\begin{lem}
\label{BG_7}
We have, for $\delta>0$ and small $\gamma_3=\gamma_3(\delta)>0$, in terms of a constant $c_1=c_1(\delta)>0$, that
\begin{align}
\label{step6first}
& \int_0^TE_{\mu_t^N}\Big[ \sum_{x\in \T^d_N} E_{\nu_\beta}\big[|a_{t,x}||f_x| |y_x\big]1(\eta^\ell_x\leq B)1(|y_x|>\delta)\Big]dt \\
&\ \ \ \ \leq 
 \frac{M}{\gamma_3}\int_0^TH(\mu^N_t|\nu^N_t)dt + \frac{CTMKBN^d\ell^2}{\delta^2N^2}+ \frac{CTMN^d}{\gamma_3\ell^d} e^{-c_1\ell^d}. 
\nonumber
\end{align}
\end{lem}

\begin{proof}
First, we have $|a_{t,x}|\leq M$.  Next, by our assumptions on $f_x$ (cf \eqref{h_bounds}), using exchangeability of the canonical measure and the uniform bounds on $u^N$ in Lemma \ref{lem:4.1}, we have that
\begin{align*}
E_{\nu_\beta}\big[|f_x||y_x\big] \leq C(|\Lambda_h|)\Big\{ \eta^\ell_x + C\Big\}
&= C(|\Lambda_h|) \Big\{\tilde y_x + \frac{1}{\ell_*^d}\sum_{|z-x|\leq \ell} u^N(t,z) + C\Big\}\\
&\leq C\tilde y_x + C.
\end{align*}

Hence, we need only bound
$M
E_{\mu_t^N}\big[\sum_{x\in \T^d_N} \big(C\tilde y_x + C\big)  1(\eta^\ell_x\leq B)1(|y_x|>\delta)\big]$.
Since
$$1(|y_x|>\delta) \leq 1\Big(\Big|\frac{1}{\ell_*^d}\sum_{|z-x|\leq \ell} \big(u^N(t,x) - u^N(t,z)\big)\Big|>\delta/2\Big) + 1(|\tilde y_x|>\delta/2)$$
and $|\tilde y_x|1(\eta^\ell_x\leq B) \leq B + \sup_x \|u^N(t,x)\|_{L^\infty}\leq 2B$ say, by Lemma \ref{lem:4.1}, for large $N$,
in turn, we need only bound 
\begin{align}
\label{step6_eq_1}
&\big(1+ \delta^{-1}\big)\int_0^T ME_{\mu^N_t}\Big[\sum_{x\in \T^d_N} |\tilde y_x| 1(|\tilde y_x|>\delta/2)\Big]dt \\
&\ \ + \int_0^T \frac{8MB}{\delta^2}\sum_{x\in \T^d_N} \Big(\frac{1}{\ell_*^d}\sum_{|z-x|\leq \ell} \big(u^N(t,x) - u^N(t,z)\big)\Big)^2dt.\nonumber
\end{align}
The second term in \eqref{step6_eq_1} is bounded by $CTMKBN^d \ell^2/\big(\delta^2 N^2)$ via Lemma \ref{u_cont_lem}, giving one of the terms in \eqref{step6first}.

However, the integrand of the first expression in \eqref{step6_eq_1} is bounded, by Lemma \ref{grid_lemma} with grid spacing $2\ell +1$, by 
\begin{align*}
& \frac{MH(\mu_t^N| \nu^N_t)}{\gamma_3} + \frac{M}{\gamma_3\ell_*^d}\sum_{x\in \T^d_N} \log \Big(1- \nu^N_t(|\tilde y_x|>\delta/2) + E_{\nu^N_t} \Big[e^{\gamma_3 \ell_*^d|\tilde y_x|}1(|\tilde y_x|>\delta/2)\Big]\Big).
\end{align*}

By Schwarz inequality, we have
\begin{align*}
&E_{\nu^N_t} \Big[e^{\gamma_3 \ell_*^d|\tilde y_x|}1(|\tilde y_x|>\delta/2)\Big] \leq \Big\{E_{\nu^N_t} \Big[e^{2\gamma_3 \ell_*^d |\tilde y_x| }\Big]\cdot {\nu^N_t}(|\tilde y_x|>\delta/2)\Big\}^{1/2}.
\end{align*}

Now, for $s>0$,
\begin{align*}
&{\nu^N_t}(|\tilde y_x|>\delta) \leq E_{\nu^N_t} \big[e^{s\tilde y_x\ell_*^d}\big] e^{-s\ell_*^d\delta} + E_{\nu^N_t}\big[e^{-s\tilde y_x\ell_*^d} \big]e^{-s\ell_*^d\delta}\\
&\ \ \leq \prod_{|z|\leq \ell}E_{\nu^N_t}\big[e^{s(\eta_{x+z} - u^N(t,x+z))}\big]e^{-s\delta} + \prod_{|z|\leq \ell}E_{\nu^N_t}\big[e^{-s(\eta_{x+z} - u^N(t,x+z))}\big]e^{-s\delta}.
\end{align*}
Moreover, recalling $\sigma^2_{z,t} = E_{\nu^N_t}[(\eta_z-u^N(t,z))^2]$, we have
$$\log E_{\nu^N_t}\big[e^{\pm s(\eta_y- u^N(t,y)}\big] = s^2\sigma^2_{y,t}/2 + o(s^2).$$
Hence, with $s=\varepsilon \delta$ and $\varepsilon>0$ small, noting that $\sigma^2_{x+z,t}$ is uniformly bounded away from $0$ and infinity by Lemma \ref{lem:4.1},
we have
\begin{align*}
{\nu^N_t}\big(|\tilde y_x|>\delta\big)
\leq  
2\prod_{|z|\leq \ell} e^{-\delta^2(\varepsilon - \varepsilon^2\sigma^2_{x+z,t}/2)} \leq e^{-c\ell^d}
\end{align*}
for a constant $c>0$ depending on $\delta$.

At the same time, for $\gamma_3>0$ small, as the means $u^N(t,\cdot)$ are uniformly bounded via Lemma \ref{lem:4.1} again,
we have, in terms of $0\leq \gamma'_3\leq \gamma_3$, that
\begin{align*}
&\log E_{\nu^N_t} \big[e^{2\gamma_3 \ell_*^d |\tilde y_x| }\big]\\
&\ \ \ \leq \ \log \Big[\prod_{|z-x|\leq \ell} E_{\nu^N_t} \big[e^{2\gamma_3(\eta_z - u^N(t,z))}\big] + \prod_{|z-x|\leq \ell} E_{\nu^N_t} \big[e^{-2\gamma_3(\eta_z - u^N(t,z))}\big]\Big]\\
&\ \ \ \leq \ C\gamma_3^2\sum_{|z-x|\leq \ell} E_{\nu^N_t} \big[(\eta_z-u^N(t,z))^2e^{2\gamma'_3|\eta_z - u^N(t,z)|}\big] = O(\gamma_3^2 \ell^d).
\end{align*}

Hence, with $c_1(\delta) = c/4$, we bound the integrand in the first term in \eqref{step6_eq_1}, taking $\gamma_3=\gamma_3(\delta)>0$ small enough compared to $c$, by
\begin{align*}
\frac{CM}{\gamma_3}H(\mu_t^N|\nu^N_t) + \frac{CMN^d}{\gamma_3\ell^d}e^{-c_1(\delta)\ell^d},
\end{align*}
which when integrated in time yields the other two terms in \eqref{step6first}. 
\end{proof}

\section*{Acknowledgements}
P. El Kettani, D. Hilhorst and H.J. Park thank IRN ReaDiNet as well as the French-Korean project STAR.
 T. Funaki was supported in part by JSPS KAKENHI, Grant-in-Aid for Scientific Researches (A) 18H03672 and (S) 16H06338, and also thanks IRN ReaDiNet.
S. Sethuraman was supported by grant ARO W911NF-181-0311, a Simons Foundation Sabbatical grant, and by a JSPS Fellowship, and thanks Waseda U. for the kind hospitality during a sabbatical visit.

\end{document}